\documentclass[a4paper, 12pt]{amsart}
 \usepackage{amsmath, amsthm, amssymb, amsfonts}
\usepackage[utf8]{inputenc}
\usepackage[normalem]{ulem}
\usepackage[colorlinks = true,
            linkcolor = black,
            urlcolor  = blue,
            citecolor = red,
            anchorcolor = red,
            hidelinks]{hyperref}
\usepackage{url}
\usepackage{longtable} 
\usepackage{enumitem}
\usepackage{relsize}
\usepackage{bm}
\usepackage{xcolor}
\usepackage{graphicx}
\usepackage{mathabx}
\usepackage{mathdots}
\usepackage{mathtools}
\usepackage{color}
\usepackage{tikz-cd,tikz}
\usepackage{braket}
\usepackage{stmaryrd}
\usepackage{amssymb}
\usepackage{tipa}
\usepackage[margin=1.2 in]{geometry}
\usepackage{appendix}
\usepackage{physics}
\usepackage{mathrsfs}  
\usepackage{longtable}

\usetikzlibrary{arrows, decorations.markings}
\usetikzlibrary{positioning}
\tikzstyle{vecArrow} = [thick, decoration={markings,mark=at position
   1 with {\arrow[semithick]{open triangle 60}}},
   double distance=1.4pt, shorten >= 5.5pt,
   preaction = {decorate},
   postaction = {draw,line width=1.4pt, white,shorten >= 4.5pt}]
\tikzstyle{innerWhite} = [semithick, white,line width=1.4pt, shorten >= 4.5pt]
\tikzstyle{vecEq} = [thick,
   double distance=1.4pt,
   preaction = {decorate},
   postaction = {draw,line width=1.4pt, white,shorten >= 4.5pt}]
\usetikzlibrary{matrix}
\usetikzlibrary{shapes}
\usetikzlibrary{arrows,decorations.markings, tikzmark}
\usepackage{stmaryrd}

\pgfarrowsdeclare{bad to}{bad to}
{
  \pgfarrowsleftextend{-2\pgflinewidth}
  \pgfarrowsrightextend{\pgflinewidth}
}
{
  \pgfsetlinewidth{0.8\pgflinewidth}
  \pgfsetdash{}{0pt}
  \pgfsetroundcap
  \pgfsetroundjoin
  \pgfpathmoveto{\pgfpoint{-3\pgflinewidth}{4\pgflinewidth}}
  \pgfpathcurveto
  {\pgfpoint{-2.75\pgflinewidth}{2.5\pgflinewidth}}
  {\pgfpoint{0pt}{0.25\pgflinewidth}}
  {\pgfpoint{0.75\pgflinewidth}{0pt}}
  \pgfpathcurveto
  {\pgfpoint{0pt}{-0.25\pgflinewidth}}
  {\pgfpoint{-2.75\pgflinewidth}{-2.5\pgflinewidth}}
  {\pgfpoint{-3\pgflinewidth}{-4\pgflinewidth}}
  \pgfusepathqstroke
}

\def\v{\textsf{v}}

\theoremstyle{definition}
\newtheorem{definition}{Definition}[section]
\newtheorem{assumption}[definition]{Assumption}
\newtheorem{theorem}[definition]{Theorem}
\newtheorem{example}[definition]{Example}
\newtheorem{proposition}[definition]{Proposition}

\newtheorem{corollary}[definition]{Corollary}

\newtheorem{lemma}[definition]{Lemma}

\newtheorem{remark}[definition]{Remark}

\newtheorem{claim}{Claim}

\newtheorem{thm}{Theorem}

\newcommand{\Gr}{\mathsf{Gr}}
\newcommand{\Flag}{\mathsf{Flag}}
\newcommand{\Vir}{\textup{Vir}}

\newcommand{\BA}{{\mathbb{A}}}

\newcommand{\BC}{{\mathbb{C}}}
\newcommand{\BD}{{\mathbb{D}}}

\newcommand{\BF}{{\mathbb{F}}}
\newcommand{\BG}{{\mathbb{G}}}
\newcommand{\BH}{{\mathbb{H}}}

\newcommand{\BN}{{\mathbb{N}}}

\newcommand{\BP}{{\mathbb{P}}}
\newcommand{\BQ}{{\mathbb{Q}}}
\newcommand{\BR}{{\mathbb{R}}}

\newcommand{\BZ}{{\mathbb{Z}}}

\newcommand{\CA}{{\mathcal A}}

\newcommand{\CD}{{\mathcal D}}

\newcommand{\CH}{{\mathcal H}}

\newcommand{\CM}{{\mathcal M}}
\newcommand{\CN}{{\mathcal N}}
\newcommand{\CO}{{\mathcal O}}
\newcommand{\CP}{{\mathcal P}}
\newcommand{\CQ}{{\mathcal Q}}

\newcommand{\CT}{{\mathcal T}}

\newcommand{\CV}{{\mathcal V}}

\newcommand{\1}{\mathbf 1}

\newcommand{\bfQ}{{\bf Q}}
\newcommand{\bfA}{{\bf A}}
\newcommand{\bfT}{{\bf T}}

\newcommand{\dgQ}{{\bold{Q}}}
\newcommand{\Ac}{\textnormal{Ac}}
\newcommand{\hproj}{\textnormal{h-proj}}

\newcommand{\bfCM}{{\bm \CM}}
\newcommand{\bfCN}{{\bm \CN}}

\newcommand{\FE}{{\mathfrak{E}}}

\DeclareMathOperator{\Sym}{Sym}
\DeclareMathOperator{\id}{id}
\DeclareMathOperator{\coker}{coker}

\newcommand{\Ext}{\mathcal{E}\text{xt}}

\newcommand{\End}{{\rm End}}
\newcommand{\Stab}{{\rm Stab}}

\newcommand{\bL}{{\mathsf{L}}}

\newcommand{\bS}{{\mathsf{S}}}
\newcommand{\bT}{{\mathsf{T}}}

\newcommand{\bR}{\mathsf{R}}

\newcommand{\GL}{{\rm GL}}

\newcommand{\pt}{{\mathsf{pt}}}

\newcommand{\congpf}{\xymatrix@1@=15pt{\ar[r]^-\sim&}}

\newcommand{\ch}{\mathrm{ch}}

\newcommand{\td}{\mathrm{td}}
\newcommand{\Coh}{\mathrm{Coh}}

\newcommand{\Perf}{\mathrm{Perf}}
\newcommand{\Rep}{\mathrm{Rep}}
\newcommand{\Hom}{\mathrm{Hom}}
\newcommand{\HHom}{\mathcal{H}\mathrm{om}}

\newcommand{\vir}{\mathrm{vir}}

\newcommand{\sym}{\text{sym}}

\newcommand{\RHom}{\mathrm{RHom}}

\newcommand{\rk}{\mathrm{rk}}
\newcommand{\pl}{\mathrm{pl}}

\newcommand{\inv}{{\mathrm{wt}_0}}
\newcommand{\inva}{\mathrm{inv}}

\newcommand{\dg}{\mathsf{dg}}
\newcommand{\gr}{\mathrm{gr}}

\newcommand{\mcomment}[1]{{\color{blue}M: #1}}
\newcommand{\wcomment}[1]{{\color{cyan}W: #1}}

\usepackage{tikz}
\usepackage{tikz-cd}
\usetikzlibrary{decorations.pathmorphing}
\AtBeginDocument{%
	\def\MR#1{}
}

\renewcommand{\Ext}{\textup{Ext}}

\usepackage{MyMnSymbol}

\begin{document}

\baselineskip=16pt
\parskip=5pt

    \title{Virasoro constraints and representations for quiver moduli spaces}

\author[W. Lim]{Woonam Lim}
\address{Utrecht University, Department of Mathematics}
\email{w.lim@uu.nl}

\author[M. Moreira]{Miguel Moreira}
\address{Massachusetts Institute of Technology, Department of Mathematics}
\email{miguel@mit.edu}

\date{}

\maketitle
\begin{abstract} 
We study the Virasoro constraints for moduli spaces of representations of quiver with relations by Joyce's vertex algebras. Using the framed Virasoro constraints, we construct a representation of half of the Virasoro algebra on the cohomology of moduli stacks of quiver representations under smoothness assumption. By exploiting the non-commutative nature of the Virasoro operators, we apply our theory for quivers to del Pezzo surfaces using exceptional collections. In particular, the Virasoro constraints and representations are proven for moduli of sheaves on $\BP^2$, $\BP^1\times \BP^1$ and $\textup{Bl}_\pt(\BP^2)$. Lastly, we unravel the Virasoro constraints for Grassmannians in terms of symmetric polynomials and Hecke operators.

\end{abstract}
\setcounter{tocdepth}{1} 
\tableofcontents

\section{Introduction}\label{sec: introduction}

\subsection{Overview}

Virasoro constraints is a ubiquitous phenomenon in enumerative geometry which predicts a rich set of relations between descendent integrals. It is called the Virasoro constraints because these relations are described by a representation of half of the Virasoro algebra $\Vir_{\geq -1}$. The first instance of the Virasoro constraints was Witten-Kontsevich theorem \cite{witten, kontsevich} about $\psi$-integrals over the moduli space of stable curves. This was then extended to Gromov--Witten theory \cite{ehx} which remains one of the foremost unsolved problems in the field. 

Virasoro constraints were introduced in a sheaf theoretic context via the Gromov--Witten/Stable pairs correspondence \cite{moop}, followed by subsequent developments \cite{moreira, bree}. It was then realized in \cite{blm} that Virasoro constraints in sheaf theory have their independent origin in terms of Joyce's vertex algebra \cite{Jo17, grossjoycetanaka, Jo21}. This provided a purely representation theoretic characterization of the Virasoro constraints in terms of primary states. More importantly, this suggests that one may formulate the Virasoro constraints in other contexts if the setting of Joyce's vertex algebra applies. 

In this paper, we study the Virasoro constraints for moduli spaces of quiver representations, which appear also in \cite{bojko}. In order to apply the theory of quivers to smooth projective varieties, we show that the Virasoro constraints are preserved under derived equivalences. This reveals the non-commutative nature of the Virasoro constraints and leads to a proof of the Virasoro constraints for certain moduli spaces of Bridgeland semistable objects on $\BP^2$, $\BP^1\times \BP^1$ and $\textup{Bl}_\pt(\BP^2)$; these include, in particular, both torsion-free and 1-dimensional Gieseker semistable sheaves.

We study in this paper the new phenomenon of geometricity of the Virasoro operators. A priori, the formulation of Virasoro constraints uses a representation of $\Vir_{\geq -1}$ on some formal algebra, called the descendent algebra. However, it turns out that this representation descends to the cohomology of moduli stacks of semistable representations. We use the framed Virasoro constraints to construct this natural representation of $\Vir_{\geq -1}$ on cohomology of smooth quiver moduli stacks and of $\Vir_{\geq 0}$ on the cohomology of smooth framed quiver moduli spaces. It is remarkable that such representations exist in this generality, opening many directions of future study; applications of these ideas appear in \cite{KLMP}. We study this phenomenon in detail for Grassmannians and explain how the Virasoro constraints and representations relate to the theory of symmetric functions and Hecke operators. 

\subsection{Main results}

Let $(Q,I)$ be a finite acyclic quiver with relations. In order to formulate the Virasoro constraints, one has to choose a generating set of relations $I=(r_1,\dots, r_n)$ and consider the corresponding quasi-smooth dg quiver $\dgQ$. Such a choice is required to fix the derived enhancement of the related moduli spaces and also to define the Virasoro operators. The following theorem is one of the main results in \cite{bojko}. We explain the ingredients of the proof, following \cite{blm, bojko}, because they appear in this paper for other uses. 

\begin{thm}[\cite{bojko}]\label{main thm: quiver Virasoro}
    Let $\dgQ$ be a quasi-smooth dg quiver. If ${M}^{\theta-\textup{ss}}_d={M}^{\theta-\textup{st}}_d$, then we have 
    $$\int_{[{M}^{\theta-\textup{ss}}_d]^\vir}\xi\circ \bL_\inv(D)=0\quad \textup{for all}
    \quad D\in \BD^{\dgQ}\,.
    $$
\end{thm}

In the theorem, $\BD^\dgQ$ is a formal algebra, called the descendent algebra, that is equipped with a natural $\Vir_{\geq -1}$ representation. The operator $\bL_\inv$ is defined on $\BD^\dgQ$ in terms of this Virasoro representation, and $\xi$ is the realization morphism, which maps (a subalgebra of) $\BD^\dgQ$ to the cohomology $H^\ast({M}^{\theta-\textup{ss}}_d)$ of the moduli space.

It was observed in \cite{blm} that Virasoro constraints for moduli spaces of sheaves have a counterpart for moduli spaces of pairs, and the two are related in a fundamental way. In the case of quiver representations, framed representations play the role of pairs. By the framed/unframed correspondence, Theorem \ref{main thm: quiver Virasoro} implies the framed Virasoro constraints. For us, the moduli spaces ${M}^{\theta}_{f\rightarrow d}$ of limit $\theta$-stable framed representations are of particular importance due to the role they play in the proof of geometricity.

\begin{thm}\label{main thm: framed Virasoro}
    Let $\dgQ$ be a quasi-smooth dg quiver. Then we have 
    $$\int_{[{M}^{\theta}_{f\rightarrow d}]^\vir}\xi\circ \bL_n^{f\rightarrow d}(D)=0\quad \textup{for all}
    \quad n\geq 0,\ \ D\in \BD^{\dgQ}\,.
    $$
\end{thm}

A priori, the Virasoro constraints concerns integral of tautological classes only in degree equal to the virtual dimension. Nevertheless, the Virasoro constraints has strong implications on the structure of tautological relations in middle degrees. In order to state the result, we make Assumption~\ref{ass: smoothness} which implies smoothness of the moduli stack and framed moduli space below and surjectivity of the realization homomorphisms 
\begin{equation*}
    \BD_d^{\dgQ}\rightarrow H^*({\CM}^{\theta-\textup{ss}}_{d}),\quad \BD_d^{\dgQ}\rightarrow H^*({M}^\theta_{f\rightarrow d})\,.
\end{equation*}
We say that the Virasoro operators are geometric if their action on the formal algebra  $\BD^{\dgQ}_d$ descend to the cohomology via realization homomorphism.

\begin{thm}\label{main thm: Virasoro representation}
    Under Assumption~\ref{ass: smoothness}, the Virasoro operators in range $n\geq -1$ (resp. $n\geq 0$) are geometric for ${\CM}^{\theta-\textup{ss}}_{d}$ (resp. ${M}^\theta_{f\rightarrow d}$). In particular, there are induced representations 
    $$\Vir_{\geq -1}\lefttorightarrow H^*({\CM}^{\theta-\textup{ss}}_{d}),\quad 
    \Vir_{\geq 0}\lefttorightarrow H^*({M}^\theta_{f\rightarrow d})\,.
    $$
\end{thm}

In other words, this statement means that the Virasoro action preserves the ideals of tautological relations 
\[\ker\big(\BD_d^{\dgQ}\rightarrow H^*({\CM}^{\theta-\textup{ss}}_{d})\big)\textup{ and }\ker\big( \BD_d^{\dgQ}\rightarrow H^*({M}^\theta_{f\rightarrow d})\big)\,.\]
Understanding the tautological relations can be thought of as an analogue of the (much harder) study of the tautological ring of moduli spaces of curves, see \cite{pandharipandecalculus} for a survey. When $\bfQ=Q$ is a quiver with no relations, the ideal of tautological relations, when there are no strictly semistables, is described in \cite{franzen}.

Let $X$ be a smooth projective variety admitting a full exceptional collection $\FE=(E_1,\dots, E_n)$ such that its left dual collection is strong. Then there is an exact equivalence of triangulated categories
$$B:D^b(X)\xrightarrow{\sim} D^b(Q,I)
$$
for some quiver with relations $(Q,I)$. When $\Rep_{Q,I}$ has homological dimension at most $2$, there is a canonical choice of relations generating $I$ which in turn define a quasi-smooth dg quiver $\dgQ$. Since Virasoro constraints for quasi-smooth dg quivers are proven in Theorem \ref{main thm: quiver Virasoro}, this suggests to study the Virasoro constraints for moduli spaces of objects in $D^b(X)$ using the exact equivalence $B$. In order to achieve this, we prove the following theorem which identifies Joyce's vertex algebra with (possibly degenerate) lattice vertex algebra and explains the non-commutative nature of the Virasoro operators. 
\begin{thm}\label{main thm: natural iso}
Let $\bfT$, $\bfT_1$ and $\bfT_2$ be dg categories satisfying Assumption \ref{ass: dg category}. 
\begin{enumerate}
    \item[(i)] We have a natural isomorphism between Joyce's vertex algebra and lattice vertex algebra
    $$H_\ast(\CN^{\,\bfT})\xrightarrow{\sim} \textup{VA}(K(\bfT), \chi^\sym_{\bfT}). 
    $$
    Via this isomorphism, $H_\ast(\CN^{\,\bfT})$ is endowed with a $\Vir_{\geq -1}$-representation. 

\item[(ii)] A quasi-equivalence ${\bf B}:\bfT_1\rightarrow \bfT_2$ induces a vertex algebra isomorphism 
    $$B_*:H_\ast(\CN^{\,\bfT_1})\xrightarrow{\sim}H_\ast(\CN^{\,\bfT_2})
    $$
    intertwining the $\Vir_{\geq -1}$-representations from (i). 
\end{enumerate}
\end{thm}

We show that the Assumption \ref{ass: dg category} holds for the dg category of representations of a dg quiver (not necessarily quasi-smooth). Let $\bfT_1$ and $\bfT_2$ be dg enhancements of $D^b(X)$ and $D^b(Q,I)\simeq D^b(\bfQ)$, respectively. Applying Theorem \ref{main thm: natural iso} to the quasi-equivalence ${\bf B}:\bfT_1\rightarrow \bfT_2$ lifting the exact equivalence $B$, we can use the Virasoro constraints on $\dgQ$ to study the Virasoro constraints for moduli spaces of objects in $D^b(X)$. To be more precise, since Theorem \ref{main thm: quiver Virasoro} concerns the moduli spaces of objects in the standard heart $\Rep_{Q,I}\subset D^b(Q,I)$, we obtain the corresponding statements for $B^{-1}(\Rep_{Q,I})\subset D^b(X)$. Since this is different from the standard heart $\Coh(X)\subset D^b(X)$, it is nontrivial to apply this technique to moduli spaces of sheaves on $X$. 

This problem naturally leads us to Bridgeland stability conditions. For any smooth projective surface $S$, there are geometric stability conditions $\sigma=\sigma_{E,H}$ that depend on $\BR$-divisors $E$ and $H$ where $H$ is ample. Such stabilities define moduli stacks $\CM^{\sigma-\textup{ss}}_\v$ parametrizing $\sigma$-semistable objects in a tilted heart $\Coh_{E,H}(S)\subset D^b(S)$ of type $\v\in K(S)$. It was conjectured in \cite{am} that for any del Pezzo surface $S$ the moduli stack $\CM^{\sigma-\textup{ss}}_\v$ can be identified with a quiver moduli stack. For those moduli stacks admitting a quiver description, see Definition \ref{def: quiver description} for a precise meaning, we prove the following. 

\begin{thm}\label{main thm: abch implies Virasoro}
    Let $S$ be a del Pezzo surface and $\sigma$ be a geometric stability condition. Assume that $\CM^{\sigma-\textup{ss}}_\v$ admits a quiver description. 
    \begin{enumerate}
        \item [(i)]  If $M^{\sigma-\textup{ss}}_\v=M^{\sigma-\textup{st}}_\v$, then $M^\sigma_\v$ satisfies the Virasoro constraints with respect to its natural virtual class. 
        \item [(ii)]  If $\Ext^2_S(F,F')=0$ for all $F,\ F'$ in $\CM_\v^{\sigma-\textup{ss}}$, then $\CM_\v^{\sigma-\textup{ss}}$ is smooth, tautologically generated, and the operators $\bR_n$ are geometric, so we have an induced representation
        \[\Vir_{\geq -1}\lefttorightarrow H^*(\CM_\v^{\sigma-\textup{ss}})\,.\]
    \end{enumerate}
\end{thm}

For some del Pezzo surfaces, all moduli stacks with respect to geometric stability conditions are known to admit a quiver description, see \cite{abch} for $\BP^2$ and \cite{am} for $\BP^1\times \BP^1$ and $\textup{Bl}_\pt(\BP^2)$. Theorem \ref{main thm: abch implies Virasoro} can be applied to those surfaces to obtain the Virasoro constraints and representations. On the other hand, geometric stability conditions at the large volume limit recover Gieseker stability. Combining these results, we obtain the following theorem. 

\begin{thm}\label{main thm: three surfaces}
    Let $S$ be one of $\BP^2$, $\BP^1\times \BP^1$ or $\textup{Bl}_\pt(\BP^2)$ with nonzero-dimensional topological type $\v$.
    \begin{enumerate}
        \item[(i)] If $M^{H-\textup{ss}}_\v=M^{H-\textup{st}}_\v$, then $M^H_\v$ satisfies the Virasoro constraints with respect to its smooth fundamental class. 
        \item[(ii)] The moduli stack $\CM^{H-\textup{ss}}_\v$ is smooth, tautologically generated, and the operators $\bR_n$ are geometric, so we have a representation
        $$\Vir_{\geq -1}\lefttorightarrow H^*(\CM^{H-\textup{ss}}_\v)\,.$$
    \end{enumerate}
\end{thm}

We remark (see also \cite{bojko}) that, unlike in \cite{moreira, blm}, the proof is completely independent of the results in \cite{moop} via Gromov--Witten theory and the GW/PT correspondence. Indeed, this makes all the results in \cite{moreira, blm} completely independent of \cite{moop} since the only input previosuly needed coming from Gromov--Witten theory was the Virasoro constraints for Hilbert schemes of points on $\BP^2,\ \BP^1\times \BP^1$ (see the proof of \cite[Proposition 3.8]{moreira}), which is contained in Theorem \ref{main thm: three surfaces} (i).

In the last section, we explore in further detail the case of the Grassmannian. We explain that, once we suitably describe Schubert calculus using the ring of symmetric functions $\Lambda$, the Virasoro constraints for the Grassmannian can be interpreted as a well-known fact in the representation theory of the Virasoro algebra. Given $c, h\in \BC$, the Virasoro Lie algebra acts naturally on the ring of symmetric functions $\Lambda$ (regarded as the Fock space) so that $1\in \Lambda$ is the highest weight vector of weight $(c, h)$; $c$ is called the central charge.

\begin{thm}\label{main thm: Grass}
    The Virasoro constraints for the Grassmannian $\Gr(k, N)$ are equivalent to the rectangular Schur polynomial $s_{(N-k)^k}$ being singular vectors for representations of the Virasoro Lie algebra on $\Lambda$ with central charge $c=1$.
\end{thm}

We also find that the wall-crossing formula for the Grassmannian leads to a new proof that the rectangular Schur polynomials admit a formula in terms of what we call symmetrized Hecke operators, re-deriving (a particular case of) a result in \cite{CJ}. 

\begin{remark}
    We explain some connections of this paper to other works. The Virasoro constraints for quiver moduli spaces (Theorem \ref{main thm: quiver Virasoro}) is proven in \cite{bojko}. A version of the framed Virasoro constraints (Theorem \ref{main thm: framed Virasoro}) is proven in \cite[Equation (12)]{bojko} for $\theta=0$. We use more general stability condition to construct the Virasoro representations on cohomology groups. Part (i) of Theorem \ref{main thm: three surfaces} is proven for torsion free sheaves on any del Pezzo surfaces in \cite{blm}, allowing strictly semistable sheaves; for $\BP^2$ and $\BP^1\times \BP^1$, this was reproved in \cite{bojko}.
\end{remark}

\subsection{Organization of the paper}
This paper consists of 4 parts listed below in order; standard results about quivers, technical heart of the paper, application to del Pezzo surfaces and a case study of Grassmannian varieties. 

In Section \ref{sec: quiver}, we recollect standard definitions and results about quivers with relations and their framed analogues. In Section \ref{sec: dg quiver}, we discuss topics related to dg quivers and their modules; this includes a discussion about bounded derived categories, dg replacement of quivers with relations, and the standard resolution.

In Section \ref{sec: Virasoro}, we state the Virasoro constraints for quasi-smooth quivers and their framed analogue. We prove the Virasoro constraints for framed quivers (Theorem \ref{main thm: framed Virasoro}) and the geometricity of the Virasoro operators under a smoothness assumption (Theorem \ref{main thm: Virasoro representation}), assuming Theorem \ref{main thm: quiver Virasoro}. In Section \ref{sec: vertexalgebra}, we introduce lattice vertex algebras and Joyce's vertex algebras and construct a natural isomorphisms between them (Theorem \ref{main thm: natural iso}). 

We work in a fairly general setting of dg categories satisfying Assumption \ref{ass: dg category}, and show in Section \ref{sec: Virasoro and wall-crossing} that this assumption is satisfied for the dg category of representations of a dg quiver. In particular, we explain a construction of the descendent algebra and the Virasoro representation on it that is intrinsic to a dg category. We finish Section \ref{sec: Virasoro and wall-crossing} by proving the Virasoro constraints for quasi-smooth quivers (Theorem \ref{main thm: quiver Virasoro}) using wall-crossing formulas in Joyce's vertex algebra. 

In Section \ref{sec: applicaton}, we apply our theory to moduli spaces of Bridgeland stable objects on del Pezzo surfaces. We use the Virasoro constraints and geometricity for quivers (Theorem \ref{main thm: quiver Virasoro} and \ref{main thm: Virasoro representation}) and the non-commutative nature of the Virasoro operators (Theorem \ref{main thm: natural iso} (ii)) to prove the Virasoro constraints and representations of moduli spaces admitting a quiver description (Theorem \ref{main thm: abch implies Virasoro} and \ref{main thm: three surfaces}). 

In Section \ref{sec: Grass}, we discuss the Virasoro constraints for Grassmannians in terms of ring of symmetric functions (Theorem \ref{main thm: Grass}).

\subsection{Notations} We explain some of the notations repeatedly used in this paper. 

{\renewcommand*{\arraystretch}{1.3}
\begin{center}
\begin{longtable}{ p{4.3cm}  p{10cm} }
   $\BN$& the set of nonnegative integers\\
    
    $(Q,I)$, $\dgQ$& (finite and acyclic) quiver with relations, dg quiver\\

    $\bfT$, [$\bfT]$& saturated dg category, associated triangulated category\\

    $K(\bfT)$& (algebraic) $K$-theory of $\bfT$\\
    
    $\CM^{Q,I}$, $\CM^X$&moduli stack paramtrizing objects in the abelian categories of $(Q,I)$-representations or sheaves on $X$, respectively\\

    $\CN^\dgQ, \CN^{\,\bfT}, \CN^X$& higher moduli stacks parametrizing objects in the triangulated categories $D^b(\dgQ), [\bfT], D^b(X)$, respectively\\

    $\CM^{\theta-\textup{ss}}_d$ (resp. $M^{\theta-\textup{ss}}_d$) & moduli stack (resp. good moduli space) of $\theta$-semistable $(Q,I)$-representations of dimension vector $d$\\

    $M^{\theta}_{f\to d}$ & moduli space of limit $\theta$-stable $(Q,I)$-representations of dimension vector $d$ and framing vector $f$\\

    $\CM^{\sigma-\textup{ss}}_\v$ (resp. $M^{\sigma-\textup{ss}}_\v$)& moduli stack (resp. good moduli space) of $\sigma$-semistable objects in $D^b(S)$ with topolotical type $\v$\\

   $\bfCN^\dgQ, \bfCN^{\,\bfT}, \bfCN^X, \bfCM_d^{\theta-\textup{ss}}$& natural derived enhancements of the corresponding stacks\\

    $\BD^\bfT$ (resp. $\BD^\dgQ, \BD^X$) & Ext (resp. cohomlogical) descendent algebra\\
    
    $V^\bfT, V^{\bfQ}, V^X$& Joyce's vertex algebra associated to $\bfT$ (with underlying vector space $H_*(\CN^{\,\bfT},\BQ)$), $\bfQ$ and $X$, respectively\\
    $\textup{VA}(\Lambda,B)$& vertex algebra associated to the lattice $(\Lambda,B)$ whose underlying vector space is $V_\Lambda=\BQ[\Lambda]\otimes \BD_\Lambda$\\
    $\bL_n$, $\bR_n$ (resp. $L_n$, $R_n$)& Virasoro operators acting on cohomology or descendent algebra (resp. Joyce's vertex algebra or lattice vertex algebra)\\
\end{longtable}
\end{center}}\vspace{-20pt}

\subsection{Acknowledgement}
We thank R. Pandharipande and W. Pi for related discussions. We thank A. Mellit for a discussion that led to the construction of Virasoro representations. This work originated from the collaboration of the authors with A. Bojko. WL is supported by SNF-200020-182181 and ERC Consolidator Grant FourSurf 101087365. MM was supported during part of the project by ERC-2017-AdG-786580-MACI. The project received funding from the European Research Council (ERC) under the European Union Horizon 2020 research and innovation programme
(grant agreement 786580).

\section{Moduli of quiver representations}\label{sec: quiver}

\subsection{Fundamentals of quiver representations}\label{sec: Fundamentals of quiver}

In this section, we recall standard definitions in quiver representations and set up notations. See \cite{King, Reineke_quiver} and references therein for further details.

\subsubsection{Quiver representations}

A quiver is a tuple $Q=(Q_0,Q_1,s,t)$ where $Q_0$ and $Q_1$ are the finite sets of vertices and arrows, respectively, and $s,t:Q_1\rightarrow Q_0$ send arrows to their source and target. We assume throughout the paper that $Q$ has no oriented cycle, i.e., acyclic. 

Representation of a quiver $Q$ is a tuple $(V,\rho)$ where $V=(V_v)_{v\in Q_0}$ is a collection of finite dimensional vector spaces and $\rho=(\rho_e:V_{s(e)}\rightarrow V_{t(e)})_{e\in Q_1}$ is a collection of linear maps. We will often abuse notation and say that $V$ is a representation of $Q$, leaving the collection of linear maps implicit. Representations of $Q$ form an abelian category $\Rep_Q$. Associated to $\Rep_Q$ are the bounded derived category $D^b(Q)$ and Grothendieck's $K$-group $K(Q)\simeq \BZ^{Q_0}$. If $V$ is a representation, its dimension vector $\dim V\coloneqq(\dim V_v)_{v\in Q_0}$ is an element of $K(Q)$.

\subsubsection{Path algebra}

A path of length $k\geq 0$ in a quiver $Q$ is a sequence of the form 
$$e_k\cdots e_1\coloneqq \Big[i_0\xrightarrow{e_1} i_1 \xrightarrow{e_2}\cdots\xrightarrow{e_{k}} i_k\Big]
$$
where $t(e_j)=s(e_{j+1})$ for all $j=1,\dots, k-1$. A source and target of a path $e_k\cdots e_1$ is defined as $s(e_1)$ and $t(e_k)$, respectively. Denote the length zero path at $v\in Q_0$ by ${\bf 1}_v$. Path algebra $\BC[Q]$ is defined as a $\BC$-algebra with a basis given by paths of arbitrary lengths and the multiplication given by composition of composable paths. Then $\BC[Q]$ is an associative (typically non-commutative) $\BC$-algebra with a unit ${\bf 1}=\sum_{v\in Q_0} {\bf 1}_v$. Since $Q$ has no oriented cycle, $\BC[Q]$ is finite dimensional. 

A quiver representation $(V,\rho)$ gives rise to a finitely generated left $\BC[Q]$-module $\oplus_{v\in Q_0} V_v$ with a left multiplication by $e\in Q_1$ induced from $\rho_e$. Conversely, given a finitely generated left $\BC[Q]$-module $V$, we get a quiver representation with $V_v\coloneqq{\bf 1}_v\cdot V$ and $\rho_e:{\bf 1}_{s(e)}\cdot V\rightarrow {\bf 1}_{t(e)}\cdot V$ given by left multiplication by $e\in \BC[Q]$. This defines an equivalence of categories\footnote{By $\BC[Q]$-mod, we mean a category of fintely generated left $\BC[Q]$-modules. } 
$$\Rep_Q\simeq \BC[Q]\textup{-mod}\,.$$

\begin{example}\label{ex: simpleprojective}
Associated to each $v\in Q_0$ are a projective representation $P(v)\coloneqq\BC[Q]\cdot {\bf 1}_v$ and a simple representation $S(v)\coloneqq{\bf 1}_v\cdot\BC[Q]\cdot {\bf 1}_v$. By definition, we have $$\Hom_Q\big(P(v),V\big)=V_v\quad\text{and}\quad \Ext^n_Q\big(P(v),V\big)=0\text{ for }n>0.$$
\end{example}

\subsubsection{Standard resolution and Euler form} \label{sec: standard resolution without relations}

Define the Euler characteristics between two $Q$-representations $V$ and $V'$ as 
$$\chi_Q\big(V,V'\big)\coloneqq\sum_{n\geq 0}(-1)^n \dim\Ext_Q^n(V,V'). 
$$
By additivity with respect to short exact sequences, the Euler form is defined in $K(Q)\simeq \BZ^{Q_0}$.

For any $Q$-representation $V$, we have a standard projective resolution 
\begin{equation}\label{eq: standard resolution}
    0\rightarrow \bigoplus_{e\in Q_1} P(t(e))\otimes V_{s(e)}\rightarrow \bigoplus_{v\in Q_0} P(v)\otimes V_v\rightarrow V\rightarrow 0.
\end{equation}
In particular, the abelian category $\Rep_Q$ has homological dimension at most 1. Applying $\Hom_Q(-,V')$ to \eqref{eq: standard resolution}, we obtain a complex of vector spaces representing  $\RHom_Q(V,V')$:
$$0\rightarrow \bigoplus_{v\in Q_0}\Hom_\BC(V_v,V'_v)\rightarrow \bigoplus_{e\in Q_1}\Hom_\BC(V_{s(e)},V'_{t(e)})\rightarrow 0
$$
whose cohomology computes the Ext groups $\Ext^n_Q(V,V')$. Furthermore, we obtain an explicit formula for the Euler form 
$$\chi_Q(-,-)\colon \BZ^{Q_0}\times \BZ^{Q_0}\rightarrow \BZ,\quad (d,d')\mapsto\sum_{v\in Q_0}d_v\cdot d'_v-\sum_{e\in Q_1}d_{s(e)}\cdot d_{t(e)}'. 
$$
Since $Q$ has no oriented cycle, there is an ordering of $Q_0$ such that the Euler form with respect to the induced basis is upper-triangular with $1$ along the diagonal. In particular, the Euler form $\chi_Q(-,-)$ is a perfect pairing on $\BZ^{Q_0}$. 

Given a perfect integral pairing $\chi:\Lambda\times \Lambda\rightarrow \BZ$ on a finitely generated free abelian group $\Lambda$, we define its diagonal as 
$$\Delta\coloneqq\sum_{i\in I}v_i\otimes \hat{v}_i\in \Lambda\otimes_\BZ \Lambda
$$
where $\{v_i\}_{i\in I}$ is an integral basis and $\{\hat{v}_i\}_{i\in I}$ is the dual basis satisfying $\chi(\hat{v}_i, v_j)=\delta_{ij}$. The diagonal is characterized by the property that the composition 
$$\Lambda\xrightarrow{\Delta\otimes \id}\Lambda\otimes_\BZ\Lambda\otimes_\BZ\Lambda\xrightarrow{\id\otimes \chi}\Lambda
$$
is an identity, or equivalently by $(\chi\otimes \id)\circ (\id\otimes \Delta)=\id$. 

\begin{example}\label{ex: K theoretic diagonal}
By Example \ref{ex: simpleprojective}, we have 
$$\chi_Q(P(v),S(w))=\sum_{n\geq 0}(-1)^n \dim\Ext^n_Q(P(v),S(w))= \delta_{vw}\,.
$$
This implies that the diagonal with respect to the perfect pairing $\chi_Q$ is
$$\sum_{v\in Q_0}[S(v)]\otimes [P(v)]\in \BZ^{Q_0}\otimes \BZ^{Q_0}
$$
where $[-]$ denotes the corresponding $K$-theory class.
\end{example}

\subsubsection{Representations of quivers with relations}

We say that a representation $V$ of $Q$ satisfies a relation $r=\sum_s \lambda_s\, e^{(s)}_{k_s}\cdots e^{(s)}_{1}$
with source $v$ and target $w$ if we have
$$0=\sum_s \lambda_s\cdot \rho(e^{(s)}_{k_s})\circ \cdots\circ \rho(e^{(s)}_1)\in \Hom_\BC(V_v, V_w). 
$$
A representation of $(Q,I)$ is a representation of $Q$ that satisfies all the relations in $I$. If we choose generators of the ideal $I=(r_1,\dots,r_n)$, then it suffices to check that $V$ satisfies the relations $r_1,\dots, r_n$. A morphism between two representations $V$ and $V'$ of $(Q,I)$ is simply a morphism between representations of the quiver $Q$, so representations of $(Q,I)$ form an abelian category $\Rep_{Q,I}$, which is a full abelian subcategory of $\Rep_Q$. We have an equivalence of categories\footnote{Again, $\BC[Q]/I\textup{-mod}$ denotes the category of finitely generated, hence finite dimensional, left $\BC[Q]/I$-modules. }
$$\Rep_{Q,I}\simeq \BC[Q]/I\textup{-mod}.
$$
We also note that there is a corresponding bounded derived category $D^b(Q,I)$ and $K$-group $K(Q,I)\simeq \BZ^{Q_0}$.  

\begin{example}\label{ex: beilinson}
We illustrate the definitions with an example that will be useful in Section \ref{sec: beilinson}. Consider a quiver $Q$ as in the following picture
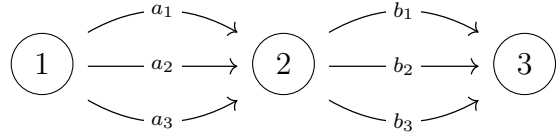
\begin{center}\begin{tikzcd}
\tikz \node[draw, circle]{1}; \arrow[rr, "a_3" description, bend right, shift right] \arrow[rr, "a_2" description] \arrow[rr, "a_1" description, bend left, shift left] &  & \tikz \node[draw, circle]{2}; \arrow[rr, "b_2" description] \arrow[rr, "b_3" description, bend right, shift right] \arrow[rr, "b_1" description, bend left, shift left] &  & \tikz \node[draw, circle]{3};
\end{tikzcd}\end{center}
with an ideal generated by six relations
$$I\coloneqq(\,b_2a_1+b_1a_2\,,\, b_3a_2+b_2a_3\,,\, b_1a_3+b_3a_1\,,\,a_1b_1\,,\, a_2b_2\,,\, a_3b_3\,).
$$
We call $(Q,I)$ the Beilinson quiver for $\BP^2$. The derived category $D^b(Q,I)$ is equivalent to $D^b(\BP^2)$ via Beilinson isomorphism \cite{beil}. 
\end{example}

\subsubsection{Euler form of quivers with relations}

Given two representations $V$ and $V'$ of $(Q,I)$, we define the Euler form 
$$\chi_{Q,I}(V,V')\coloneqq\sum_{n\geq 0}(-1)^n \dim\Ext^n_{Q,I}(V,V')
$$
as in the case of a quiver without relations. This extends to $K(Q,I)\simeq \BZ^{Q_0}$.

Questions regarding homological dimension and the Euler form of $\Rep_{Q,I}$ are more delicate compared to the case of $\Rep_Q$. While the homological dimension of $\Rep_{Q,I}$ is bounded by the maximal length of paths \cite[page 98]{HZ}, it is difficult to find a sharper or exact bound with respect to $I$. Also, an explicit formula for the Euler form is not available in general. We come back to these questions in Section \ref{sec: dg quiver} using the language of differential graded quivers. 

\subsection{Stability conditions, moduli spaces and moduli stacks}\label{sec: quiver, moduli}

Let $(Q, I)$ be a quiver with relations and $d\in \BN^{Q_0}$. We have a moduli stack $\mathcal{A}_d$ of representations of $Q$ with fixed dimension vector $d$. We explain here its construction. Let
\[A_d=\bigoplus_{e\in Q_1}\Hom_\BC(\BC^{d_{s(e)}},\BC^{d_{t(e)}})\,.\]
The product of general linear groups
\[\GL_d=\prod_{v\in Q_0}\GL_{d_v}\]
acts on $A_d$ via conjugation. The global quotient stack 
\[\CA_d\coloneqq [A_d/\GL_d]\]
parametrizes representations of the quiver $Q$ without imposing relations. Note that the dimension of the stack $\CA_d$ is given by
\begin{align*}
\dim A_d - \dim \GL_d
&=\sum_{e\in Q_1} d_{s(e)}\cdot d_{t(e)} - \sum_{v\in Q_0} d_v^2
= -\chi_Q(d,d)\,. 
\end{align*}
By definition of $A_d$, there exists $(\tilde{\CV},\tilde{\rho})$ where $\tilde{\CV}$ is a collection of trivial vector bundles $\tilde{\CV}_v=\BC^{d_v}\otimes \CO_{A_d}$ on $A_d$ and $\tilde{\rho}$ is a collection of morphisms $\tilde{\rho}_e:\tilde{\CV}_{s(e)}\rightarrow \tilde{\CV}_{t(e)}$ that restricts tautologically to each point $[\rho]\in A_d$. Furthermore, $\GL_d$-action on $A_d$ lifts to a natural $\GL_d$-equivariant structure on $\tilde{\CV}_i$ such that  the morphisms $\tilde{\rho}_e$ are invariant. Therefore, the data $(\tilde{\CV},\tilde{\rho})$ descends to a universal representation $(\CV,\rho)$ on $\CA_d$ which is unique up to a unique isomorphism. 

Given a family of representations, such as $A_d$ or $\CA_d$, relations impose Zariski closed condition. To be more explicit, we choose generating relations of the ideal $I=(r_1,\dots, r_n)$. We have a vector bundle over $A_d$
$$\tilde{E}_d\coloneqq \bigoplus_{i=1}^n \CH om(\tilde{\CV}_{s(r_i)},\tilde{\CV}_{t(r_i)})
$$
and a section 
$$\tilde{s}_d=\bigoplus_{i=1}^n\tilde{s}_d^{(r_i)}\in \Gamma(A_d, \tilde{E}_d)
$$
canonically associated to the choice of generators of $I$ defined as follows: if 
\[r=\sum_s \lambda_s\, e^{(s)}_{k_s}\cdots e^{(s)}_{1}\in I\subseteq \BC[Q]\] then 
$$\tilde{s}_d^{(r)}=\sum_s \lambda_s\cdot \tilde{\rho}_{e^{(s)}_{k_s}}\circ \cdots\circ \tilde{\rho}_{e^{(s)}_1}\in \Hom(\tilde{\CV}_{s(r)}, \tilde{\CV}_{t(r)})\,.
$$
We let $R_d\subseteq A_d$ be the zero locus
\begin{center}
\begin{tikzcd}
                  & \tilde{E}_d \arrow[d]                   \\
R_d=\textnormal{Zero}(\tilde{s}_d) \arrow[r,"\tilde{i}_d", hook] & A_d\arrow[u, "\tilde{s}_d"', bend right]
\end{tikzcd}
\end{center}
The moduli stack 
\[\CM^{Q,I}_d=\CM_d\coloneqq [R_d/\GL_d]\hookrightarrow [A_d/\GL_d]=\CA_d\]
parametrizes representations of $(Q, I)$. Since the vector bundle $\tilde E_d$ and the section $\tilde s_d$ are respected by the $\GL_d$-action, they descend to the moduli stack $\CA_d$; Therefore we can describe the closed substack $\CM_d\hookrightarrow \CA_d$ as the zero locus of a section of a vector bundle 
\begin{center}
\begin{tikzcd}
                  & E_d \arrow[d]                   \\
\CM_d=\textnormal{Zero}(s_d) \arrow[r,"i_d", hook] & \CA_d \arrow[u, "s_d"', bend right]
\end{tikzcd}
\end{center}
The universal representation $(\CV, \rho)$ on $\CA_d$ restricts to a universal representation on $\CM_d$, for which we use the same notation.

\begin{remark}
  Note that, although the construction of $R_d$ and $\CM_d$ uses a choice of generators (and the corresponding vector bundle and section), $\CM_d$ and $R_d$ do not actually depend on this choice, since satisfying all the relations on $I$ is the same as satisfying the generating relations $r_1, \ldots, r_n$. We will see later that a choice of the generators defines a quasi-smooth dg quiver which then induces a choice of derived enhancement of the above moduli stack.
\end{remark}

Note that $\mathbb{G}_m\subseteq \GL_d$, included via the diagonal, acts trivially on $R_d$. Therefore the action of $\GL_d$ on $R_d$ factors through the projective linear group $\textnormal{PGL}_d\coloneqq \GL_d/\mathbb{G}_m$.\footnote{Note that $\textnormal{PGL}_d$ is in general different from the product of $\textnormal{PGL}_{d_i}$ because we only remove one copy of $\BG_m$.} This defines a projective linear moduli stack $\CM^{\textnormal{pl}}_d\coloneqq [R_d/\textnormal{PGL}_d]$. The natural map 
$$\Pi^{\textnormal{pl}}_d\colon \CM_d\rightarrow \CM^{\textnormal{pl}}_d
$$
is a $\BG_m$-gerbe as long as $d\neq 0$.

\subsubsection{Stability conditions and moduli spaces}
\label{subsubsec: stabilityquivers}

A stability condition on $Q$ is a linear functional $\theta\in \Hom(\BZ^{Q_0},\BQ)$. Note that a functional $\theta$ can be uniquely written as
\[\theta(d)=\sum_{v\in Q_0}\theta_vd_v\]
for some weights $\theta_v\in \BQ$. Given a stability condition $\theta$, we define a slope function 
$$\mu^\theta:\BN^{Q_0}\backslash\{0\}\rightarrow \BQ,\quad d=(d_v)\mapsto\frac{\theta(d)}{\sum_v d_v}.
$$
The slope of a nonzero quiver representation $V$ with respect to $\theta$ is defined to be $\mu^\theta(V)=\mu^\theta(\dim V)$. We say that $V$ is $\theta$-(semi)stable if $\mu^\theta(V')(\leq) \mu^\theta(V)$ for all nonzero proper subrepresentations $0\subsetneq V'\subsetneq V$.

A stability condition $\theta$ defines $\GL_d$-equivariant Zariski open subsets
$$R^{\theta-\textup{st}}_d\subseteq R^{\theta-\textup{ss}}_d\subseteq R_d$$
that correspond to the locus of $\theta$-(semi)stable representations of $(Q, I)$, which are possibly emptysets. By \cite{King}, $\theta$-(semi)stability can be interpreted as GIT-(semi)stability with respect to certain character of $\GL_d$ constructed from $\theta$. Since $Q$ is acyclic, this implies that there is a projective good moduli scheme
$$\CM^{\theta-\textup{ss}}_d=[R^{\theta-\textup{ss}}_d/\GL_d]\rightarrow R^{\theta-\textup{ss}}_d\sslash\GL_d\eqqcolon M^{\theta-\textup{ss}}_d. 
$$

If $V$ is $\theta$-stable, then $\Hom_{Q, I}\big(V, V\big)\simeq \BC$. Therefore, the $\textnormal{GL}_d$-action on $R^{\theta-\textup{st}}_d$ has only constant stabilizer group $\BG_m\subseteq \GL_d$ and hence the good moduli map is a $\BG_m$-gerbe
$$\CM_d^{\theta-\textup{st}}=[R^{\theta-\textup{st}}_d/\GL_d]\rightarrow [R^{\theta-\textup{st}}_d/\textnormal{PGL}_d]\eqqcolon M^{\theta-\textup{st}}_d.
$$
We call $M^{\theta-\textup{ss}}_d$ (resp. $M^{\theta-\textup{st}}_d$) the moduli space of $\theta$-semistable (resp. $\theta$-stable) representations of $(Q, I)$. The stable one defines a Zariski open subset $M^{\theta-\textup{st}}_d\subseteq M^{\theta-\textup{ss}}_d$. When $M^{\theta-\textup{st}}_d=M^{\theta-\textup{ss}}_d$ we will abreviate and write only $M^{\theta}_d$.

\subsection{Framed quiver}

We review here the notions of framed representations and corresponding moduli spaces, mostly following the presentation in \cite{Reineke_framed}. Framed representations play an important role in the wall-crossing theory of Joyce \cite{grossjoycetanaka, Jo21} as they are the analogues of Joyce--Song pairs in the quiver setting. 

Moduli of framed representations can be understood in terms of unframed representations of a larger quiver. Nevertheless, we treat the framed analog because Virasoro constraints will have simpler formulation without going through the framed/unframed correspondence. They will also play an important role in the study of Virasoro representations on the cohomology of the stacks $\CM^{\theta-\textup{ss}}_d$, cf. Theorem \ref{main thm: Virasoro representation}.

\subsubsection{Framed representations of quiver}

Fix a framing vector $f\in \BN^{Q_0}\backslash\{0\}$. An $f$-framed representation of $(Q, I)$ is a tuple $((V,\rho),\phi)$ where $(V,\rho)$ is a representation of $(Q, I)$ and $\phi=(\phi_v:\BC^{f_v}\rightarrow V_v)_{v\in Q_0}$ is a collection of linear maps. We will abreviate the notation to $(V, \phi)$ and leave $\rho$ implicit. The $f$-framed representations form a category $\Rep^f_Q$, where the morphisms are morphisms of quiver representations that respect the framings. We remark that $\Rep^f_Q$ is no longer an abelian category.

\subsubsection{Moduli stack of framed representations}

For a framing vector $f\in \BN^{Q_0}\backslash\{0\}$ and a usual dimension vector $d\in \BN^{Q_0}$, we define a moduli stack $\CM_{f\shortrightarrow d}$ of $f$-framed representations of $Q$ with a dimenison vector $d$. Precisely, the moduli stack is given by the global quotient stack 
$$\CM_{f\shortrightarrow d}=[R_{f\shortrightarrow d}/\GL_d]
$$
where $R_{f\shortrightarrow d}=R_d\oplus \big(\oplus_{v\in Q_0}\Hom_\BC(\BC^{f_v},\BC^{d_v})\big)$ is the space of all linear maps involved and $\GL_d=\prod_{v\in Q_0}\GL_{d_v}$ is a product of general linear groups acting on $R_{f\shortrightarrow d}$ by
$$(g_v)\cdot \big( (\rho_e),(\phi_v)\big)\coloneqq\Big(\big(g_{t(e)}\circ \rho_e\circ g_{s(e)}^{-1}\big),\big(g_v\circ \phi_v\big)\Big). 
$$
The stack $\CM_{f\to d}$ maps to $\CM_d$ by forgetting the framing. In factr, $\CM_{f\to d}$ it is the total space of a vector bundle of rank $f\cdot d=\sum_{v\in Q_0}f_v\cdot d_v$ over $\CM_d$. 

The moduli stack $\CM_{f\shortrightarrow d}$ is equipped with a universal framed representation $\big((\CV,\rho),\phi\big)$ where $(\CV,\rho)$ is the universal representation pulled back from $\CM_d$, and $\phi=(\phi_v:\BC^{f_v}\otimes \CO_{\CM_{f\shortrightarrow d}}\rightarrow \CV_v)_{v\in Q_0}$ is a collection of universal framings. We remark that $\mathbb{G}_m\subseteq \GL_d$ no longer acts trivially on $R_{f\shortrightarrow d}$ if $f\neq 0$ and $d\neq 0$.

\subsubsection{Limit stability condition and moduli spaces}

Let $\theta$ be a stability condition of a quiver $Q$. We say that a framed representation $\big((V,\rho),\phi\big)$ is limit $\theta$-stable (which is the same as being limit $\theta$-semistable) if
\begin{enumerate}
\item  $\phi$ is nonzero
    \item $\mu^\theta(V')\leq \mu^\theta(V)$ for all $0\subsetneq V'\subsetneq V$, and
    \item $\mu^\theta(V')< \mu^\theta(V)$ for all $0\subsetneq V'\subsetneq V$ for which the framing $\phi$ factors through, i.e. $\phi_v(\BC^{f_v})\subseteq V'_v$. 
\end{enumerate}
Note that the second condition is equivalent to the $\theta$-semistability of $V$, so a framed representation being limit $\theta$-semistable implies that the underlying representation $V$ is $\theta$-semistable. On the other hand, if $V$ is $\theta$-stable then $(V, \phi)$ is automatically limit $\theta$-stable for any nonzero framing $\phi$. 

The limit $\theta$-stability condition defines a $\GL_d$-invariant Zariski open subset 
$$R^{\theta}_{f\shortrightarrow d}\subseteq R_{f\shortrightarrow d}.$$
The $\GL_d$-action on $R^\theta_{f\shortrightarrow d}$ is free (without even $\BG_m$-stabilizers) and the quotient stack 
$$M^\theta_{f\shortrightarrow d}=[R^\theta_{f\shortrightarrow d}/\GL_d]
$$
is a projective scheme as we will see in the next subsection. We call this the moduli space of limit $\theta$-stable $f$-framed representations of $(Q,I)$ with a dimension vector $d$. This moduli space is equipped with a universal framed representation restricted from the entire stack $\CM_{f\shortrightarrow d}$. Since the underlying representation $(V,\rho)$ of a limit $\theta$-stable representation is $\theta$-semistable, there is a forgetful morphism 
$$\pi:M^{\theta}_{f\shortrightarrow d}\rightarrow M^{\theta-\textup{ss}}_d. 
$$
When there are no strictly $\theta$-semistable representations in $M^{\theta-\textup{ss}}_d$, the forgetful morphism $\pi$ is a projective bundle of dimension $f\cdot d-1$. 

\subsubsection{Framed/unframed correspondence}\label{sec: f/uf corr}

Fix a quiver with relations $(Q, I)$ and a framing vector $f\in \BN^{Q_0}\backslash\{0\}$. We explain how $f$-framed representations of $(Q,I)$ can be understood as usual representations of a bigger quiver with relations $(Q^f, I^f)$ constructed from $(Q,I)$ and $f$. 

Define a quiver $Q^f$ by adding a new vertex $\infty$ and $f_v$ number of arrows $\infty\rightarrow v$ for each $v\in Q_0$. Precisely, we have $Q^f=(Q^f_0,Q^f_1,\tilde{s},\tilde{t}\,)$ where
\begin{enumerate}
    \item $Q^f_0=\{\infty\}\sqcup Q_0$,
    \item $Q^f_1=\{e_{v,s}\,|\,v\in Q_0,\ 1\leq s\leq f_v\}\sqcup Q_1$,
    \item $\tilde{s}:Q^f_1\rightarrow Q^f_0$ with $Q_1\ni e\mapsto s(e)$ and $e_{v,s}\mapsto \infty$,
    \item $\tilde{t}:Q^f_1\rightarrow Q^f_0$ with $Q_1\ni e \mapsto t(e)$ and  $e_{v,s}\mapsto v$. 
\end{enumerate}
Note that since $Q$ is acyclic, so is $Q^f$. The path algebra $\BC[Q]$ embeds naturally in $\BC[Q^f]$ and we define the ideal $I^f$ by extending scalars, i.e.
\[I^f=I\otimes_{\BC[Q]}\BC[Q^f]\,.\]
A representation of $\BC[Q^f]$ satisfies the relations in $I^f$ if and only if the induced representation of $\BC[Q]$ satisfies all the relations in $I$.

It is clear from the construction of $(Q^f, I^f)$ that $f$-framed representations of $(Q,I)$ of dimension vector $d\in \BN^{Q_0}$ are the same as usual representations of $(Q^f,I^f)$ of dimension vector $(1,d)\in \BN\times \BN^{Q_0}=\BN^{Q^f_0}$. In other words, we have
\begin{equation}\label{eq: framed/unframed}
    R_{f\shortrightarrow d}(Q, I)=R_{(1,d)}(Q^f, I^f). 
\end{equation}
However, the notion of isomorphism differs in this correspondence because in the framed case we only allow identity maps between the framing vector spaces $\BC^{f_v}$, but in the unframed case we allow a $\BG_m$-action on the one dimensional vector space $\BC^1$ at the infinity vertex $\infty$. 

Precisely, we have 
$$\CM_{(1,d)}(Q^f, I^f)=[R_{(1,d)}(Q^f, I^f)/\GL_{(1,d)}]=[R_{f\shortrightarrow d}/\BG_m\times \GL_d]
$$
whereas the framed moduli stack is
$$\CM_{f\shortrightarrow d}(Q, I)=[R_{f\shortrightarrow d}/\GL_d]. 
$$
Therefore, $\CM_{f\shortrightarrow d}(Q, I)$ is the projective linearization
$$\Big(\CM_{(1,d)}(Q^f, I^f)\Big)^{\textnormal{pl}}\simeq \CM_{f\shortrightarrow d}(Q, I)\,. 
$$
Through this isomorphism, the universal framed representation on $\CM_{f\shortrightarrow d}(Q, I)$ corresponds to a family of $Q^f$-representations on $\Big(\CM_{(1,d)}(Q^f, I^f)\Big)^{\textnormal{pl}}$ with $\CV_\infty=\CO$. By the universal property of the stack $\CM_{(1,d)}(Q^f, I^f)$, this family defines a section of the projective linearization map
$$\Pi^{\textnormal{pl}}_{(1,d)}:\CM_{(1,d)}(Q^f, I^f)\rightarrow \Big(\CM_{(1,d)}(Q^f, I^f)\Big)^{\textnormal{pl}}.
$$
This canonical section of the projective linearization map of a quiver $Q^f$ is related to the $\infty$-normalization used in the proof of Proposition \ref{prop: f/uf Virasoro}.

Given a stability condition $\theta$ of $Q$, we define a stability condition $\tilde{\theta}\in \Hom(\BZ^{\{\infty\}}\times \BZ^{Q_0},\BQ)$ of $Q^f$ as follows: $\tilde{\theta}$ maps $\BZ^{Q_0}$ using $\theta$ and $1\in\BZ^{\{\infty\}}$ to $\mu^{\theta}(d)+\epsilon$. Here $\epsilon\in \BQ_{>0}$ is small enough  according to $(Q,\theta,d,f)$. Then $\tilde{\theta}$-semistability is the same as $\tilde{\theta}$-stability and recovers the limit $\theta$-stability via the framed/unframed correspondence \eqref{eq: framed/unframed}. Therefore, we have 
$$M^{\tilde{\theta}}_{(1,d)}(Q^f,I^f)\simeq 
\Big(\CM_{(1,d)}^{\tilde{\theta}}(Q^f, I^f)\Big)^{\textnormal{pl}}
\simeq M_{f\shortrightarrow d}^{\theta}(Q, I).
$$
In particular, this shows that $M^{\theta}_{f\shortrightarrow d}(Q, I)$ is a projective scheme. When there are no relations, it also follows that $M^{\theta}_{f\shortrightarrow d}(Q)$ is smooth.

\begin{example}\label{ex: Grassmannian}
    Let $A_1$ be the quiver with only 1 vertex and no arrows, $I=0$ and let $\theta$ be arbitrary. It is straightforward to check that a framed representation $\phi\colon \BC^N\to V$ is limit $\theta$-stable if and only if $\phi$ is surjective, hence
    \[M^\theta_{N\to k}(A_1)=\Gr(N, k)\,.\]
    Under the framed/unframed correspondence, the Grassmannian is also identified with the moduli of representations
    \[M_{(1, k)}^{\tilde \theta}(K_N)\,,\]
    where $K_N$ is the Kronecker quiver with two vertices $\{\infty, 1\}$ and $N$ arrows from $\infty$ to $1$ and $\tilde \theta$ is such that $\tilde \theta_\infty>\tilde \theta_1$.
    
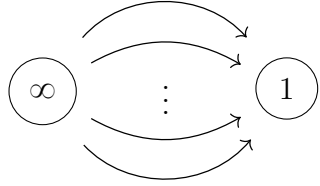
\begin{center}
\begin{tikzcd}
\tikz \node[draw, circle]{$\infty$}; \arrow[rr, bend left=49, shift left] \arrow[rr, bend left] \arrow[rr, "\vdots", phantom] \arrow[rr, bend right] \arrow[rr, bend right=49, shift right] &  & \tikz \node[draw, circle]{1};
\end{tikzcd}
    \end{center}
\end{example}

\begin{example}
    Let $A_\ell$ be the linear quiver with $\ell$ vertices $\{1, 2, \ldots, \ell\}$ and arrows $i\to i+1$. The flag variety (of quotients) $\mathsf{Flag}(N; k_1, \ldots, k_\ell)$ is identified with the moduli space of framed quiver representations
    \[\mathsf{Flag}(N; k_1, \ldots, k_\ell)=M_{(N, 0, \ldots, 0)\to (k_1, \ldots, k_\ell)}^\theta(A_\ell)\]
    where $\theta$ is such that $\theta_i\gg\theta_{i+1}$. Under the framed/unframed correspondence the flag variety is identified with representations of the following quiver:
    
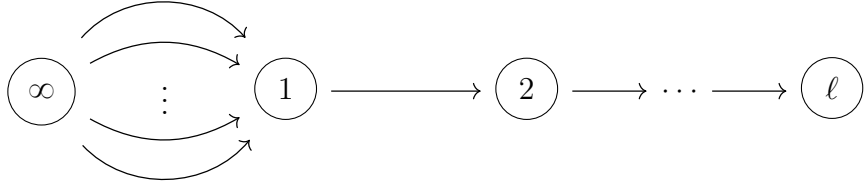
\begin{center}
\begin{tikzcd}
\tikz \node[draw, circle]{$\infty$};\arrow[rr, bend left=49, shift left] \arrow[rr, bend left] \arrow[rr, "\vdots", phantom] \arrow[rr, bend right] \arrow[rr, bend right=49, shift right] &  & \tikz \node[draw, circle]{1}; \arrow[rr] &  & \tikz \node[draw, circle]{2}; \arrow[r] & \cdots \arrow[r] & \tikz \node[draw, circle]{$\ell$};
\end{tikzcd}
    \end{center}
\end{example}

\begin{example}
    When $I=0$ and $\theta=0$, the moduli spaces $M_{f\to d}^{\theta=0}(Q)$ were studied in \cite{Reineke_framed}. Theorem 4.10 in loc. cit. gives a description of such moduli spaces as an iterated Grassmann bundle.
\end{example}

\subsubsection{Approximation by framed moduli spaces}

Framed moduli spaces can be used to approximate the cohomology of the moduli stack of representations. This will be used later to construct the Virasoro representations on the cohomology of moduli stacks under some assumptions.
\begin{proposition}\label{prop: approximation}
    Let $(Q,I)$ be a quiver with relations and $\CM^{\theta-\textup{ss}}_d$ be a smooth moduli stack. If a framing vector $f$ satisfies $\min_{v\in Q_0}(f_v-d_v)>k$, then pull back under the forgetful morphism defines an isomorphism in the specified degrees
    $$H^{\leq 2k}(\CM_d^{\theta-\textup{ss}})\xrightarrow{\sim}H^{\leq 2k} (M^\theta_{f\rightarrow d}). 
    $$
\end{proposition}
\begin{proof}
    This type of statement is well-known in various context, see for instance \cite[Lemma 4.1]{Davison_Meinhardt}. We write the proof for completeness. By definition of limit $\theta$-stable $f$-framed representations, we have a diagram
\begin{equation}\label{eq: framed approximation}
    \begin{tikzcd}
\HHom_{f\rightarrow d}^\theta\coloneqq\bigoplus\limits_{v\in Q_0}\HHom(\CO^{f_v}, \CV_v) \arrow[d, "\pi"] & {M}^{\theta}_{f\rightarrow d} \arrow[l, "j"', hook'] \\
{\CM}^{\theta-\textup{ss}}_d                &                         
\end{tikzcd}
\end{equation}
where $\HHom_{f\rightarrow d}^\theta$ is the moduli stack parametrizing $\theta$-semistable representations $V$ together with any framings $\phi=(\phi_v:\BC^{f_v}\rightarrow V_v)$ and $j$ corresponds to the open locus of limit $\theta$-stable framings. Since $\HHom_{f\rightarrow d}^\theta$ is a total space of a vector bundle over ${\CM}^{\theta-\textup{ss}}_d$, pull back under $\pi$ defines an isomorphism of cohomology rings. Since every space in the diagram \eqref{eq: framed approximation} is smooth by assumption, it suffices to show that $\textup{codim}(j)\geq \min_{v\in Q_0}(f_v-d_v)$ which would imply
\begin{equation*}
    H^{\leq 2k}({\CM}^{\theta-\textup{ss}}_d )\xrightarrow[\sim]{\pi^*} 
H^{\leq 2k}(\HHom^\theta_{f\rightarrow d} )\xrightarrow[\sim]{j^*} 
H^{\leq 2k}({M}^{\theta}_{f\rightarrow d})\,.
\end{equation*}

Over a fixed representation $V$ in ${\CM}^{\theta-\textup{ss}}_{d}$, the complement of $j$ is 
$$\Big(\bigoplus_{v\in Q_0}\Hom_\BC(\BC^{f_v},V_v)\Big)^{\textup{unstable}}\hookrightarrow\bigoplus_{v\in Q_0}\Hom_\BC(\BC^{f_v},V_v)\,.
$$
By definition of limit $\theta$-stability, an unstable framing $\phi=(\phi_v:\BC^{f_v}\rightarrow V_v)$ factors through some subrepresentation $V'\subsetneq V$ of dimension vector $d'\lneq d$. Therefore, the dimension of the space of unstable framing is bounded above by 
$$\sum_{v\in Q_0}(d_v-d_v')d_v'+f_v\cdot d_v'
$$
where the first term corresponds to a choice of arbitrary subspaces $V'\subsetneq V$ which are not necessarily a subrepresentation. In other words, the codimension of $j$ is bounded below by 
\begin{align*}
    \textup{codim}(j)&\geq \min_{d'\lneq d}\sum_{v\in Q_0}f_v d_v - \sum_{v\in Q_0}\big((d_v-d_v')d_v'+f_v\cdot d_v'\big)\\
    &=\min_{d'\lneq d} \sum_{v\in Q_0}(d_v-d_v')(f_v-d_v')\\
    &\geq \min_{v\in Q_0}(f_v-d_v)
\end{align*}
which completes the proof.
\end{proof}

\section{Differential graded quivers}\label{sec: dg quiver}

Many of the constructions in this paper, except Section \ref{sec: quiver}, depend not only on $(Q,I)$ but also on the choice of generators $I=(r_1,\dots, r_n)$. For example, choosing generators is necessary to endow the moduli spaces of stable representations $M^{\theta-\textup{st}}_d$ with a virtual fundamental class. Quasi-smooth differential graded (=dg) quivers provide a natural language to keep track of such a choice of generators. 

One of the advantages of working in the generality of dg quivers is that it provides a more natural approach to vertex algebras for quivers with relations. The vertex algebras constructed in \cite{Jo21} for quivers with relations require a choice of generators of $I$, and in particular they are not intrinsic to $D^b(Q, I)$. This raises an important question in applications: if $X$ is a variety with $D^b(X)\simeq D^b(Q, I)$, do we have an isomorphism between the vertex algebra associated to $X$ and a vertex algebra associated to $(Q, I)$ together with a choice of generators of $I$? The answer is typically no if $\dim X>2$. Instead, the vertex algebra associated to $X$ is isomorphic to the vertex algebra canonically associated to a certain dg quiver $\bfQ$. When $\bfQ$ is quasi-smooth, this recovers the construction in \cite{Jo21} with a choice of generators. 

In this section, we review definitions and known results on dg quivers and their modules which we need for applications to Virasoro constraints. 

\subsection{Fundamentals of dg quivers}

\subsubsection{Dg algebra and derived category}\label{sec: dg algebra}
Before explaining dg quivers, we briefly review dg algebras (over the complex numbers) and their derived categories, see \cite{Toen, Keller} and \cite[\href{https://stacks.math.columbia.edu/tag/09JD}{Tag 09JD}]{stacks-project} for details.

A dg algebra is a  $\BZ$-graded  associative algebra ${\bf A}=\bigoplus\limits_{n\in \BZ}{\bf A}^{(n)}$ together with a unit $1\in {\bf A}^{(0)}$ and a differential $\partial:{\bf A}\rightarrow {\bf A}$ satisfying the graded Leibniz rule
$$\partial(ab)=\partial(a)\cdot b+(-1)^{\deg(a)}\,a\cdot\partial(b). 
$$
All the differentials will be of degree $1$ so we will not mention this again.

A dg $\bfA$-module is a $\BZ$-graded ${\bf A}$-module $M=\bigoplus\limits_{n\in \BZ}M^{(n)}$ together with a differential $\partial:M\rightarrow M$ satisfying the graded Leibniz rule. For any $k\in \BZ$, we define the $k$-shifted dg ${\bf A}$-module $M[k]$ with a shifted grading $M[k]^{(n)}\coloneqq M^{(n+k)}$ and the same structure of ${\bf A}$-module and differential $\partial$. Given two dg modules $M$ and $N$, we denote by $\Hom_{\bf A}^\gr(M,N)$ the vector space of degree preserving ${\bf A}$-module homomorphisms.\footnote{Note that $f\in \Hom^\gr_\bfA(M,N)$ is required to preserve the gradings but not the differentials.} We define a complex of vector spaces 
$$\Hom_{\bf A}^\bullet(M,N)\coloneqq\bigoplus_{n\in \BZ}\Hom_{\bf A}^\gr(M,N[n])
$$
whose differential is given by
$$\partial(f)=f\circ\partial_M-(-1)^{n}\partial_N\circ f\quad \textnormal{for}\quad f\in \Hom_{\bf A}^\gr(M,N[n]). 
$$
This defines a dg category $\Rep^\dg_\bfA$ of dg $\bfA$-modules.

Given any dg category $\bfT$, we can take the associated categories $Z^0(\bfT)$ and $[\bfT]$ whose objects are the same as $\bfT$ but morphisms between two elements $M$ and $N$ are given by 
$$Z^0(\Hom^\bullet_\bfT(M,N)),\quad H^0(\Hom^\bullet_\bfT(M,N)),
$$
respectively. We call $[\bfT]$ the homotopy category of $\bfT$. According to these notations, $Z^0(\Rep^\dg_\bfA)$ is the category of dg $\bfA$-modules with dg homomorphisms and $[\Rep^\dg_\bfA]$ is the homotopy category of dg $\bfA$-modules. 

We say that a dg $\bfA$-module $M$ is acyclic if $H^n(M,\partial_M)=0$ for all $n\in \BZ$. Denote the full subcategory of acyclic dg modules by $\Ac_{\bf A}^\dg\subset \Rep^\dg_\bfA$ which is naturally a dg category. A dg $\bfA$-module $P$ is called h-projective if $(\Hom_{\bfA}^\bullet(P,M),\partial)$ is acyclic for every acyclic $M$. Denote by $\hproj_\bfA^\dg$ the full sub dg category of h-projective dg modules. If $P$ is h-projective and $P\to M$ induces an isomorphism on cohomology, we say $P$ is an h-projective resolution of $M$.

The derived category of ${\bf A}$ is defined as a Verdier quotient
$$D(\bfA)\coloneqq[\Rep_{\bf A}^\dg]\,/\,[\Ac_{\bf A}^\dg]
$$
which is naturally a triangulated category. 
Equivalently, one can first take a dg quotient category $\Rep_{\bf A}^\dg/\Ac_{\bf A}^\dg$ and then take its homotopy category
$$D(\bfA)\simeq [\Rep_{\bf A}^\dg/\Ac_{\bf A}^\dg]\,.
$$
Therefore, the dg category $\Rep_{\bf A}^\dg/\Ac_{\bf A}^\dg$ gives a dg enhancement of the triangulated category $D(\bfA)$. Since $\hproj_\bfA^\dg$ is quasi-equivalent to $\Rep_{\bf A}^\dg/\Ac_{\bf A}^\dg$, we can also think of it as a dg enhancement 
$$D(\bfA)\simeq [\hproj_\bfA^\dg]\,.
$$

An object $M\in D(\bfA)$ is called perfect (also known as compact) if $\Hom_{D(\bfA)}(M,-)$ commutes with arbitrary direct sums. We define a bounded derived category, denoted by $D^b(\bfA)$, as the full subcategory of perfect objects in $D(\bfA)$.\footnote{In the literature, this is often called as a perfect derived category and denoted by $D_{\textup{perf}}(\bfA)$. We use the nonstandard notation and terminology because we work only in cases where the bounded derived category agrees with the perfect derived category.} This is again a triangulated category and is equipped with a dg enhancement induced from that of $D(\bfA)$. Given $M\in D(\bfA)$ we have a derived functor $\RHom_{\bfA}(M, -)\colon D(\bfA)\to D(\BC)$ and we write $\Ext^i_\bfA(M, N)$ for $H^i(\RHom_\bfA(M, N))$. If $\bfA$ is a compact dg algebra (i.e. $\sum_{n\in \BZ}\dim_\BC H^n(\bfA)<\infty$) then $\RHom_\bfA(M, N)\in D^b(\BC)$ as long as $M, N\in D^b(\bfA)$ by \cite[Lemma 3.2]{Shk}.

A morphism between two dg algebras $\bfA$ and $\bfA'$ is a degree preserving algebra homomorphism $\phi:\bfA\rightarrow \bfA'$ with $\phi\circ \partial_{\bfA}=\partial_{\bfA'}\circ \phi$. Such a morphism is called a quasi-isomorphism if it induces an isomorphism $H^n(\bfA,\partial_{\bfA})\simeq H^n(\bfA',\partial_{\bfA'})$ for all $n\in \BZ$. A quasi-isomorphism induces an equivalence of triangulated categories
\begin{center}
    \begin{tikzcd}
D(\bfA) \arrow[r, "f", bend left] & D(\bfA') \arrow[l, "g", bend left]
\end{tikzcd}
\end{center}
where $f$ and $g$ are the extension and restriction of scalars, respectively, with respect to the quasi-isomorphism [Lemma 37.1 of Stacks project]. Since $f$ and $g$ preserves direct sum, they also preserve perfect objects. Therefore, we also have an equivalence $D^b(\bfA)\simeq D^b(\bfA')$.

\subsubsection{Dg quiver}

Roughly speaking, the notion of a dg quiver is obtained by replacing $\BC[Q]/I$ by a $\BZ_{\leq 0}$-graded dg algebra. Precisely, a dg quiver $\dgQ$ is a usual (finite and acyclic) quiver $\dgQ=(\dgQ_0,\dgQ_1,s,t)$ together with a $\BZ_{\leq 0}$-grading function $|\cdot|\colon \dgQ_1\to \BZ_{\leq 0}$ and a differential $\partial$ on $\BC[\dgQ]$ which we explain now. First, a grading function on $\dgQ_1$ is a decomposition into a disjoint union
$$\dgQ_1=\coprod_{n\leq 0}\dgQ_1^{(n)}
$$
which induces a $\BZ_{\leq 0}$-graded algebra structure on $\BC[\dgQ]=\bigoplus\limits_{n\leq 0}\BC[\dgQ]^{(n)}$. Second, a differential on $\BC[\dgQ]$ is a degree $1$ map
$$\partial:\BC[\dgQ]\rightarrow \BC[\dgQ]
$$
satisfying the graded Leibniz rule and preserving the source and target. By Leibniz rule, $\partial$ is uniquely determined by its value on the set of edges $\bfQ_1$. This makes $(\BC[\dgQ],\partial)$ a dg algebra, so all the notions from Section \ref{sec: dg algebra} apply to dg quivers.  Note that $\BC[\dgQ]$ is a compact dg algebra because $\bfQ$ is finite and acyclic. Denote by $\Rep_{\dgQ}^\dg$ the dg category of dg $\BC[\dgQ]$-modules. The associated (bounded) derived category $D^b(\dgQ)$ has dg enhancements induced from $\hproj^\dg_{\bfQ}$.

\begin{example}\label{ex: hprojectivequiver}
    For each $v\in \dgQ_0$, we have a projective dg module $P(v)\coloneqq\BC[\dgQ]\cdot {\bf 1}_v$ and a simple dg module $S(v)\coloneqq{\bf 1}_v\cdot\BC[\dgQ]\cdot {\bf 1}_v$. Given a left dg $\BC[\bfQ]$-module $M$, $M_v=\1_v\cdot M$ is a complex of vector spaces. It is easy to check that the map $\Hom_{\dgQ}^\gr(P(v), M[n])\to M_v^{(n)}$ defined by $f\mapsto (-1)^{\frac{n(n+1)}{2}}f(\1_v)$ produces an isomorphism of complexes of vector spaces 
    $$\Hom_{\dgQ}^\bullet(P(v), M)\simeq M_v\,.
    $$ 
    In particular, $P(v)$ is h-projective and perfect since clearly $M$ acyclic implies that $M_v$ is acyclic and the $(-)_v$ construction commutes with arbitrary direct sums. 
\end{example}

The underlying quiver of a dg quiver $\dgQ$ is defined as $Q\coloneqq(\dgQ_0, \dgQ_1^{(0)},s,t)$ consisting only of degree zero arrows. The underlying quiver is equipped with an ideal of relations 
$$I\coloneqq\textnormal{image}\big(\partial:\BC[\dgQ]^{(-1)}\rightarrow \BC[\dgQ]^{(0)}\big)\subseteq \BC[\dgQ]^{(0)}=\BC[Q].
$$
By definition, we have 
$$H^0(\BC[\dgQ],\partial)=\BC[Q]/I. 
$$

\begin{example}\label{ex: qsdgquiver}
Let $(Q,I)$ be a quiver with relations. Fix a choice of generating relations $I=(r_1,\dots,r_n)$. Given such a choice, we have a dg quiver $\dgQ$ such that 
$$\dgQ_0=Q_0,\quad \dgQ_1^{(0)}=Q_1,\quad \dgQ_1^{(-1)}=\{\tilde{r}_1,\dots, \tilde{r}_n\},\quad \dgQ_1^{(<-1)}=\emptyset
$$
with $s(\tilde{r}_i)=s(r_i)$ and $t(\tilde{r}_i)=t(r_i)$. Putting $\partial(\tilde{r}_i)=r_i$, this uniquely extends to a differential on $\BC[\dgQ]$ by the graded Leibniz rule. We clearly have $H^0(\BC[\dgQ])=\BC[Q]/I$. 
\end{example}

We call a dg quiver quasi-smooth if $\dgQ_1^{(n)}=\emptyset$ for $n<-1$. The above construction shows that a quasi-smooth dg quiver is equivalent to a quiver with a choice of generating relations $I=(r_1,\dots, r_n)$. We will use these two notions interchangeably. When $\bfQ$ is quasi-smooth, we will write $Q_2=\bfQ_1^{(-1)}$, so that $Q_0, Q_1, Q_2$ are the sets of vertices, edges and relations, respectively; this matches the notation in \cite[Definition 6.2]{Jo21}.

\subsubsection{Path algebra as a tensor algebra}
\label{subsubsec: tensoralgebra}
The description of the path algebra that we now give will later be used to describe the standard resolution on a dg quiver.

Let $Q$ be a (usual) quiver and let $S$ be the algebra
\[S=\bigoplus_{v\in Q_0}\BC\1_v\]
with multiplication given by ${\bf 1}_v^2={\bf 1}_v$ and $\1_v\1_w=0$ for $v\neq w$. Note that a left $S$-modules is the same as a $\BC$-vector space $T$ together with a decomposition $T=\oplus_{v\in Q_0}T(v, -)$, where $T(v, -)=\1_v \cdot T$. Similarly, a right $S$-module is a vector space together with a decomposition with summands $T(-, v)=T\cdot \1_v$ and a bimodule comes with a decomposition with summands $T(v,w)=\1_v\cdot T\cdot \1_w$. If $T_1$ is a right $S$-module and $T_2$ is a left $S$-module then
\[T_1\otimes_S T_2=\bigoplus_{v\in Q_0}T_1(-, v)\otimes_\BC T_2(v, -)\,.\]

The source and target maps give $T:=\bigoplus_{e\in Q_1}\BC e$ the structure of $S$-bimodule where $T(v, w)$ is spanned by the arrows $w\mapsto v$. Then the path algebra of a (usual) quiver can be described as the tensor algebra of $T$ over $S$
\[\CT_S(T)=\bigoplus_{n\geq 0}\underbrace{T \otimes_S T\otimes_S\ldots \otimes_S T}_{n\textup{ times}}\,.\]

If we replace $Q$ by a dg quiver $\bfQ$, then $T$ becomes naturally a graded $S$-bimodule. Endowing the path algebra $\BC[\bfQ]$ with a differential is the same as giving a map of degree 1 between the graded $S$-bimodules $T\to \CT_S(T)=\BC[\dgQ]$.

\subsection{Dg replacement of quiver with relations}

A dg replacement of a quiver with relations $(Q,I)$ is a dg quiver $\dgQ$ with the following properties: 
\begin{enumerate}
    \item the underlying quiver of $\dgQ$ is $Q$,
    \item the zeroth cohomology is $H^0(\BC[\dgQ],\partial)=\BC[Q]/I$,
    \item $H^i(\BC[\dgQ],\partial)=0$ for all $i<0$. 
\end{enumerate}
By definition, we have a quasi-isomorphism $\BC[\dgQ]\rightarrow \BC[Q]/I$ where $\BC[Q]/I$ is given a trivial dg structure. This induces an equivalence between triangulated categories
\begin{center}
    \begin{tikzcd}
D^b(\dgQ) \arrow[r, "f", bend left] & D^b(Q,I) \arrow[l, "g", bend left]
\end{tikzcd}
\end{center}
where $f$ and $g$ are the extension and restriction of scalars, respectively. In particular, $\hproj_{\bfQ}^\dg$ induces a dg enhancement of $D^b(Q,I)$.

A dg replacement of $(Q,I)$ always exists, see \cite[Lemma 2.2]{kalckyang} or \cite[Construction 2.6]{oppermann}. Moreover, it was observed in \cite[Remark 2.9]{oppermann} that Koszul duality provides a canonical (up to choice of basis) dg replacement $\bfQ$ with number of arrows specified by the dimensions of Ext groups of the simple modules in $\Rep_{Q,I}$. For the convenience of the reader, we state this and provide a proof below.

\begin{theorem}\label{thm: koszuldgquiver}
    Let $(Q,I)$ be an (acyclic, finite) quiver with relations and set $A\coloneqq\BC[Q]/I$. Then there is a dg replacement $\bfQ$ of $(Q,I)$ such that the number of arrows from $v\in Q_0=\bfQ_0$ to $w\in Q_0=\bfQ_0$ of degree $-k$ is
    \[\dim \Ext^{k+1}_{A}(S(v), S(w))\,.\]

    In particular, if $\Rep_{Q,I}$ has homological dimension 2 then $(Q, I)$ admits a quasi-smooth dg replacement, i.e., there is a choice of generators of $I$ such that the dg quiver in Example \ref{ex: qsdgquiver} is a dg replacement.
\end{theorem}

\begin{proof}
This is a consequence of Koszul duality for augmented $A_\infty$-algebras over $S$, see \cite[Theorem A]{LPWZ} (over a field) and \cite[Lemma 3.6, Lemma 3.8]{su}. An $A_\infty$-algebra over $S$ is a graded $S$-bimodule together with a family of higher multiplications $m_n\colon A^{\otimes_S n}\to A$ satisfying certain compatibilities; in particular $m_1$ is a derivation and any dg algebra is a $A_\infty$ algebra with $m_n=0$ for $n>2$. Given an augmented $A_\infty$-algebra $A$ over $S$ with augmentation ideal $J$, the authors in \cite{LPWZ, su} define the Koszul dual of $A$, which is an augmented dg algebra over $S$ (note that any $A_\infty$-algebra is quasi-isomorphic to a dg algebra). As a graded $S$-bimodule, it is the dual of the bar construction
\[E(A)=\CT_S\big(D(J[1])\big)\,,\]
where $D(-)=\Hom_S(-, S)$ and $J[1]$ denotes the suspension of $J$. 

We apply \cite[Theorem A]{LPWZ}, \cite[Lemma 3.8]{su} to the algebra $A=\BC[Q]/I$ (regarded as a $A_\infty$-algebra with $m_n=0$ for $n\neq 2$) with augmentation ideal $J$ given by paths of length $\geq 1$; it says that $E(E(A))$ is a dg algebra quasi-isomorphic to $A$. By a theorem of Kadeishvili \cite{kadeishvili}, for any $A_\infty$-algebra $E$ there is a $A_\infty$-algebra structure on the cohomology $H^\ast(E)$ (with differential $m_1=0$) that makes $H^\ast(E)$ quasi-isomorphic to $E$. In particular, since the dual of the bar construction computes $\Ext$ groups, there is a $A_\infty$-algebra structure on the $\Ext$ algebra 
\[\Ext^*_A(S, S)\,\]
that makes is quasi-isomorphic to $E(A)$; the augmentation ideal is $\Ext^{\geq 1}_A(S, S)$. It follows that $A$ is quasi-isomorphic to the dg algebra
\[E(E(A))\simeq \CT_S\big(D\big(\Ext^{\geq 1}_A(S,S)[1]\big)\big)\]
where the differential is induced by the $A_\infty$-algebra structure on the $\Ext$ algebra. After choosing a basis of $D\big(\Ext^{\geq 1}_A(S,S)[1]\big)$ we get the required dg quiver $\bfQ$, with arrows in bijection with the elements of the aforementioned basis.
\end{proof}

\begin{example}\label{ex: beilinsondgquiver}
 Consider the Beilinson quiver with relations $(Q,I)$ from Example~\ref{ex: beilinson}. Denote by $\dgQ$ the quasi-smooth dg quiver according to the choice of the 6 generators of $I$ in Example \ref{ex: beilinson} (recall Example \ref{ex: qsdgquiver}). This dg quiver is the canonical dg replacement of $(Q, I)$ given by Theorem \ref{thm: koszuldgquiver}, as we now proceed to explain. Recall \cite{beil} that there is an isomorphism between $D^b(Q,I)$ and $D^b(\BP^2)$ sending the simple modules $S(1), S(2), S(3)$ to
\[E_1=\CO_{\BP^2}(-1)[2]\,,\,E_2=\CO_{\BP^2}[1]\,,\,
E_3=\CO_{\BP^2}(1)\,,\]
respectively. The $\Ext$-algebra appearing in the proof of Theorem \ref{thm: koszuldgquiver} is
    \[S\oplus \Ext^1(E_1, E_2)\oplus \Ext^1(E_2, E_3)\oplus \Ext^2(E_1, E_3)\,.\]
    Its $A_\infty$ structure, in this case, is just of a usual graded associative algebra, i.e. $m_n=0$ for $n>2$ since there are no paths of length $> 2$ in $Q$. Its augmentation ideal is $\Ext^{\geq 1}$. Therefore, the path algebra of the canonical dg replacement is
     \[\CT_S\big(\hspace{-2pt}\underbrace{D\Ext^1(E_1, E_2)\oplus D\Ext^1(E_2, E_3)}_{\deg\, 0}\oplus \underbrace{D\Ext^2(E_1, E_3)}_{\deg\, -1}\hspace{-2pt}\big)\,.\]
     The differential is the dual of the multiplication map
     \[\Ext^1(E_1, E_2)\otimes \Ext^1(E_2, E_3)\to \Ext^2(E_1, E_3)\,.\]
     Note now that 
     \[\Ext^1(E_1, E_2)\simeq H^0(\CO_{\BP^2}(1))\simeq \Ext^1(E_2, E_3)\textup{ and }\Ext^2(E_1, E_3)\simeq H^0(\CO_{\BP^2}(2))\,,\]
     so the multiplication map is identified with 
     \[H^0(\CO_{\BP^2}(1))\otimes H^0(\CO_{\BP^2}(1))\to \Sym^2(H^0(\CO_{\BP^2}(1)))\simeq H^0(\CO_{\BP^2}(2))\,.\]
     Thus, the differential is
     \[\delta\colon \Sym^2(DH^0(\CO_{\BP^2}(1)))\to DH^0(\CO_{\BP^2}(1))\otimes DH^0(\CO_{\BP^2}(1))\,,\]
     given by the symmetrization map $xy\mapsto x\otimes y+y\otimes x$. By picking a basis $\{x_1, x_2, x_3\}$ of $DH^0(\CO_{\BP^2}(1))$ and corresponding basis $\{a_1, a_2, a_3\}$ and $\{b_1, b_2, b_3\}$ of $D\Ext^1(E_1, E_2)$ and $D\Ext^1(E_2, E_3)$, respectively, we get a basis of $D\Ext^2(E_1, E_3)\simeq \Sym^2(DH^0(\CO_{\BP^2}(1)))$ given by $\{x_ix_j\}_{1\leq i\leq j\leq 3}$. These choices of basis produce exactly the quasi-smooth dg quiver $\dgQ$ from Example \ref{ex: beilinson}.
\end{example}

\subsection{Standard resolution of dg modules}\,

We now recall some statements about standard resolutions of representations in the context of dg quivers and quivers with relations; see Section \ref{sec: standard resolution without relations} for the case of a quiver with no relations. 

We start with the dg quiver case and then use dg replacements to deduce the results we want in the case of quivers with relations. Let $\bfQ$ be a dg quiver and denote by $\bfA=\BC[\bfQ]$ its path dg algebra. Recall from \ref{subsubsec: tensoralgebra} that $\BC[\bfQ]$ is described as a tensor algebra $\bfA=\CT_S(T)$ with differential induced by a degree 1 map of $S$-bimodules $\partial\colon T\to \bfA$.

\begin{theorem}[{\cite[Proposition 3.7]{kellerdefcy}}]\label{thm: dgstandardresolution}
Let $M$ be a perfect left dg $\bfA$-module. Then we have the following distinguished triangle in $D^b(\bfA)$:
\[\bfA\otimes_S T\otimes_S M\rightarrow \bfA\otimes_S M\rightarrow M\,.\]
Hence the cone $C(\bfA\otimes_S T\otimes_S M\rightarrow \bfA\otimes_S M)$ is an h-projective resolution of $M$. 
\end{theorem}
\begin{proof}
    This is shown in loc. cit. when $M=\bfA$ in the category of $\bfA$-bimodules, so the statement here just follows from tensoring on the right with $M$.

    We just make a few remarks about the theorem. The two maps in the distinguished triangle are $a\otimes e\otimes x\mapsto a e\otimes x-a\otimes e x$ and $a\otimes x\mapsto ax$, respectively. Identifications from Section \ref{subsubsec: tensoralgebra} give

    \begin{equation}
        \label{eq: termsdgresolution}
    \bfA\otimes_S M=\bigoplus_{v\in \bfQ_0} P(v)\otimes_\BC M_v
    \quad\textup{and}\quad
    \bfA\otimes_S T\otimes_S M=\bigoplus_{e\in \bfQ_1} P(t(e))\otimes_\BC M_{s(e)}[-|e|]\,.\end{equation}
    Note that $M_v\simeq \RHom_\bfA(P(v),M)$ is a bounded complex since $\bfA$ is compact and $P(v), M$ are perfect. Therefore, the two terms in \eqref{eq: termsdgresolution} are h-projective and perfect. The differential $\delta$ on $\bfA\otimes_S T\otimes_S \bfA$ (making it an $\bfA$-bimodule, and hence making $\bfA\otimes_S T\otimes_S M$ a left $\bfA$-module) is defined by the composition 
    \[T\xrightarrow{\partial} \bfA\xrightarrow{\nabla} \bfA\otimes_S T\otimes_S \bfA\]
    where $\nabla$ is defined on paths by
    \[\nabla(e_1\ldots e_k)=\sum_{i=1}^n e_1\ldots e_{i-1}\otimes e_i\otimes e_{i+1}\ldots e_n\,.\qedhere\]
\end{proof}

This h-projective resolution gives an easy computation of $\Ext$ groups and, in particular, of the Euler pairing
\[\chi_{\bfQ}(M, N)=\sum_{n\in \BZ}(-1)^n \dim \Ext^n_\dgQ(M, N)\,.\]
As usual, the Euler pairing is well-defined in $K(\dgQ)=\BZ^{\bfQ_0}$.

\begin{proposition}\label{prop: extdga} 
    Let $M, N$ be perfect left dg $\bfA$-modules. Then $\RHom_{\bfA}(M, N)$ is isomorphic in $D^b(\BC)$ to the cone of
    \[\oplus_{e\in \bfQ_1}\RHom_\BC(M_{s(e)}, N_{t(e)})[\,|e|\,]\to \oplus_{v\in \bfQ_0}\RHom_\BC(M_v, N_v)\,.\]
    In particular, we have
    \[\chi_{\bfQ}(d, d')=\sum_{v\in \bfQ_0}d_v\cdot d'_v-\sum_{e\in \bfQ_1}(-1)^{|e|}d_{s(e)}\cdot d'_{t(e)}\,.\]
\end{proposition}
\begin{proof}
    This follows from applying $\RHom_{\bfA}(-, N)$ to the standard resolution in Theorem \ref{thm: dgstandardresolution} and using the Hom--tensor adjunction, as well as the remark in Example~\ref{ex: hprojectivequiver}. Recalling equation \eqref{eq: termsdgresolution}, we have
    \begin{align*}\RHom_{\bfA}\big(\bfA\otimes_S T\otimes_S M, N\big)&=\bigoplus_{e\in \bfQ_1}\RHom_{\bfA}\big(P(t(e))\otimes_\BC M_{s(e)}[-|e|], N\big)\\
    &=\bigoplus_{e\in \bfQ_1}\RHom_{\BC}\big(M_{s(e)}, \RHom_\bfA(P(t(e)), N)\big)[\,|e|\,]\\
    &=\bigoplus_{e\in \bfQ_1}\RHom_{\BC}\big(M_{s(e)}, N_{t(e)}\big)[\,|e|\,]\,.
    \end{align*}
    A similar analysis holds for $\bfA\otimes_S M$.
\end{proof}

The next proposition uses the standard resolution for dg quivers to give a standard resolution for quivers with relations. The result, without the injectivity part, appears in \cite[Proposition 2.1]{bardzell}.

\begin{proposition}\label{prop: standardquasismooth}
Let $(Q, I)$ be a quiver with relations and let $A=\BC[Q]/I$. Let $M$ be a left $A$-module. Fix a set of generators of $I$ and corresponding quasi-smooth quiver $\bfQ$ with set of arrows $Q_2=\bfQ_1^{(-1)}$ of degree $-1$ corresponding to these generators. Then we have an exact sequence of $A$-modules
\[\bigoplus_{\tilde{r}\in Q_2}P(t(\tilde{r}))\otimes_\BC M_{s(\tilde{r})}\to \bigoplus_{e\in Q_1}P(t(e))\otimes_\BC M_{s(e)}\to \bigoplus_{v\in Q_0}P(v)\otimes_\BC M_v\to M\to 0\,.\]
Furthermore, if $\bfQ$ is a dg replacement of $Q$ then the left most arrow is injective and $\Rep_{Q,I}$ has homological dimension at most 2. 
\end{proposition}
\begin{proof}
Let $\bfA=\BC[\bfQ]$. Then $A=H^0(\bfA)$ and there is a map $\bfA\to A$ of dg algebras (where we regard $A$ as a dg algebra supported in degree 0). Applying Theorem~\ref{thm: dgstandardresolution} to the module $M_{\bfA}$ obtained by restriction of scalars and tensoring the exact triangle by $A\otimes_{\bfA}^L-$ we obtain the following exact triangle in $D^b(A)$:
\[P\to A\otimes_S M\to A\otimes_{\bfA}^L M_{\bfA}\]
where $P$ is the complex 
\[A\otimes_S T^{(-1)}\otimes_S M\xrightarrow{\delta} A\otimes_S T^{(0)}\otimes_S M\,.\]
The exact triangle gives a long exact sequence on cohomology. Since $H^0(A\otimes_\bfA^L M_\bfA)=M$ we obtain
\begin{equation}\label{eq: lescohomology}\mathrm{Tor}_1^{\bfA}(A, M_{\bfA})\rightarrow H^0(P)\rightarrow A\otimes_S M\rightarrow M\rightarrow 0\,.\end{equation}
We claim that $\mathrm{Tor}_1^{\bfA}(A, M_\bfA)=0$. Let $\tau^{\leq -1}\bfA$ be the dg $\bfA$-bimodule obtained by truncating $\bfA$:
\[\ldots\to A^{(-3)}\to A^{(-2)}\to \ker(\partial_{-1})\]
where $\ker(\partial_{-1})$ is in degree $-1$. We have a distinguished triangle 
\[\tau^{\leq -1}\bfA\to \bfA\to A\,.\]
Applying $-\otimes_\bfA^L M_\bfA$ to this triangle and looking at the long exact sequence on cohomology gives an exact sequence
\[0=\mathrm{Tor}_1^{\bfA}(\bfA, M_\bfA)\to \mathrm{Tor}_1^{\bfA}(A, M_\bfA)\to \mathrm{Tor}_0^{\bfA}(\tau^{\leq -1}\bfA, M_\bfA)=0\]
where the last equality holds due to the fact that $\tau^{\leq -1}\bfA$ is supported in $(-\infty, -1]$ and $M_\bfA$ is supported in degree $0$.

By definition, $H^0(P)$ is $\coker(\delta)$, so we obtain from \eqref{eq: lescohomology} an exact sequence 
\[A\otimes_S T^{(-1)}\otimes_S M\xrightarrow{\delta} A\otimes_S T^{(0)}\otimes_S M\rightarrow A\otimes_S M\rightarrow M\rightarrow 0\,,\]
which is exactly the claim. 

Suppose now that $\bfQ$ is a dg replacement of $Q$, i.e. $\bfA\to A$ is a quasi-isomorphism. Then $A\otimes_\bfA^L M_\bfA$ is isomorphic to $M$; in particular, $\mathrm{Tor}_i^\bfA(A, M_{\bfA})=H^{-i}(A\otimes_\bfA^L M_\bfA)=0$ for $i\neq 0$. It follows that
\[\ker(\delta)=H^{-1}(P)\simeq \mathrm{Tor}_2^\bfA(A, M_{\bfA})=0\,.\qedhere\]
\end{proof}

\begin{proposition}\label{prop: extquasismooth}
    Let $(Q, I)$ be a quiver with relations such that $A=\BC [Q]/I$ has homological dimension $\leq 2$ and let $\bfQ$ be a quasi-smooth dg replacement of $(Q, I)$. Given two left $A$-modules $M$ and $N$, $\RHom_A(M, N)$ is represented by the complex
    \[\bigoplus_{\tilde{r}\in Q_2}\Hom_\BC(M_{s(\tilde{r})}, N_{t(\tilde{r})})\to \bigoplus_{e\in Q_1}\Hom_\BC(M_{s(e)}, N_{t(e)})\to \bigoplus_{v\in Q_0}\Hom_\BC(M_v, N_v)\,.\]
    In particular,
    \[\chi_{Q, I}(d, d')=\sum_{v\in Q_0}d_v\cdot d_v'-\sum_{e\in Q_1}d_{s(e)}\cdot d'_{t(e)}+\sum_{\tilde r\in Q_2}d_{s(\tilde r)}\cdot d_{t(\tilde r)}'\,.\]
\end{proposition}
\begin{proof}
The proof is analogous to that of Proposition \ref{prop: extdga}, but now we use the standard resolution in Proposition \ref{prop: standardquasismooth} instead of the one in Theorem \ref{thm: dgstandardresolution}.
\end{proof}

As an example, the Euler pairing $\chi_{Q, I}$ for the Beilinson quiver with relations (cf. Examples \ref{ex: beilinson}, \ref{ex: beilinsondgquiver}) is
    \[{\chi}_{\dgQ}\big((d_1, d_2, d_3), (d_1', d_2', d_3')\big)=d_1d'_1+d_2 d'_2+d_3d'_3-3 d_1 d'_2-3d_2 d'_3+6 d_1 d'_3\,.\]

\section{Virasoro constraints for quasi-smooth quivers}\label{sec: Virasoro}

\subsection{Derived enhancements and virtual classes} \label{sec: explicit virtual class}

In Section \ref{sec: quiver, moduli}, we considered various moduli stacks and spaces corresponding to a quiver with relations $(Q,I)$. In the same section, it is explained that if we choose generating relations $I=(r_1, \dots, r_n)$, or equivalently the corresponding quasi-smooth dg quiver $\dgQ$ (cf. Example \ref{ex: qsdgquiver}), we have an explicit description of the moduli stack of representations of $(Q,I)$ as a zero locus 
\begin{center}
\begin{tikzcd}
                  & E_d \arrow[d]                   \\
\CM_d=\textnormal{Zero}(s_d) \arrow[r,"i_d", hook] & \CA_d \arrow[u, "s_d"', bend right]
\end{tikzcd}
\end{center}
inside the smooth stack $\CA_d$. By instead taking the derived zero locus in the same diagram, we obtain a derived stack 
$$\bm{\CM}_d\coloneqq\textbf{Zero}(s_d)\hookrightarrow \CA_d. 
$$
This is a quasi-smooth derived stack of (virtual) dimension 
\begin{align*}
    \dim(\bm{\CM}_d)
    &=\dim(\CA_d)-\rk(E_d)\\
    &=-\sum_{v\in Q_0} d_v^2+\sum_{e\in Q_1} d_{s(e)}\cdot d_{t(e)}-\sum_{\tilde{r}\in Q_2} d_{s(\tilde{r})}\cdot d_{t(\tilde{r})}\\
    &=-\chi_{\dgQ}(d,d). 
\end{align*}
The same construction works for the $\theta$-(semi)stable locus and for their good moduli spaces \cite[Theorem 2.12]{AHPS}. In particular, we have a derived good moduli scheme 
$$\bm{\CM}^{\theta-\textup{ss}}_d=[\bm{R}^{\theta-\textup{ss}}_d/\GL_d]\rightarrow \bm{M}^{\theta-\textup{ss}}_d
$$
and a $\BG_m$-gerbe
$$\bm{\CM}_d^{\theta-\textup{st}}=[\bm{R}^{\theta-\textup{st}}_d/\GL_d]\rightarrow  \bm{M}^{\theta-\textup{st}}_d
$$
whose classical truncations recover the constructions in Section \ref{sec: quiver, moduli}. Here, the derived enhancement of the $\theta$-stable locus can be described as a zero locus
\begin{center}
\begin{tikzcd}
                  & E_d^\pl \arrow[d]                   \\
\bm{M}_d^{\theta-\textup{st}}=\textbf{Zero}(s_d^\pl) \arrow[r,"\bm{i}_d", hook] & A_d^{\theta-\textup{st}}. \arrow[u, "s_d^\pl"', bend right]
\end{tikzcd}
\end{center}
where $E_d^{\pl}$ is the descent of $\tilde{E}_d$ from $A^{\theta-\textup{st}}_d\subseteq A_d$ by the $\textnormal{PGL}_d$-action. This defines a quasi-smooth derived enhancement of $M_d^{\theta-\textup{st}}$ which induces a virtual class 
$$\big[M^{\theta-\textup{st}}_d\big]^\vir\in A_{1-{\chi}_{\dgQ}(d,d)}(M^{\theta-\textup{st}}_d). 
$$
We emphasize again that even if the moduli space $M^{\theta-\textup{st}}_d$ is independent of the choice of generators $I=(r_1,\dots,r_n)$, its derived enhancement and virtual class, even the virtual dimension, depends on the choice of generators, or equivalently on the corresponding quasi-smooth dg quiver $\dgQ$. 

Derived structures in the presence of framing can also be defined analogously. Given a quasi-smooth dg quiver $\dgQ$ and a framing vector $f\in \BN^{\dgQ_0}\backslash \{0\}$, we can define an unframed dg quiver $\dgQ^f$ in the same way as in Section \ref{sec: f/uf corr} where all the added arrows from $\infty$ are assigned to be degree zero. Then the framed/unframed correspondence also holds in the derived sense. We leave the details to the reader.

\subsection{Descendent algebra and Virasoro operators}\,
\label{subsec: virasoroops}
\subsubsection{Descendent algebra}\label{sec: descendent algebra}
Let $\dgQ$ be a quasi-smooth dg quiver with underlying quiver with relations $(Q,I)$. We explain how to obtain natural cohomology classes on the moduli stack $\CM$ via the universal representation. The descendent algebra of $\dgQ$, denoted by $\BD^{\dgQ}$, is a free commutative $\BQ$-algebra generated by symbols
$$\Big\{\ch_k(v)\,\Big|\, k\in \BN,\ v\in Q_0\Big\}\,. 
$$
We will denote\footnote{In the analogy with the descendent algebra for moduli spaces of sheaves on a variety $X$ (cf. \cite[Section 2.2]{blm}), $H(\dgQ)$ plays the role of the cohomology $H^\ast(X)$.} by $H(\dgQ)=\BQ^{Q_0}$ the set of formal linear combinations of vertices. Given $a\in H(\dgQ)$ with component $a_v\in \BQ$ in the $v$-th entry, we define
\[\ch_k(a)=\sum_{v\in Q_0}a_v\ch_k(v)\in \BD^\dgQ\,.\]

We define a $\BQ$-algebra homomorphism, called realization homomorphism,
$$\xi:\BD^{\dgQ}\rightarrow H^*(\CM,\BQ),\quad \ch_k(v)\mapsto \ch_k(\CV_{v}). 
$$
If we restrict this homomorphism to $\CM_d$, or any of its open substacks, it factors through the quotient algebra
$$\xi_d=\xi:\BD^{\dgQ}_d:=\BD^\dgQ/\langle \ch_0(v)=d_v\rangle\rightarrow H^*(\CM_{d},\BQ)\,.
$$

\subsubsection{Virasoro operators}

We define the Virasoro operators $\{\bL_n\,|\,n\geq -1\}$ on the descendent algebra $\BD^{\dgQ}$. The operators naturally decompose into two parts $\bL_n=\bR_n+\bT_n$. First, $\bR_n$ is a derivation operator such that
$$\bR_n(\ch_k(v))=k\cdot(k+1)\cdots (k+n)\,\ch_{k+n}(v). 
$$
Second, $\bT_n$ is a multiplication operator by an element
$$\bT_n=\sum_{a+b=n}a!b!\left(
\sum_{v\in Q_0} \ch_a(v)\ch_b(v)
-\sum_{e\in Q_1} \ch_a(s(e))\ch_b(t(e))
+\sum_{\tilde{r}\in Q_2} \ch_a(s(\tilde{r}))\ch_b(t(\tilde{r}))
\right)
$$
for which we use the same notation. Note that this can be succinctly written using the Euler form as
\[\bT_n=\sum_{a+b=n}a!b!\sum_{v,w\in Q_0}{\chi}_{\dgQ}(v,w)\ch_a(v)\ch_b(w)\,.\]

One can check that these operators satisfy (the dual version of) the Virasoro bracket formula 
$$[\bL_n,\bL_m]=(m-n)\bL_{n+m}\in \End(\BD^{\dgQ}). 
$$
Since the $\bR_n$ operators annihilate $\ch_0(v)$ classes, the Virasoro operators also descend to $\BD^{\dgQ}_d$. To define cohomology classes on the projective linear stack, we use the weight zero descendent subalgebras defined as
$$\BD^{\dgQ}_{\textnormal{wt}_0}\coloneqq\ker(\bR_{-1})\subset \BD^{\dgQ},\quad \BD^{\dgQ}_{d,\textnormal{wt}_0}\coloneqq\ker(\bR_{-1})\subset \BD^{\dgQ}_d\,. 
$$

\begin{lemma}\label{lem: weignt zero realization}
    There is a well-defined homomorphism 
$$\xi:\BD^{\dgQ}_{d,\textnormal{wt}_0}\rightarrow H^*({\CM}^{\pl}_d,\BQ). 
$$
\end{lemma}
\begin{proof} We closely follow \cite[Lemma 4.10]{blm}. When $d=0$, the homomorphism is defined so that all $\ch_k(i)$ are mapped to zero. Now we assume that $d\neq 0$. Then the projective linearlization defines a $\BG_m$-gerbe
    $$\Pi^\pl_d:{\CM}_d\rightarrow {\CM}^\pl_d
    $$
    that comes from the $B\BG_m$-action 
    $$\rho:B\BG_m\times {\CM}_d\rightarrow {\CM}_d
    $$
    introduced in Section \ref{sec: Joyce's VA}. From this description, it follows that 
    $$H^*({\CM}^\pl_d,\BQ)=\big\{
    x\in H^*({\CM}_d,\BQ)\,\big|\,\rho^*(x)=1\boxtimes x
    \big\}\subset H^*({\CM}_d,\BQ).
    $$
    
On the other hand, we have a commuting diagram 
\begin{center}
        \begin{tikzcd}
\BD^{\dgQ}_d \arrow[r, "e^{\zeta \bR_{-1}}"] \arrow[d, "\xi"'] & \BD^{\dgQ}_d\llbracket\zeta\rrbracket \arrow[d, "\xi"] \\
H^*({\CM}_d,\BQ) \arrow[r, "\rho^*"']                & H^*(B\BG_m,\BQ)\otimes H^*({\CM}_d,\BQ)       
\end{tikzcd}
    \end{center}
    as in \cite[Lemma 2.8]{blm}. Here $\xi$ on the right column maps $\zeta$ to $c_1(\CQ)$ where $\CQ$ is the universal line bundle over $B\BG_m$. Therefore, if $D\in \BD^{\dgQ}_{d,\textnormal{wt}_0}$, then 
    $$\rho^*\circ\xi(D)=\xi\circ e^{\zeta \bR_{-1}}(D)=1\boxtimes \xi(D).
    $$
    This implies that $\xi(D)$ for $D\in \BD^{\dgQ}_{d,\textnormal{wt}_0}$ lies in the subspace $H^*({\CM}^\pl_d,\BQ)\subset H^*({\CM}_d,\BQ)$.\qedhere
\end{proof}

By restricting to an open subset ${M}_d^{\theta-\textup{st}}\subseteq {\CM}_d^{\pl}$ we can get an analogous realization homomorphism $\xi:\BD^{\dgQ}_{d,\textnormal{wt}_0}\rightarrow H^*({M}_d^{\theta-\textup{st}})$ for the moduli space.

\subsection{Virasoro constraints}
We now state the Virasoro constraints for a quasi-smooth dg quiver $\dgQ$. Define the weight zero Virasoro operator 
$$\bL_{\textnormal{wt}_0}\coloneqq\sum_{n\geq -1}\frac{(-1)^n}{(n+1)!}\bL_n\circ (\bR_{-1})^{n+1}:\BD^{\dgQ}\rightarrow \BD^{\dgQ}_{\textnormal{wt}_0}. 
$$

\begin{theorem}\label{thm: wt0virasoro}
    Let $\dgQ$ be a quasi-smooth dg quiver. If ${M}^{\theta-\textup{st}}_d={M}^{\theta-\textup{ss}}_d$, then we have 
\begin{equation}
    \label{eq: wt0virasoro}
\int_{[{M}^{\theta}_d]^\vir}\xi\circ \bL_{\textnormal{wt}_0}(D)=0\quad\textnormal{for all}\quad D\in \BD^{\dgQ}. 
\end{equation}
\end{theorem}
Proof of this theorem will be carried out in Section \ref{sec: Virasoro and wall-crossing}. In the same section, we also explain how this statement generalizes to the case with strictly $\theta$-semistable representations using the vertex algebra formalism and Joyce's invariant class.

\subsubsection{Framed analog}\label{subsubsec: framedvirasoro}

Let $\dgQ$ be a quasi-smooth dg quiver and $f\in \BN^{Q_0}\backslash\{0\}$ be a framing vector. We have the same realization homomorphism for ${\CM}_{f\shortrightarrow*}$
$$\xi:\BD^{\dgQ}\rightarrow H^*({\CM}_{f\shortrightarrow*}),\quad \ch_k(v)\mapsto \ch_k(\CV_{v})
$$
and for the open subset ${M}^{\theta}_{f\shortrightarrow d}$. Recall that, unlike in the unframed case, the projective scheme $M_{f\to d}^\theta$ admits a unique universal object and a map to the stack ${M}^{\theta}_{f\shortrightarrow d}$.

Define the framed Virasoro operators $\{\bL_n^{f\shortrightarrow *}\,|\,n\geq -1\}$ on the descendent algebra $\BD^{\dgQ}$ by $\bL_n^{f\shortrightarrow *}=\bR_n+\bT_n^{f\shortrightarrow *}$ where $\bR_n$ is the same as before and 
$$\bT_n^{f\shortrightarrow *}=\bT_n-n!\left(\sum_{v\in Q_0} f_v\cdot \ch_n(v)\right).
$$
One can check that these operators also satisfy the Virasoro bracket formula
$$[\bL_n^{f\shortrightarrow *},\bL_m^{f\shortrightarrow *}]=(m-n)\bL_{n+m}^{f\shortrightarrow *}\in \End(\BD^{\dgQ}).
$$

The framed Virasoro constraints are stated as follows. 
\begin{theorem}\label{thm: framedvirasoro}
Let $\dgQ$ be a quasi-smooth dg quiver. For any framing vector $f\in \BN^{Q_0}\backslash\{0\}$ and stability condition $\theta$, we have
\begin{equation}\label{eq: framedvirasoro}
    \int_{[{M}^{\theta}_{f\shortrightarrow d}]^\vir}\xi\circ \bL_{n}^{f\shortrightarrow d}(D)=0\quad\textnormal{for all}\ \ n\geq 0,\ D\in \BD^{\dgQ}.
\end{equation}
\end{theorem}

We now show that the framed/unframed correspondence (cf. Section \ref{sec: f/uf corr}) preserves Virasoro constraints; thus Theorem \ref{thm: framedvirasoro} is implied by Theorem \ref{thm: wt0virasoro}. In the proposition below, $\dgQ$ is a quasi-smooth dg quiver with a framing vector $f\in \BN^{Q_0}\backslash\{0\}$ and $\dgQ^f$ is the corresponding unframed dg quiver. Also, let $\theta$ be any stability condition of $\dgQ$ and $\tilde{\theta}$ be the corresponding one for $\dgQ^f$ according to Section \ref{sec: f/uf corr}.

\begin{proposition}\label{prop: f/uf Virasoro}
Through the framed/unframed correspondence isomorphism $${M}^{\theta}_{f\shortrightarrow d}(Q,I)\simeq {M}_{(1, d)}^{\tilde \theta}(Q^f,I^f)\,,$$ the framed Virasoro constraints \eqref{eq: framedvirasoro} for ${M}^{\theta}_{f\shortrightarrow d}(Q,I)$ become equivalent to the unframed Virasoro constraints \eqref{eq: wt0virasoro} for ${M}_{(1, d)}^{\tilde \theta}(Q^f,I^f)$.

\end{proposition}
\begin{proof}
    Let $\CV=(\CV_v)_{v\in {Q^f_0}}$ be the universal representation on $M_{(1,d)}^{\tilde \theta}(Q^f,I^f)$ normalized so that $\CV_\infty$ is the trivial line bundle. In particular, we have $\xi_{\CV}(\ch_1(\infty))=0$. This is analogue to the notion of $\delta$-normalized universal sheaf in \cite[Definition 2.13]{blm}. In \cite[Proposition 2.16]{blm} it is proven (in the context of sheaves, but the same proof applies verbatim) that the Virasoro constraints for $M_{(1,d)}^{\tilde \theta}(Q^f,I^f)$ are equivalent to
    \begin{equation}
        \label{eq: virasoroinftynormalized}
    \int_{[M_{(1,d)}^{\tilde \theta}(Q^f,I^f)]^\vir}\xi_\CV\big(\big(\bL_n^{\dgQ^f}+\bS_n^\infty\big)(D)\big)=0\quad\textup{ for all }n\geq -1, D\in\BD^{\dgQ^f}\,, \end{equation}
    where $\bL_n^{\bfQ^f}=\bR_n+\bT_n^{\bfQ^f}$ are the Virasoro operators defined in $\BD^{\bfQ^f}$ (see Section \ref{subsec: virasoroops}) and 
    \[\bS_n^\infty=-(n+1)!\ch_n(\infty)-(n+1)!\ch_{n+1}(\infty)\circ \bR_{-1}\,.\]
    By identifying $\BD^{\dgQ^f}=\BD^{\dgQ}\otimes \BQ[\ch_0(\infty), \ch_1(\infty), \ldots]$ we can write
    \[\bT^{\dgQ^f}_n=\bT^{\dgQ}_n-\sum_{v\in Q_0}f_v\sum_{a+b=n}a!b!\ch_a(\infty)\ch_b(v)+\sum_{a+b=n}a!b!\ch_a(\infty)\ch_b(\infty)\,.\]
    Since $\xi_{\CV}(\ch_a(\infty))=0$ for all $a>0$, equation \eqref{eq: virasoroinftynormalized} for $n>0$ (note that for $n=-1$ the equation is trivial and for $n=0$ it is the dimensional constraint) is equivalent to
    \[\int_{[M_{(1,d)}^{\tilde \theta}(Q^f,I^f)]^\vir}\xi_\CV\Big(\Big(\bR_n+\bT_n^{\dgQ}-\sum_{v\in Q_0}f_v n!\ch_n(v)\Big)(D)\Big)=0\quad\textup{ for }D\in \BD^{\dgQ^f}\,.\]
Note that the universal representation on $M^{\theta}_{f\shortrightarrow d}(Q,I)$ is $(\CV_v)_{v\in {Q_0}}$, so when $D\in \BD^{\dgQ}\subseteq \BD^{\dgQ^f}$ these are precisely the framed Virasoro constraints \eqref{eq: framedvirasoro} for $\dgQ$, showing that the unframed Virasoro constraints imply the framed Virasoro constraints. Conversely, the framed Virasoro constraints imply the unframed Virasoro constraints since equation \eqref{eq: virasoroinftynormalized} holds for any $D$ of the form $D=\ch_k(\infty)D'$ with $k>0$ as all the terms appearing trivially vanish due to  $\xi_{\CV}(\ch_{k}(\infty))=\xi_{\CV}(\ch_{k+n}(\infty))=0$.
\end{proof}

\subsection{Geometricity of Virasoro operators}

A simple but fairly interesting observation that the authors learned from A. Mellit is that the framed Virasoro constraints imply that the $\bR_n$ operators descend to cohomology of moduli spaces of framed representations when those are smooth. We also prove the analogous statement for the smooth moduli stack of unframed representations by approximating its cohomology by the cohomology of framed moduli spaces. We call this phenomena the geometricity of the Virasoro operators. Let $
\dgQ$ be a quasi-smooth dg quiver and $(Q,I)$ be the underlying quiver with relations below.

\begin{definition}
    Let $M$ be a moduli space or a moduli stack admitting a universal representation of $(Q,I)$, and hence a realization homomorphism $\xi\colon \BD^\dgQ\to H^\ast(M)$. We say that the Virasoro operator $\bR_n$ is geometric on $M$ if $\bR_n$ descends via $\xi$, i.e. there is a dashed arrow completing the diagram
    \begin{center}
        \begin{tikzcd}
            \BD^\dgQ\arrow[r, "\bR_n"]\arrow[d, "\xi"]& \BD^\dgQ\arrow[d, "\xi"]\\
            H^\ast(M)\arrow[r, dashed, "\bR_n"]& H^\ast(M)\,.
        \end{tikzcd}
    \end{center}
\end{definition}
\begin{remark} (i) Note that $\bT_n$ part of the Virasoro operator descends for a trivial reason since it is just multiplication by an element. So asking for $\bR_n$ to descend is the same as asking for $\bL_n$ to descend. (ii) When $M$ is the moduli stack of representations $\CM^{\theta-\textup{ss}}_d$, the operator $\bR_{-1}$ is always geometric because it comes from the $B\BG_m$-action on the moduli stack.
\end{remark}

This notion is more natural to consider when the realization homomorphism $\xi$ is surjective. In this case, the geometricity of the Virasoro operators is equivalent to the ideal of relations $\ker(\xi)$ being closed under the action of $\bR_n$. We make the following assumptions to guarantee smoothness of the involved moduli stacks and spaces, and we prove surjectivity of $\xi$ under these assumptions.
\begin{assumption}\label{ass: smoothness}
Let $(Q,I)$ be a quiver with relations together with $d$ and $\theta$. 
    \begin{enumerate} 
        \item $(Q,I)$ has homological dimension at most two. 

        \item For any $V_1, V_2\in{\CM}^{\theta-\textup{ss}}_d$, we have $\Ext^2_{Q,I}(V_1,V_2)=0$.

    \end{enumerate}
\end{assumption}

Under these assumptions, we write $\dgQ$ for the canonical quasi-smooth dg quiver associated to $(Q,I)$. Note that this assumption is automatic if $I=0$. It is also satisfied when the quiver $(Q, I)$ is obtained from certain exceptional collections on del Pezzo surfaces and $\theta$-stability is identified with stability of sheaves on the del Pezzo; we will further explore this in Section \ref{sec: applicaton}.

\begin{remark}
    Under the above assumptions, the moduli stack ${\CM}^{\theta-\textup{ss}}_d$ is smooth by deformation theory. For any framing vector $f\in \BN^{Q_0}\backslash\{0\}$, the moduli space ${M}^\theta_{f\rightarrow d}$ of limit $\theta$-stable framed representations is also smooth by the diagram \eqref{eq: framed approximation}.
\end{remark}

\begin{proposition}\label{prop: surjectivityrealization}
  Under Assumption \ref{ass: smoothness}, the realization homomorphisms 
    $$\BD_d^{\dgQ}\rightarrow H^*({\CM}^{\theta-\textup{ss}}_{d}),\quad \BD_d^{\dgQ}\rightarrow H^*({M}^\theta_{f\rightarrow d})
    $$
    are surjective.
\end{proposition}
\begin{proof} 
The surjectivity of the realization homomorphism for ${M}^{\theta}_{f\rightarrow d}$ can be deduced from \cite[Theorem 1]{King-Walter}, as we now explain. The moduli $M={M}^{\theta}_{f\rightarrow d}$ parametrizes $\BC[Q^f]/I^f$-modules and admits a universal representation 
\[\CV^f=\bigoplus_{v\in Q_0^f}\CV_v=\CO_M\oplus \bigoplus_{v\in Q_0}\CV_v\,,\] 
which is naturally a $\big(\BC[Q^f]/I^f\big)\otimes \CO_M$-module.
First, we observe that for a $\BC[Q^f]/I^f$-module $V^f$ (i.e. a framed representation) we have by Proposition \ref{prop: standardquasismooth} a resolution
\begin{equation}\label{eq: resolutionframed}
0\to \bigoplus_{\tilde{r}\in Q_2}P(t(\tilde{r}))\otimes V_{s(\tilde{r})}\to \bigoplus_{e\in Q_1^f}P(t(e))\otimes V_{s(e)}\to \bigoplus_{v\in Q_0^f}P(v)\otimes V_v\to V^f\to 0\,.\end{equation}
Note that injectivity on the left follows from the corresponding injectivity for the the resolution of the underlying $\BC[Q]/I$ module, which in turn is guaranteed by Assumption \ref{ass: smoothness} (1). In particular, $(Q^f, I^f)$ also has homological dimension 2. By comparing the leftmost arrows in the resolutions given by Proposition \ref{prop: extquasismooth} for $(Q, I)$ and $(Q^f, I^f)$ it is also clear that $\Ext^{2}(V^f, W^f)=0$ if $\Ext^2(V, W)=0$, where $V, W$ are the unframed representations underlying $V^f, W^f$. Hence, by Assumption~\ref{ass: smoothness}~(2) $\Ext^2(V^f, W^f)=0$ for $V^f, W^f$ in ${M}^{\theta}_{f\rightarrow d}$. In particular ${M}^{\theta}_{f\rightarrow d}$ is smooth and we have $[{M}^{\theta}_{f\shortrightarrow d}]^\vir=[{M}^{\theta}_{f\shortrightarrow d}]$. This also verifies the conditions (i) and (ii) in \cite[Theorem 1]{King-Walter}. From \eqref{eq: resolutionframed} we get a resolution of the $\big(\BC[Q^f]/I^f\big) \otimes \CO_M$-module
\begin{equation}\label{eq: globalstandardresolution}
0\to \bigoplus_{\tilde{r}\in Q_2}P(t(\tilde{r}))\otimes \CV_{s(\tilde{r})}\to \bigoplus_{e\in Q_1^f}P(t(e))\otimes \CV_{s(e)}\to \bigoplus_{v\in Q_0^f}P(v)\otimes \CV_v\to \CV^f\to 0\,,\end{equation}
which is condition (iii) in loc. cit., so we conclude that $H^\ast({M}^{\theta}_{f\rightarrow d})$ is generated as an algebra by the Chern classes (or, equivalently, Chern characters) of the vector bundles $\CV_v$; note that $\CV_\infty$ is the trivial vector bundle, so it is enough to use $\CV_v$ for $v\neq \infty$ to generate, which proves the surjectivity of the realization map to the cohomology.

We prove the corresponding statement for the moduli stack of representations by using the approximation result in Proposition \ref{prop: approximation}. Let $\pi:M^\theta_{f\rightarrow d}\rightarrow \CM_d^{\theta-\textup{ss}}$ be the forgetful morphism. This defines a commuting diagram
\begin{center}
    \begin{tikzcd}
                       &   H^\ast(M^\theta_{f\rightarrow d}  )             \\
\BD^{\dgQ}_d \arrow[r] \arrow[ru] & H^\ast(\CM_d^{\theta-\textup{ss}}) \arrow[u, "\pi^*"']
\end{tikzcd}
\end{center}
because $\CV$ in the universal framed representation $(\CV,\phi)$ on $M^\theta_{f\rightarrow d}$ is pulled back from $\CM_d^{\theta-\textup{ss}}$. We have already proven for the framed moduli space $M^\theta_{f\rightarrow d}$ that the realization homomorphism is surjective. The surjectivity for $\CM_d^{\theta-\textup{ss}}$ follows from Proposition \ref{prop: approximation} because we can choose $f$ arbitrarily large according to the cohomological degree.
\end{proof}

\begin{theorem}\label{thm: virasoro rep}
    Under Assumption \ref{ass: smoothness}, the Virasoro operators are geometric on ${\CM}^{\theta-\textup{ss}}_{d}$ and ${M}^\theta_{f\rightarrow d}$, i.e. $\bR_n$ descends via the surjective realization homomorphisms 
    $$\BD_d^{\dgQ}\rightarrow H^*({\CM}^{\theta-\textup{ss}}_{d}),\quad \BD_d^{\dgQ}\rightarrow H^*({M}^\theta_{f\rightarrow d})\,.
    $$
\end{theorem}
\begin{proof}
We use the framed Virasoro constraints to show that the $\bR_n$ operators descend via $\xi:\BD_d^{\dgQ}\rightarrow H^*({M}^\theta_{f\rightarrow d})$. Since $\xi$ is surjective by Proposition \ref{prop: surjectivityrealization}, it is enough to show that $\bR_n$ preserves the kernel of $\xi$. Let $D\in \BD^{\dgQ}_d$ be such that $\xi(D)=0$ and let $E\in \BD^{\dgQ}_d$ be arbitrary. By Theorem \ref{thm: framedvirasoro} we have the following:
\begin{equation*}
    0=\int_{M^\theta_{f\to d}}\xi\big(\bL_{n}^{f\to d}(DE)\big)=\int_{M^\theta_{f\to d}}\xi(\bR_n(D)E)+
\int_{M^\theta_{f\to d}}\xi\big(D\bL_{n}^{f\to d}(E)\big)\,.
\end{equation*}
    Since we assume that $\xi(D)=0$ the last integral vanishes, so \[\int_{M_{f\to d}^\theta}\xi(\bR_n(D))\xi(E)=0\] for any $E$. Since $\xi$ is surjective, by Poincaré duality it follows that $\xi(\bR_n(D))=0$.

The statement for the stack $\CM_d^{\theta-\textup{ss}}$ can be deduced from the framed statement and the approximation result in Proposition \ref{prop: approximation}, as in the proof of Proposition~\ref{prop: surjectivityrealization}.    
\end{proof}

\begin{remark}
The notion of geometricity of the Virasoro operators also makes sense for moduli spaces of sheaves. Geometricity of the Virasoro operators on moduli stacks of semistable bundles on curves or torsion-free sheaves on del Pezzo surfaces can be deduced from the Virasoro constraints shown in \cite{blm} by using arguments similar to the ones here, replacing moduli of limit stable framed representations by moduli of Joyce--Song pairs.
\end{remark}

\section{Vertex algebras from dg categories}
\label{sec: vertexalgebra}

As discovered in \cite{blm}, vertex algebras, in particular the ones constructed by Joyce in \cite{Jo17, grossjoycetanaka, Jo21}, are deeply connected to the Virasoro constraints for moduli spaces of sheaves. In this section, we introduce lattice vertex algebras and Joyce's vertex algebras from dg categories. The main result of this section is a construction of a natural isomorphism from Joyce's vertex algebra to lattice vertex algebra under the assumptions that are satisfied for our applications.

We now briefly recall the necessary basic notions of vertex algebras; the reader can find a much more detailed exposition in \cite{Ka98}. A vertex algebra is a vector space $V$ over a field (which, in this paper, will always be $\BQ$) equipped with a vacuum vector $\ket{0}\in V$, a translation operator $T\colon V\to V$ and a state-field correspondence $Y\colon V\to \End(V)\llbracket z^{-1}, z\rrbracket$. This data has to satisfy certain axioms called vacuum axiom, translation covariance and locality, see \cite[Section 1.3]{Ka98}.\footnote{Vertex algebras in \cite{Ka98} are super-vertex algebras. The vertex algebras considered in this paper are entirely even. Moreover, even though every vertex algebra in this paper admits a $\BZ$-grading, it will not be considered.} The state-field correspondence can be encoded as a $\BZ$-collection of bilinear products $-_{(n)}-\colon V\otimes V\to V$ defined by
\[Y(u, z)v=\sum_{n\in \BZ}u_{(n)} v\, z^{-1-n}\,.\]

Given a vertex algebra $V$, there is a standard way, due to Borcherds \cite{Borcherds}, to construct a Lie algebra associated to it. Define
\[\widecheck V\coloneqq V/T(V)\]
and define a Lie bracket on $\widecheck V$ by
\[[\overline u, \overline v]=\overline{u_{(0)} v}\,,\]
where we denote by $\overline u\in \widecheck V$ the image of $u\in V$ in the quotient. It can be shown that this does not depend on the representatives $u, v$ chosen and that it satisfies the axioms of a Lie algebra. 

Vertex algebras often come with a conformal element $\omega$, see \cite[Definition 4.10]{Ka98}; when this is the case we say that $(V, \omega)$ is a vertex operator algebra. The element $\omega$ induces a representation of the Virasoro Lie algebra on $V$. More precisely, the operators $L_n\coloneqq\omega_{(n+1)}\in \End(V)$ satisfy the Virasoro bracket relations
     \[\big[L_n,L_m\big] = (n-m)L_{m+n} +\frac{n^3-n}{12}\delta_{n+m,0}\, c\cdot \id
     \,,\]
     where $c\in \BQ$ is some constant called the central charge of $(V, \omega)$.

\subsection{Lattice vertex algebras} \label{sec: lattice VA}

In this paper, a lattice is a finitely generated free abelian group $\Lambda$ together with a symmetric bilinear form $B:\Lambda\times \Lambda\rightarrow \BZ$. We do not assume that the bilinear form is non-degenerate. The main example is the lattice of a quasi-smooth dg quiver $\dgQ$, namely $\Lambda=\BZ^{\dgQ_0}$ with the symmetrized Euler form $\chi_{\dgQ}^{\sym}(-,-)$. 

We recall how the lattice $\Lambda$ gives arise to a lattice vertex algebra
$$\textup{VA}(\Lambda,B)\coloneqq\Big(V_\Lambda,\ |0\rangle\in V_\Lambda,\ T:V_\Lambda\rightarrow V_\Lambda,\ Y:V_\Lambda\otimes V_\Lambda\rightarrow V_\Lambda(\hspace{-2pt}(z)\hspace{-2pt})\Big). 
$$
The underlying vector space is defined as a tensor product
$$V_\Lambda=\BQ[\Lambda]\otimes \BD_\Lambda
$$
of a group algebra $\BQ[\Lambda]$ and a free algebra $\BD_\Lambda\coloneqq\Sym(\Lambda\otimes t^{-1}\BQ[t^{-1}])$. Recall that the group algebra has a basis $\{e^\alpha\}_{\alpha\in \Lambda}$ with a multiplication rule $e^{\alpha}\cdot e^{\beta}=e^{\alpha+\beta}$. We denote the element $v\cdot t^{-k}\in \BD_\Lambda$ for $v\in \Lambda$ and $k>0$ simply by $v_{-k}$. Therefore, general elements in $V_\Lambda$ are linear combinations of elements of the form 
$$e^\alpha\otimes v^1_{-k_1}\cdots v^\ell_{-k_\ell} 
$$
where $\alpha,v^1,\dots, v^\ell\in \Lambda$ and $k_1,\dots, k_\ell>0.$ In this notation, the vacuum vector is defined as $\ket{0}\coloneqq e^0\otimes 1\in V_\Lambda$. We will abbreviate $e^\alpha\otimes 1$ to $e^\alpha$ and $e^0\otimes v^1_{-k_1}\cdots v^\ell_{-k_\ell}$ to $v^1_{-k_1}\cdots v^\ell_{-k_\ell}$.

For each $v\in \Lambda$ and $k>0$, we define the creation operation $v_{(-k)}$ as a left multiplication by $v_{-k}$. This defines a free $\BD_\Lambda$-module structure on $V_\Lambda$ with a basis $\{e^\alpha\}_{\alpha\in \Lambda}$. Therefore, to define an operator $A:V_\Lambda\rightarrow V_\Lambda$ it suffices to specify the commutator $[A,v_{(-k)}]$ for all $v\in \Lambda$, $k>0$ and its values on the basis $A(e^\alpha)$ for all $\alpha\in \Lambda$. We define a translation operator $T:V_\Lambda\rightarrow V_\Lambda$ in this way:
$$[T,v_{(-k)}]= k\cdot v_{(-k-1)},\quad T(e^\alpha)=e^\alpha\otimes \alpha_{-1}. 
$$
The definition of the state-field correspondence uses annihilation operators $v_{(k)}$ for $k\geq 0$. These are defined by
\begin{align*}&v_{(0)}(e^\alpha)=B(v, \alpha)e^\alpha\,,\quad v_{(k)}(e^\alpha)=0\quad\textup{ for }k>0\,,\quad \\
&[v_{(k)}, w_{(-l)}]=k\cdot \delta_{k, l}\cdot B(v,w)\cdot \id\quad\textup{ for }k\geq 0, l>0\,.
\end{align*}
Combining the creation and annihilation operators, we define
\begin{equation}\label{eq: statefield1}Y(v_{-1}, z)=\sum_{n\in \BZ}v_{(k)}z^{-1-k}.\end{equation}
We also define
\begin{equation}\label{eq: statefield2}Y(e^\alpha, z)=\epsilon_{\alpha, \beta}z^{B(\alpha, \beta)}e^\alpha\exp\left(-\sum_{k<0}\frac{\alpha_{(k)}}{k}z^{-k}\right)\exp\left(-\sum_{k>0}\frac{\alpha_{(k)}}{k}z^{-k}\right)\end{equation}
when restricted to $e^\beta\otimes \BD_\Lambda\subseteq V_\Lambda$. Here $e^\alpha$ is the operator sending $e^\beta\otimes w$ to $e^{\alpha+\beta}\otimes w$ and $\epsilon_{\alpha, \beta}=\pm 1$ are signs satisfying some compatibility, see \cite[5.4.14]{Ka98}; the vertex algebra is independent from the choice of signs, and when $B$ is obtained as the symmetrization of a bilinear form $b\colon \Lambda\times \Lambda\to \BZ$ there is a canonical choice given by $\epsilon_{\alpha, \beta}=(-1)^{b(\alpha, \beta)}$. The state-field correspondence is entirely determined by \eqref{eq: statefield1}, \eqref{eq: statefield2} together with the reconstruction theorem \cite[Theorem 4.5]{Ka98}.

\subsubsection{Virasoro operators and primary states}
\label{subsubsec: virasorolattice}

When $B$ is a non-degenerate pairing, the lattice vertex algebra $V_\Lambda$ is well known to carry a natural conformal element $\omega\in V_\Lambda$ \cite[Proposition 5.5]{Ka98}. Given a basis $\{v\}$ of $\Lambda\otimes \BQ$, let $\{\hat v\}$ be its dual basis with respect to $B$; then the conformal element is given by 
\begin{equation}\label{eq: conformallattice}\omega=\frac{1}{2}\cdot e^0\otimes \sum_{v}\hat{v}_{-1}v_{-1}\,.\end{equation}
As explained in the beginning of Section \ref{sec: vertexalgebra}, a conformal element induces Virasoro operators $L_n=\omega_{(n+1)}$ for all $n\in \BZ$. The central charge of $\omega$ is equal to the rank of~$\Lambda$.

Even when $B$ is degenerate, it is possible to define a representation of half of the Virasoro Lie algebra. Define operators $L_n\colon V_\Lambda\to V_\Lambda$ for $n\geq -1$ by 
\begin{equation}\label{eq: virlva1}
    L_n(e^\alpha)=\begin{cases}e^\alpha\otimes \alpha_{-1} &\textup{ if }n=-1\\
\frac{1}{2}B(\alpha, \alpha)\big(e^\alpha\otimes 1\big) &\textup{ if }n=0\\
0 &\textup{ if }n>0
    \end{cases}
\end{equation}
and
\begin{equation}\label{eq: virlva2}[L_n, v_{(-k)}]=k\cdot v_{(-k+n)}\,.\end{equation}
Note in particular that $L_{-1}=T$.

\begin{remark}
    We remark that in the case of non-degenerate lattice $(\Lambda,B)$, $L_n\coloneqq\omega_{(n+1)}$ indeed satisfies the above properties; Equation \eqref{eq: virlva1} is $(5.5.21-23)$ in \cite{Ka98} and Equation \eqref{eq: virlva2} follows from \cite[Corollary 4.10]{Ka98}. 
\end{remark}

The operators $L_n$ defined by \eqref{eq: virlva1} and \eqref{eq: virlva2} are functorial in the following sense. 

\begin{proposition}\label{prop: functorial}
    Let $f:(\Lambda,B)\rightarrow (\widetilde \Lambda, \widetilde B)$ be a lattice embedding. Then the induced vertex algebra homomorphism $\phi_f:V_\Lambda\rightarrow V_{\widetilde\Lambda}$ intertwines the operators $L_n$ and $\widetilde L_n$ defined by \eqref{eq: virlva1} and \eqref{eq: virlva2}, i.e., $\phi_f\circ L_n=\widetilde L_n\circ \phi_f$.  
\end{proposition}
\begin{proof}
    The homomorphism $\phi_f$ is defined so that $\phi_f(e^\alpha)=e^{f(\alpha)}$ for $\alpha\in \Lambda$ and the operator $v_{(k)}$ on $V_\Lambda$ and $V_{\widetilde \Lambda}$ for $v\in \Lambda$ are compatible. This implies the proposition.   
\end{proof}
We can study the half of the Virasoro operators by embedding the lattice $\Lambda$ into a non-degenerate lattice. This in particular shows that they indeed satisfy the Virasoro brackets. 

\begin{corollary}\label{cor: virbracket}
    The operators $L_n$ defined by \eqref{eq: virlva1} and \eqref{eq: virlva2} satisfy the Virasoro Lie bracket
    \[[L_n, L_m]=(n-m)L_{n+m}\quad\textup{for }n,m\geq -1\,.\]
\end{corollary}
\begin{proof}
    In the non-degenerate case, this follows from the fact that $L_n=\omega_{(n+1)}$ where $\omega$ is a conformal element. By Proposition \ref{prop: functorial}, it suffices to find an embedding of a lattice into a non-degenerate one which can always be done as follows. Let $C$ be any non-degenerate pairing on $\Lambda$ and consider the pairing
    \[\widetilde B=\begin{bmatrix}
        B & C\\
        C^t & 0
    \end{bmatrix}\]
    on $\Lambda\oplus \Lambda$. Then $(\widetilde \Lambda, \widetilde B)$ is a non-degenerate symmetric lattice, and the inclusion $\Lambda\hookrightarrow \Lambda\oplus \Lambda$ onto the first component is an inclusion of lattices.
\end{proof}
\begin{remark}
    In \cite{blm} the authors worked with the pair vertex algebra to ensure that there is an isomorphism with a non-degenerate lattice vertex algebra (see Lemma 4.6 in loc. cit.), and hence a conformal element. The construction of a non-degenerate lattice in the proof of Corollary \ref{cor: virbracket} resembles the construction of the pair vertex algebra from the sheaf vertex algebra.
\end{remark}
The notion of primary states, usually defined for vertex algebras admiting a conformal element, extends to vertex algebras of degenerate lattices as well.
\begin{definition}\label{def: primarystates}
    Let $V_\Lambda$ be a (possibly degenerate) lattice vertex algebra. Define the spaces of primary states 
    \begin{align*}P_i&=\{v\in V_\Lambda\colon L_0(v)=iv\textup{ and }L_n(v)=0\textup{ for }n>0\}\subseteq V_\Lambda\\
    \widecheck P_0&=P_1/T(P_0)\subseteq \widecheck V_\Lambda\,.
    \end{align*}
\end{definition}
\begin{corollary}\label{cor: liesubalgebra}
The space of primary states $\widecheck P_0\subseteq \widecheck V_\Lambda$ is a Lie subalgebra.
\end{corollary}
\begin{proof}
    This follows again by embedding $\Lambda$ into a non-degenerate lattice and applying the well-known result in the non-degenerate case \cite{Borcherds} (see also \cite[Proposition~3.11]{blm}).\qedhere
\end{proof}

\begin{corollary}\label{cor: primarywt0}
Let $a\in e^\alpha\otimes \BD_\Lambda\subseteq V_\Lambda$ with $\alpha\neq 0$. Then $\overline a\in \widecheck P_0$ if and only if
\[\sum_{n\geq -1} \frac{(-1)^{n}}{(n+1)!}T^{n+1}\circ L_n(a)=0\,.\]
\end{corollary}
\begin{proof}
 The non-degenerate case is \cite[Proposition 3.16]{blm}. Once again we can deduce the general case from the non-degenerate one by embedding $\Lambda$.\qedhere
\end{proof}

\subsection{Joyce's vertex algebras}\label{sec: Joyce's VA}

Joyce defined in \cite{Jo17, grossjoycetanaka, Jo21} a vertex algebra framework to study wall-crossing for moduli spaces of stable objects in certain abelian or triangulated categories. In this section, we explain his construction of vertex algebras associated to certain dg categories.

Let $\bfT$ be a saturated (= smooth, proper, triangulated) dg category over $\BC$ in the sense of \cite{TV07}. Examples include dg enhancements of the triangulated categories $D^b(\dgQ)$ and $D^b(X)$ where $\dgQ$ is a finite acyclic dg quiver and $X$ is a smooth projetive variety. In loc. cit., the authors construct a derived stack $\bfCN^{\,\bfT}$ whose $\BC$-points up to homotopy parametrize objects in the triangulated category $[\bfT]$ associated to $\bfT$. This is equipped with a canonical perfect complex ${\bf Ext_\bfT}$ on $\bfCN^{\,\bfT}\times \bfCN^{\,\bfT}$ whose fiber at $(E,F)$ is given by the complex $\Hom^\bullet_{\bfT}(E,F)$. Let $\CN^{\,\bfT}\coloneqq t_0(\bfCN^{\,\bfT})$ be the classical truncation which is a higher stack and $\Ext_\bfT$ be the restriction of ${\bf Ext_\bfT}$ to the truncation. The stack decomposes into connected component 
$$\CN^{\,\bfT}=\coprod_{\alpha}\CN_\alpha
$$
and rank of the $\Ext_{\bfT}$ complex on $\CN_\alpha\times \CN_\beta$ is given by the pairing $\chi_{\bfT}(\alpha,\beta)$. 

Joyce constructs a natural vertex algebra structure on the homology group of the (higher) moduli stack $\CN^{\,\bfT}$. The definition requires the following geometric data naturally coming from the dg category $\bfT$:

\begin{enumerate}
    \item A zero map $0:\pt\rightarrow \CN^{\,\bfT}$.
    \item A direct sum map $\Sigma\colon \CN^{\,\bfT}\times \CN^{\,\bfT}\to \CN^{\,\bfT}$. 
    \item A $B\BG_m$ action $\rho\colon B\BG_m\times \CN^{\,\bfT}\to \CN^{\,\bfT}$. 
    \item A $K$-theory class $\Theta$ on $\CN^{\,\bfT}\times \CN^{\,\bfT}$ defined by a symmetrization
    \[\Theta=\sigma^\ast \Ext_{\bfT}+\Ext^\vee_{\bfT}\]
    where $\sigma\colon \CN^{\,\bfT}\times \CN^{\,\bfT}\to \CN^{\,\bfT}\times \CN^{\,\bfT}$ is the permutation map.
\end{enumerate}
Note that rank of $\Theta$ on $\CN_\alpha\times \CN_\beta$ is given by the symmetrized pairing $$\chi_{\bfT}^{\sym}(\alpha,\beta)\coloneqq\chi_\bfT(\alpha,\beta)+\chi_{\bfT}(\beta,\alpha)\,.$$ 

The vertex algebra structure
\[V^{\bfT}=\Big(H_\ast(\CN^{\,\bfT})\,,\, \ket{0}\,,\, T\,,\, Y\Big)\] 
is defined as follows from this data:
\begin{enumerate}
    \item The underlying vector space is $H_\ast(\CN^{\,\bfT})$.
    \item The vacuum vector corresponds to the trivial representation, i.e., 
    \[\ket{0}=0_\ast[\pt]\in H_0(\CN^{\,\bfT})\,.\]
    \item The translation operator is defined by
    \[T(u)=\rho_\ast\big(t\boxtimes u\big)\]
    where $t$ is the generator of $H_2(B\BG_m)$.
    \item The state-field correspondence is
    \[Y(u, z)v= (-1)^{\chi_{\bfT}(\alpha,\beta)}z^{\chi_{\bfT}^{\sym}(\alpha,\beta)}\Sigma_\ast\Big[e^{zT}\boxtimes \text{id}\big(c_{z^{-1}}(\Theta)\cap (u\boxtimes v)\big)\Big] \]
for any $u\in H_\ast(\CN_\alpha)$ and $v\in H_\ast(\CN_\beta)$.
\end{enumerate}

\begin{remark}
    In \cite{Jo17, grossjoycetanaka, Jo21}, the abelian category version of the above construction is also introduced. There is often a vertex algebra homomorphism from the abelian category version to the dg category version. So every wall-crossing formulas proven in the abelian category version hold also in the dg category version. We work with dg category version because it often admits an explicit description in terms of the lattice vertex algebras as we address in the next section. 
\end{remark}

\begin{remark}\label{rmk: naturalva}
In Joyce's setup, the $K$-theory class $\Theta$ is often an extra choice. By defining $\Theta$ as the symmetrization of the Ext complex, the construction becomes canonically associated to the dg category $\bfT$. In particular, quasi-equivalent dg categories produce isomorphic vertex algebras.
\end{remark}

\subsection{Comparison of two vertex algebras}
\label{subsec: comparisonva}
In Section \ref{sec: lattice VA} and Section \ref{sec: Joyce's VA}, we explained two constructions of vertex algebras, one from a lattice and the other from a moduli stack. Note that a saturated dg category $\bfT$ is a common source of a lattice and a moduli stack. In this section, we show that the two constructions are naturally isomorphic under the assumptions below.

\begin{assumption}\label{ass: dg category}
    Let $\bfT$ be a saturated dg category.
    \begin{enumerate}
    \item The natural map 
    $K(\bfT)\rightarrow \pi_0(\CN^{\,\bfT})$ is an isomorphism.
    \item The realization homomorphism \eqref{eq: assumption iso} is an isomorphism for all $\alpha\in K(\bfT)$. 
    \item The class in $K$-theory of the diagonal bimodule $\Hom_{\bfT}^\bullet(-, -)$, regarded as a $\bfT\otimes \bfT^{\textup{op}}$-module (cf. \cite[Definition 2.4(4)]{TV07}), is in the image of the map
    \[K(\bfT)\otimes K(\bfT^{\textup{op}})\to K(\bfT\otimes \bfT^{\textup{op}})\to K(\widehat{\bfT\otimes \bfT^{\textup{op}}})\,.\]
\end{enumerate}
\end{assumption}
\begin{remark}Assumption \ref{ass: dg category} (3) has several implications on the structure of $K(\bfT)$, which we now explain. Concretely, the assumption means that the diagonal bimodule $\Hom_\bfT^\bullet(- ,-)$ can be written as successive extensions of finitely many bimodules of the form $\Hom^\bullet_\bfT(-, G_i^L)\otimes \Hom^\bullet_\bfT(G_i^R, -)$ indexed by $i\in I$. If we let $\alpha_i^L, \alpha_i^R\in K(\bfT)$ be the classes of $G_i^L$ and $G_i^R$, then
\[\Delta^K=\sum_{i\in I}\alpha_i^L\otimes \alpha_i^R\in K(\bfT)\otimes K(\bfT)\simeq K(\bfT)\otimes K(\bfT^{\textup{op}})\,\]
maps to the class of the diagonal bimodule. 

Let $\alpha\in K(\bfT)$ and choose its representative $F\in\bfT$. Using the construction from the previous paragraph, the $\bfT^\textup{op}$-module $\Hom^\bullet_\bfT(-, F)$ can be obtained by successive extensions of modules $\Hom^\bullet_\bfT(-, G_i^L)\otimes \Hom^\bullet_\bfT(G_i^R, F)$. Since $\bfT$ is assumed to be triangulated, meaning that the Yoneda embedding is a quasi-equivalence and in particular $K(\widehat \bfT_{\textup{pe}})\simeq K(\bfT)$, this implies an equality 
\[\alpha=\sum_{i\in I}\alpha_i^L\cdot \chi_\bfT(\alpha_i^R, \alpha)\in K(\bfT)\quad\textup{for any }\alpha\in K(\bfT).\]
Thus, it follows from Assumption \ref{ass: dg category} (3) that $K(\bfT)$ is a finitely generated free abelian group, that $\{\alpha_i^L\}_{i\in I}$ generates $K(\bfT)$ over $\BZ$, that $\chi_\bfT$ is a perfect pairing and that $\Delta^K$ is the $K$-theoretic diagonal with respect to $\chi_{\bfT}$ (recall the definition of diagonal before Example~\ref{ex: K theoretic diagonal}).
\end{remark}

\begin{example}
    In Section \ref{sec: assumptiondgquivers}, we will see that all these properties are satisfied for the dg category associated to a finite acyclic dg quiver $\bfQ$. This is also satisfied for the dg enhancement of $D^b(X)$ where $X$ is a smooth projective variety admitting a full exceptional collection. This follows from the fact that such $D^b(X)$ is equivalent to the bounded derived category of a finite acyclic dg quiver \cite{Bod}. 
\end{example}

\begin{theorem}\label{thm: natural iso}
    Let $\bfT$ be a saturated dg category satisfying Assumption \ref{ass: dg category}. Then there is a natural isomorphism of vertex algebras
    $$V^\bfT\rightarrow \textup{VA}(K(\bfT),\chi_{\bfT}^{\sym}). 
$$
\end{theorem}

\begin{remark}\label{rem: natural} 
We explain the naturality in Theorem \ref{thm: natural iso}. Let $\bf f:\bfT_1\rightarrow \bfT_2$ be a quasi-functor between dg categories satisfying Assumption \ref{ass: dg category} such that\footnote{Many examples of quasi-functors satisfying (i) and (ii) can be constructed from admissible subcategories  with the induced dg enhancements, for instance Kuznetsov components.}
\begin{enumerate}
    \item [(i)] the induced functor $f:[\bfT_1]\rightarrow [\bfT_2]$ is fully faithful,
    \item [(ii)] and there exists a quasi-functor $\bf g: \bfT_2\rightarrow \bfT_1$ such that $g:[\bfT_2]\rightarrow [\bfT_1]$ is the left adjoint of $f$. 
\end{enumerate}
Since $f$ is fully faithful, $K(\bfT_1)\rightarrow K(\bfT_2)$ is a lattice embedding, hence defining a vertex algebra embedding $f_*:\textup{VA}(K(\bfT_1),\chi^\sym_{\bfT_1})\rightarrow \textup{VA}(K(\bfT_2),\chi^\sym_{\bfT_2})$. On the other hand, we have a morphism (in the homotopy category) between the higher moduli stacks $\CN^{\,\bfT_1}\rightarrow \CN^{\,\bfT_2}$ by applying the contravariance of the construction $\bfT\mapsto \bfCN^{\,\bfT}$ in \cite{TV07} to the quasi-functor $\bf g$. Note that $\CN^{\,\bfT_1}\rightarrow \CN^{\,\bfT_2}$ sends its $\BC$-point $F\in [\bfT_1]$, regarded as $\Hom^\bullet_{\bfT_1}(-,F)$ via the Yoneda embedding, to $f(F)\in [\bfT_2]$ by the adjunction $\Hom_{\bfT_1}^\bullet({\bf g}(-),F)\simeq \Hom_{\bfT_2}^\bullet(-,{\bf f}(F))$. This morphism $\CN^{\,\bfT_1}\rightarrow \CN^{\,\bfT_2}$ preserves the zero map, direct sum map and $B\BG_m$-action. Most importantly, fully faithfulness of $f$ implies that $\Ext_{\bfT_2}$ pulls back to $\Ext_{\bfT_1}$, hence defining a vertex algebra homomorphism $f_*:V^{\bfT_1}\rightarrow V^{\bfT_2}$. Having this understood, naturality in Theorem \ref{thm: natural iso} means that the diagram below commutes:
    \begin{center}
    \begin{tikzcd}
V^{\bfT_1} \arrow[r, "f_*"] \arrow[d, "\sim" {rotate=90, anchor=north}] & V^{\bfT_2} \arrow[d, "\sim" {rotate=90, anchor=north}] \\
\textup{VA}(K(\bfT_1),\chi^\sym_{\bfT_1}) \arrow[r, "f_*"]                  & \textup{VA}(K(\bfT_2),\chi^\sym_{\bfT_2}).                 
\end{tikzcd}
    \end{center}
\end{remark}

\begin{corollary}\label{cor: intertwining Virasoro}
    Let ${\bf f}:\bfT_1\rightarrow \bfT_2$ be a quasi-functor as in Remark \ref{rem: natural}. Then the induced vertex algebra homomorphism 
    $$f_*:V^{\bfT_1}\rightarrow V^{\bfT_2}
$$
intertwines the Virasoro operators $L_n$ for all $n\geq -1$. 
\end{corollary}
\begin{proof}
    By naturality of the vertex algebra isomorphism in Theorem \ref{thm: natural iso}, the statement follows from that of lattice vertex algebras as in Proposition \ref{prop: functorial}. 
\end{proof}

The proof of Theorem \ref{thm: natural iso} will be given in the next two subsections. 

\subsubsection{Ext descendent}\label{sec: Ext descendent}

Let $\bfT$ be a saturated dg category satisfying Assumption \ref{ass: dg category} (1) throughout this section. The moduli stack $\CN^{\,\bfT}$ carries natural cohomology classes, called tautological classes. The descendent algebra and realization homomorphism, such as in Section \ref{sec: descendent algebra} or in \cite{blm}, keep track of tautological classes on the moduli stack $\CN^{\,\bfT}$. In this section, we introduce a new construction of descendent algebra and realization homomorphism that uses the canonical perfect complex $\Ext_\bfT$. This unifies the treatment of descendent algebra and Virasoro operators in different contexts, such as smooth projective varieties and dg quivers, using a purely dg category language.

\begin{remark}
Let $\bfT$ be a saturated dg category. By \cite[Definition 2.4, Corollary 2.13]{TV07}, the associated triangulated category $[\bfT]$ is equivalent to a perfect derived category $D^b(B)$ of a smooth proper dg algebra $B$. In \cite[Theorem 4.2]{Shk}, it is proven that $D^b(B)$ is equipped with a Serre functor. Moreover, such a Serre functor lifts to the dg enhancement by its explicit description in loc. cit. Therefore, $[\bfT]$ is equipped with a Serre functor $S_{\bfT}$ and it admits a quasi-functor (=morhpism in the homotopy category of dg categories) lift for which we use the same notation. We also denote by $S_\bfT$ the induced map on $K(\bfT)$. 
\end{remark}

\begin{definition}\label{def: ext descendent}
    The Ext descendent algebra $\BD^{\bfT}$ is the unital commutative $\BQ$-algebra generated by symbols\footnote{The superscript $L$ and $R$ stands for left and right which will be clear from the definition of the Ext realization homomorphism. }
    $$\ch^L_k(\alpha),\ \ch^R_k(\alpha)\quad\textnormal{for each}\quad k\in \BN,\ \alpha\in K(\bfT)
    $$
    with linearity relations 
    $$
    \begin{cases}
        \ch^L_k(\lambda_1\alpha_1+\lambda_2\alpha_2)=\lambda_1\ch_k^L(\alpha_1)+\lambda_2\ch_k^L(\alpha_2)\\
        \ch^R_k(\lambda_1\alpha_1+\lambda_2\alpha_2)=\lambda_1\ch_k^R(\alpha_1)+\lambda_2\ch_k^R(\alpha_2)
\end{cases} \lambda_1,\lambda_2\in \BQ,\ \alpha_1,\alpha_2\in K(\bfT)
    $$
    and Serre duality relations 
    $$\ch^L_k(\alpha)=(-1)^k\ch_k^R(S_\bfT(\alpha))\,.
    $$
\end{definition}

Recall that we have a canonical perfect complex $\Ext_\bfT$ over $\CN^{\,\bfT}\times \CN^{\,\bfT}$. For any $F\in \bfT$, we can consider a restriction $\Ext_\bfT\big|_{\CN^{\,\bfT}\times \{F\}}\eqqcolon\Ext_\bfT(-,F)$ as a perfect complex over $\CN^{\,\bfT}$; this construction is essentially the counit $\bfT\to L_{\textup{pe}}(\CN^\bfT)$ of the adjuction \cite[Proposition 3.4]{TV07}. This construction descends to a homomorphism between $K$-theories
$$K(\bfT)\rightarrow K^0(\CN^{\,\bfT}),\quad \alpha\mapsto \Ext_\bfT(-,\alpha)\,.
$$
We extend this homomorphism by tensoring $-\otimes_\BZ\BQ$ and denote the image of $\alpha\in K(\bfT)_\BQ$ again by $\Ext_\bfT(-,\alpha)$. Similarly, one can define $\Ext_\bfT(\alpha,-)$ using the restriction $\Ext_\bfT\big|_{\CN^{\,\bfT}\times \{F\}}$ in the other way; this can again be thought as the counit of the adjuction, but now applied to $\bfT^{\textup{op}}$. 

\begin{definition}
    The Ext realization homomorphism is defined by 
    $$\xi:\BD^{\bfT}\rightarrow H^*(\CN^{\,\bfT}),\quad \ch^L_k(\alpha)\mapsto \ch_k(\Ext_\bfT(\alpha,-)),\ \ \ch^R_k(\alpha)\mapsto \ch_k(\Ext_\bfT(-,\alpha))\,.
    $$
\end{definition}

\begin{remark}
    For the above definition to be well-defined, we need to show that the assignment preserves the relations. Linearity relations are trivial and Serre duality relation follows from the defining property of the Serre functor\footnote{Here we use a quasi-functor lift of the Serre functor. }
$$\Ext_\bfT(\alpha,-)\simeq \Ext_\bfT(-,S_\bfT(\alpha))^\vee\,. 
$$
Since the Serre duality relations express the left symbols in terms of the right symbols and via versa, we could have kept only the half of the generators. It is however natural to keep both sets of the generators in some cases such as in the definition of the Ext Virasoro operators. 
\end{remark}

Consider a connected component $\CN_\alpha$ corresponding to $\alpha\in K(\bfT)$. Let
$$\BD^{\bfT}_\alpha\coloneqq\BD^{\bfT}/\langle\ch_0^L(\beta)-\chi_\bfT(\beta,\alpha)\cdot 1,\ \ch_0^R(\beta)-\chi_\bfT(\alpha,\beta)\cdot 1\,|\,\beta\in K(\bfT)\rangle\,.
$$
Since $\ch_0(-)$ computes the rank of a perfect complex and $\Ext$ has rank $\chi_\bfT(\alpha,\beta)$ over $\CN_\alpha\times \CN_\beta$, the realization homomorphism factors through the quotient
\begin{equation}\label{eq: assumption iso}
    \xi_\alpha:\BD^{\bfT}_\alpha\rightarrow H^*(\CN_\alpha)\,.
\end{equation}

In many geometric examples, the homomorphism \eqref{eq: assumption iso} is an isomorphism, see \cite{Gr19, blm} for the case of smooth projective varieties of type D.\footnote{In \cite{Gr19} and \cite{blm}, the isomorphism is proven in cohomological setting. The cohomological and the Ext descendent algebras are identified by setting $\ch^L_k(\alpha)=\ch^H_k(\ch^\vee(\alpha)\cdot \td(X))$.} We prove the isomorphism in the setting of dg quivers in Proposition \ref{prop: realization iso}.

\subsubsection{Construction of vertex algebra isomorphism}

In this section, we construct a natural isomorphism from Joyce's vertex algebra to the lattice vertex algebra. On the level of vector space, the map is defined by constructing the diagram below:
\begin{equation}\label{eq: construction diagram}
    \begin{tikzcd}
  H^*(\CN_\alpha)   \arrow[r,phantom,"\otimes" description]       &[-2em] H_*(\CN_\alpha) \arrow[rr,"{\langle-,-\rangle}"] \arrow[d, "{\xi_\alpha^\dagger}"] & \ & \BQ \arrow[d, equal] \\
 \BD^{\bfT}_\alpha  \arrow[u,"{\xi_\alpha}"]\arrow[r,phantom,"\otimes" description]    &[-2em]   e^\alpha\otimes \BD_{K(\bfT)} \arrow[rr,"{\widetilde\chi_\bfT(-,-)}"]   &    \   & \BQ\,.                      
\end{tikzcd}
\end{equation}
The first row is the topological pairing which is (graded) perfect in the sense that it induces an isomorphism between cohomology and graded dual of the homology. By Assumption \ref{ass: dg category} (2), $\xi_\alpha$ is an isomorphism. If we define the perfect pairing in the second row, then we obtain a dual isomorphism $\xi_\alpha^\dagger$. The perfect pairing $\widetilde\chi_\bfT(-,-)$ in the second row is induced from the perfect pairing $\chi_\bfT(-,-)$ on $K(\bfT)$ as we describe next. 

Recall that $\BD^{\bfT}_\alpha$ is the symmetric algebra generated by symbols $\ch_n^L(\beta)$ with $n\geq 1$ and $\beta\in K(\bfT)$ which is $\BQ$-linear with respect to $\beta$.\footnote{We use here the Serre duality relations to disregard all the right symbols $\ch^R_k(\alpha)$. } Similarly, $\BD_{K(\bfT)}$ is the symmetric algebra generated by symbols $\beta_{-n}$ with $n\geq 1$ and $\beta\in K(\bfT)$ which is $\BQ$-linear with respect to $\beta$. We define $\widetilde\chi_\bfT(-,-)$ on the generators
$$\widetilde\chi_\bfT(\ch^L_n(\beta),\beta'_{-n'})
\coloneqq\frac{\delta_{n,n'}}{(n-1)!}\,\chi_\bfT(\beta,\beta'),\quad n,n'\geq 1,\ \beta,\beta'\in K(\bfT)
$$
and extend it using the symmetric algebra structures as in \cite[Section 2.1]{blm}.

By taking the direct sum of the isomorphism $\xi_\alpha^\dagger$ on each connected component, we obtain an isomorphism 
\begin{equation}\label{eq: vector space iso}
    H_*(\CN^{\,\bfT})=\bigoplus_{\alpha\in K(\bfT)}H_*(\CN_\alpha)\xrightarrow{\oplus\, \xi_\alpha^\dagger} \bigoplus_{\alpha\in K(\bfT)}e^\alpha \otimes \BD_{K(\bfT)}= V_{K(\bfT)}\,. 
\end{equation}
between the underlying vector spaces of the two vertex algebras. 
\begin{proposition}\label{prop: vertex iso}
  Let $\bfT$ be a saturated dg category satisfying Assumption \ref{ass: dg category}. Then the linear isomorphism \eqref{eq: vector space iso} is a vertex algebra homomorphism. 
\end{proposition}
\begin{proof}
We have a commutative diagram as follows:

\begin{center}
    \begin{tikzcd}
        K(\bfT)\otimes K(\bfT)\arrow[r, equal]&K(\bfT)\otimes K(\bfT^{\textup{op}})\arrow[r]\arrow[d]&K^0(\CN^{\,\bfT})\otimes K^0(\CN^{\,\bfT})\arrow[d]\\
         &K(\widehat{\bfT\otimes \bfT^{\textup{op}}})\arrow[r] & K^0(\CN^{\,\bfT}\times \CN^{\,\bfT})\,.
    \end{tikzcd}
\end{center}
The map on the bottom can be obtained from the counit of the adjunction \cite[Proposition 3.4]{TV07} applied to $(\widehat{\bfT\otimes \bfT^{\textup{op}}})_{\textup{pe}}$; note that the stack associated to this dg category is $\CN^{\,\bfT}\times \CN^{\,\bfT}$ by \cite[Lemma 3.3]{TV07} and the fact that $\CN^{(-)}$ preserves limits.

The class of $\bfT$ in $K(\widehat{\bfT\otimes \bfT^{\textup{op}}})$ is sent to $\Ext_\bfT$ via the bottom map. On the other hand, the image of $\alpha^L\otimes \alpha^R\in  K(\bfT)\otimes K(\bfT)$ via the map on top is $\Ext_\bfT(-, \alpha^L)\otimes \Ext_\bfT(\alpha^R, -)$. Thus, by  Assumption \ref{ass: dg category} (3) there is
\[\Delta^K=\sum_{i\in I}\alpha_i^L\otimes \alpha_i^R\in K(\bfT)\otimes K(\bfT)\,\]
for which we have the following equality in the $K$-theory of $\CN^{\,\bfT}\times \CN^{\,\bfT}$:
\begin{equation}\label{eq: Ktheorysplitting}
\Ext_\bfT=\sum_{i\in I} \Ext_\bfT(-,\alpha_i^L)\boxtimes\Ext_\bfT(\alpha_i^R,-)\,.
\end{equation}

Applying the Chern character to \eqref{eq: Ktheorysplitting}, we can write in the language of the Ext descendent algebra
$$\ch_k(\Ext)=\sum_{a+b=k}\sum_{i\in I} \ch^R_a(\alpha_i^L)\boxtimes \ch_b^L(\alpha_i^R)\,.
$$
This is the analogue of equation (45) in \cite{blm}. The rest of the argument is a straightforward adaptation of the proof in loc. cit.
\end{proof}

\begin{remark}
    We remark that it is proven along the way that the operators
    $$\ch_n^L(\alpha)\cap -: H_\ast(\CN^{\,\bfT})\rightarrow H_\ast(\CN^{\,\bfT})
    $$
    satisfy the following bracket with the creation/annihilation operators: 
    \begin{equation}\label{eq: pldagger}
        [n!\ch_n^L(\alpha)\cap-,\beta_{(-m)}]=n\cdot\delta_{n,m}\cdot\chi_\bfT(\alpha,\beta)\,.
    \end{equation}
\end{remark}

\begin{proof}[Proof of Theorem \ref{thm: natural iso}]
By Proposition \ref{prop: vertex iso}, it suffices to prove that the vector space isomorphism \eqref{eq: vector space iso} is natural in the sense of Remark \ref{rem: natural}. Since the isomorphism is constructed by the realization homomorphism $\xi_\alpha$ and a perfect pairing $\widetilde\chi_\bfT(-,-)$ in the diagram \eqref{eq: construction diagram}, it suffices to show the naturality of $\xi_\alpha$ and $\widetilde\chi_\bfT(-,-)$.

Let ${\bf f}:\bfT_1\rightarrow \bfT_2$ be a quasi-functor with the properties described in Remark \ref{rem: natural}. Naturality of the pairing $\widetilde\chi_\bfT(-,-)$ directly follows from the fact that it is induced from $\chi_{\bfT}$ and we have a lattice embedding $K(\bfT_1)\rightarrow K(\bfT_2)$. To show the naturality of $\xi_\alpha$, let $\alpha_1\in K(\bfT_1)$ and $\alpha_2\coloneqq f(\alpha_1)\in K(\bfT_2)$. Denote the induced morphism between the moduli stacks by $f:\CN^{\,\bfT_1}_{\alpha_1}\rightarrow\CN^{\,\bfT_2}_{\alpha_2}$. Note that the construction of the Ext descendent algebra is a contravariant functor in the sense that we have
$$f^*:\BD^{\bfT_2}_{\alpha_2}\rightarrow \BD^{\bfT_1}_{\alpha_1},\quad \ch^L_k(\beta)\mapsto \ch^L_k(g(\beta))
$$
where $g$ is the left adjoint of $f$. Naturality of the realization homomorphism then refers to the commutativity of the diagram 
    \begin{center}
\begin{tikzcd}
H^*(\CN^{\,\bfT_1}_{\alpha_1})                 & H^*(\CN^{\,\bfT_2}_{\alpha_2}) \arrow[l, "f^*"']                  \\
\BD^{\bfT_1}_{\alpha_1} \arrow[u, "\xi_{\alpha_1}"] & \BD^{\bfT_2}_{\alpha_2} \arrow[u, "\xi_{\alpha_2}"'] \arrow[l, "f^*"']
\end{tikzcd}
    \end{center}
which follows from the adjunction quasi-isomorphism
$$f^*\Ext_{\bfT_2}(\beta,-)\simeq \Ext_{\bfT_1}(g(\beta),-)\,.\qedhere
$$\end{proof}

\subsubsection{Virasoro operators}
We finish the section by introducing Virasoro operators on the Ext descendent algebra and comparing them to the operators acting on the lattice vertex algebra. 

\begin{definition}\label{def: K Virasoro}
    For each $n\geq -1$, we define the Ext Virasoro operator $\bL_n$ acting on $\BD^\bfT$ as a summation $\bL_n=\bR_n+\bT_n$ of a derivation operator $\bR_n$ such that
    \begin{align*}
    \bR_n(\ch^L_k(\alpha))&=k\cdot(k+1)\cdots(k+n)\,\ch_{k+n}^L(\alpha)\,,\\
    \bR_n(\ch^R_k(\alpha))&=k\cdot(k+1)\cdots(k+n)\,\ch_{k+n}^R(\alpha)\,,
    \end{align*}
    and a multiplication operator 
$$\bT_n\coloneqq\sum_{a+b=n}a!b!(-1)^a\ch^R_a\ch^L_b(\Delta^K)=\sum_{a+b=n}\sum_{i\in I}a!b!(-1)^a \sum_{i\in I}\ch^R_a(\alpha_i^L)\ch_b^L(\alpha_i^R)\,.
$$
\end{definition}

\begin{proposition}\label{prop: duality of operators}
    Virasoro operators $\bL_n$ and $L_n$ acting on $\BD^\bfT_\alpha$ and $e^\alpha\otimes \BD_{K(\bfT)}$, respectively, are dual to each other with respect to the pairing $\widetilde\chi_\bfT(-,-)$ in \eqref{eq: construction diagram}. 
\end{proposition}
This is the dg category version of \cite[Theorem 4.12]{blm}. We will give here a new proof that is cleaner and does not assume the existence of a conformal element. Given an operator $P$ we will denote by $P^\dagger$ its dual with respect to the perfect pairing $\widetilde\chi_\bfT(-,-)$.

\begin{proof}
    We first note that when $n=-1$ the claim can be proven exactly in the same way as \cite[Lemma 4.9]{blm}. We focus on the case $n>0$; note that the statement for $n=0$ is easy to check directly, but it also follows from $n=-1$ and $n=1$ together with the Virasoro bracket.
    
    We need to show that the operators $\bL_n^\dagger$, $n>0$ satisfy \eqref{eq: virlva1} and \eqref{eq: virlva2}.
    \begin{claim}
        We have $\bL_n^\dagger(e^\alpha)=0$ for $n>0$, i.e. equation \eqref{eq: virlva1} holds.
    \end{claim}
    \begin{proof}
     The operator $\bL_n$ increases cohomological degree by $2n$, so $\bL^\dagger_n$ decreases homological degree by $2n$.\qedhere
    \end{proof}
We denote by $p_l(\beta)\colon \BD^{\bfT}_\alpha\to \BD^{\bfT}_\alpha$ the operator of multiplication by $l!\ch_l^L(\beta)$. Its dual $p_l^\dagger(\beta)$ is the operator of capping with $l!\ch_l^L(\beta)$ and its bracket with the creation/annihilation operators is given by \eqref{eq: pldagger}.
    \begin{claim}\label{claim: 2}
        Let $\beta\in K(\bfT)$. We have 
        \[[\bR_n^\dagger, \beta_{(-k)}]=\begin{cases}
            k\cdot \beta_{(-k+n)}&\textup{ if }k>n> 0\\
            0 &\textup{ if  }n\geq k>0\,.
        \end{cases}\]
    \end{claim}
    \begin{proof}
We prove the claim by showing that the duals of both sides agree when applied to $1\in \BD^{\bfT}_\alpha$ and have the same commutator with the operators $p_l(\gamma)$ of multiplication by descendents. Indeed we have $\bR_n(1)=\beta_{(-k)}^\dagger (1)=0$ and also $\beta_{(-k+n)}^\dagger (1)=0$ when $k>n$; this follows from the fact that $\beta_{(-k)}^\dagger$ has cohomological degree $-2k<0$. Hence the duals of both sides annihilate $1$.

We now use \eqref{eq: pldagger} to calculate 
\begin{align*}[[\bR_n^\dagger, \beta_{(-k)}]^\dagger, p_{l}(\gamma)]&=[[\beta_{(-k)}^\dagger, \bR_n],p_l(\gamma)]=-[[\beta_{(-k)}^\dagger, p_l(\gamma)], \bR_n]+[\beta_{(-k)}^\dagger, [\bR_n, p_l(\gamma)]]\\
&=[\beta_{(-k)}^\dagger, l \cdot p_{l+n}(\gamma)]=l\cdot k\cdot\delta_{l+n, k}\cdot\chi_{\bfT}(\beta,\gamma)\cdot\id\,.
\end{align*}
If $n\geq k$ then $\delta_{l+n, k}=0$ for every $l> 0$. If $n<k$ then this agrees with
\[[k\cdot \beta_{(-k+n)}^\dagger, p_l(\gamma)]=k[p_l(\gamma)^\dagger, \beta_{(-k+n)}]^\dagger=l\cdot k\cdot\delta_{l+n, k}\cdot\chi_{\bfT}(\beta,\gamma)\cdot\id\,.\qedhere\]
    \end{proof}

To deal with $T_n$ we prove the following auxiliary claim.

    \begin{claim}\label{claim: 3}
For any $\beta\in K(\bfT)$ and $l\geq 0$ we have
\[p_l^\dagger(\beta)+\sum_{i\in I}\chi_\bfT(\alpha_i^R,\beta)\,p_l^\dagger(S_{\bfT}^{-1}(\alpha_i^L))=\beta_{(l)}\,.\]
    \end{claim}
    \begin{proof}
If $l>0$ then both sides annihilate $e^\alpha$. If $l=0$ then both sides applied to $e^\alpha$ give $\chi^\sym_\bfT(\beta, \alpha)\cdot e^\alpha$ since it is so defined on the right hand side and for the left hand side we have
\begin{align*}
    \chi_\bfT(\beta,\alpha)+\sum_{i\in I}\chi_\bfT(\alpha_i^R,\beta)\chi(S_{\bfT}^{-1}(\alpha_i^L),\alpha)
    &=\chi_\bfT(\beta,\alpha)+\sum_{i\in I}\chi_{\bfT}(\alpha,\alpha_i^L)\chi_\bfT(\alpha_i^R,\beta)\\
    &=\chi^\sym_{\bfT}(\beta,\alpha)\,.
\end{align*}
Hence, it is enough to show that both sides have the same commutator with creation operators $\gamma_{(-s)}$ for $s>0$ and $\gamma\in K(\bfT)$. Indeed,
\begin{align*}
&\Big[ p_l^\dagger(\beta)+\sum_{i\in I}\chi_\bfT(\alpha_i^R,\beta)\,p_l^\dagger(S_{\bfT}^{-1}(\alpha_i^L)), \gamma_{(-s)}\Big]\\
=\,&\left(l\cdot\delta_{l,s}\cdot \chi_{\bfT}(\beta,\gamma)+l\cdot\delta_{l,s}\sum_{i\in I}\chi_\bfT(\alpha_i^R,\beta)\chi_{\bfT}(S_{\bfT}^{-1}(\alpha_i^L),\gamma)\right)\cdot \id\\
=\,&l\cdot\delta_{l,s}\cdot \chi^\sym_{\bfT}(\gamma,\beta)\cdot \id=[\beta_{(l)},\gamma_{(-s)}]\,.\qedhere
\end{align*}
    \end{proof}
    
    \begin{claim}\label{claim: 4}
        Let $\beta\in K(\bfT)$. We have 
        \[[\bT_n^\dagger, \beta_{(-k)}]=\begin{cases}
            k\cdot \beta_{(n-k)}&\textup{ if }n\geq k> 0\\
            0 &\textup{ if  }k> n>0\,.
        \end{cases}\]
    \end{claim}
    \begin{proof}
Recall that 
$$\bT_n=\sum_{a+b=n}\sum_{i\in I}a!\ch_a^L(S^{-1}_{\bfT}(\alpha_i^L))\cdot b!\ch_b^L(\alpha_i^R)\,.
$$
Therefore, we compute the bracket as 
\begin{align*}
    [\bT_n^\dagger, \beta_{(-k)}]
    &=\sum_{a+b=n}\sum_{i\in I}\Big[p_b^\dagger(\alpha_i^R)p_a^\dagger(S^{-1}_{\bfT}(\alpha_i^L)),\beta_{(-k)}\Big]\\
    &=\sum_{a+b=n}\sum_{i\in I}\Big(
    \Big[p_b^\dagger(\alpha_i^R),\beta_{(-k)}\Big]p_a^\dagger(S^{-1}_{\bfT}(\alpha_i^L))+
    p_b^\dagger(\alpha_i^R)\Big[p_a^\dagger(S^{-1}_{\bfT}(\alpha_i^L)),\beta_{(-k)}\Big]\Big)\,.
\end{align*}
If $k>n\geq a,b$ then every commutator in the last line vanishes and we have $[\bT_n^\dagger, \beta_{(-k)}]=0$. Otherwise, 
\begin{align*}[\bT_n^\dagger, \beta_{(-k)}]
&=k\sum_{i\in I}\Big(
    \chi_\bfT(\alpha_i^R,\beta)p_{n-k}^\dagger(S^{-1}_{\bfT}(\alpha_i^L))+
    p_{n-k}^\dagger(\alpha_i^R)\chi_\bfT(S^{-1}_{\bfT}(\alpha_i^L),\beta)\Big)\\
    &=k\sum_{i\in I}\Big(
    \chi_\bfT(\alpha_i^R,\beta)p_{n-k}^\dagger(S^{-1}_{\bfT}(\alpha_i^L))+
    p_{n-k}^\dagger(\alpha_i^R)\chi_{\bfT}(\beta,\alpha_i^L)\Big)\\
&=k\sum_{i\in I}
    \chi_\bfT(\alpha_i^R,\beta)p_{n-k}^\dagger(S^{-1}_{\bfT}(\alpha_i^L))+kp_{n-k}^\dagger(\beta)=k\cdot \beta_{(n-k)}
\end{align*}
by Claim \ref{claim: 3}.\qedhere
    \end{proof}

Now Claims \ref{claim: 2} and \ref{claim: 4} together show that $\bL_n^\dagger$ satisfy equation \eqref{eq: virlva2}, so $\bL_n^\dagger=L_n$. \qedhere\end{proof}

\section{Virasoro constraints via wall-crossing}\label{sec: Virasoro and wall-crossing}

In this section, we prove the Virasoro constraints for quasi-smooth dg quivers as stated in Theorem \ref{thm: wt0virasoro}. In fact, we prove its generalization in terms Joyce's invariant classes as in Theorem \ref{thm: quivervirasoro}, allowing strictly semistable representations. Proof strategy is comparable to that of \cite{blm}. We first wall-cross to simpler moduli spaces where we can verify the Virasoro constraints directly and then show the compatibility between wall-crossing and Virasoro constraints using vertex algebras.

\subsection{Wall-crossing}\label{sec: wall-crossing}

Joyce's vertex algebra was introduced with the purpose of writing down and proving wall-crossing formulas for descendent integrals. In the case of quivers, this gives a complete algorithm to calculate descendent integrals in any moduli space of stable representations of an acyclic quiver with relations. 

Let $\dgQ$ be a quasi-smooth dg quiver and denote the associated quiver with relations by $(Q,I)$. We denote by $V^\bfQ=H_\ast(\CN^\dgQ)$ the vertex algebra associated to $\dgQ$. As explained in the beginning of Section \ref{sec: vertexalgebra}, the quotient $\widecheck V^\bfQ=V^\bfQ/T(V^\bfQ)$ is naturally a Lie algebra.

If $d\in \BN^{\dgQ_0}$ and $\theta$ is a stability condition such that $ M_{d}^{\theta-\textup{ss}}= M_{d}^{\theta-\textup{st}}$, then the moduli space admits a virtual fundamental class
\[[ M_{d}^{\theta}]^\vir\in H_{2-2{\chi}_\dgQ(d,d)}( M_{d}^{\theta})\,.\]

This virtual fundamental class can be push-forwarded to $\widecheck V$. Indeed, it can be shown \cite[Proposition 3.24]{Jo17} that if $d\neq 0$ then
\[H_\ast(\CN^{\dgQ}_d)/T\big(H_\ast(\CN^{\dgQ}_d)\big) \cong H_\ast\Big(\big(\CN^\dgQ_d\big)^{\textup{pl}}\Big)\]
is isomorphic to the homology of the rigidication of the stack $\CN^\dgQ_d$. There is a map $\iota\colon M_{d}^\theta\to \big(\CN^\dgQ_d\big)^{\textup{pl}}$, so we can pushforward the virtual fundamental class and regard it as an element of the Lie algebra
\[[ M_{d}^{\theta}]^\vir\in H_\ast\Big(\big(\CN^\dgQ_d\big)^{\textup{pl}}\Big)\subseteq \widecheck V^\bfQ\,.\]
Alternatively, we can think of the class $[M_d^\theta]^\vir\in \widecheck V^{\bfQ}$ as follows: by Proposition \ref{prop: realization iso}, $V_d^{\bfQ}=H_\ast(\CN_d^\dgQ)$ is the space of functionals on $\BD^{\dgQ}_d$. Since $T$ is dual to $\bR_{-1}$, $\widecheck V_d^{\bfQ}$ is the space of functionals on $\BD^{\dgQ}_{d, \inv}=\ker\big(\bR_{-1}\colon \BD^{\dgQ}_{d}\to \BD^{\dgQ}_{d}\big)$. Recall that by Lemma \ref{lem: weignt zero realization} there is a realization morphism from $\BD^{\dgQ}_{d, \inv}$ to $H^\ast(M_d^\theta)$; composing this realization morphism with integration against the virtual fundamental class defines the required functional. Wall-crossing formulas are written using the Lie bracket on $\widecheck V^{\bfQ}$. 

\begin{theorem}[{\cite[Theorems 5.7-5.9, 6.16]{Jo21}}]\label{thm: joycewc}
For any $d$ and $\theta$ there are uniquely defined classes $[M_d^{\theta}]^\inva\in \widecheck V^{\bfQ}$ that agree with the virtual fundamental classes when it is defined, and satisfy the wall-crossing formula. More precisely,
\[[M_d^{\theta}]^\inva=[M_d^{\theta}]^\vir\,,\quad\textup{when }M_{d}^{\theta-\textup{ss}}=M_{d}^{\theta-\textup{st}}\,,\]
and for any two stability conditions $\theta_1, \theta_2$ we have 
\[[M_d^{\theta_2}]^\inva=\sum_{d_1+\ldots+d_\ell=d}U(d_1, \ldots, d_\ell; \theta_1, \theta_2)\Big[\Big[\ldots\Big[[M_{d_1}^{\theta_1}]^\inva,[M_{d_2}^{\theta_1}]^\inva\Big],\ldots \Big],[M_{d_\ell}^{\theta_1}]^\inva\Big]\]
in $\widecheck V^{\bfQ}$, where $U(d_1, \ldots, d_\ell; \theta_1, \theta_2)\in \BQ$ are explicitly computable coefficients.
\end{theorem}
\begin{remark}\label{rem: abelian and triangulated}
    The previous theorem in \cite{Jo21} is proved using a vertex algebra whose underlying vector space is given by the homology group of the moduli stack $\CM^{Q,I}$ of representations of $(Q,I)$. Since any representation of $\BC[Q]/I$ induces a dg module of $\BC[\dgQ]$ by restriction of scalars, we have a morphism of (higher) stacks
    $$i:\CM^{Q,I}\rightarrow \CN^{\dgQ}\,. 
$$
This morphism clearly intertwines the zero maps, the direct sum maps and the $B\BG_m$-actions which were part of the geometric input for defining the vertex algebra structure. Most importantly, the perfect complex $\Ext^{\textnormal{Joy}}$ used in \cite{Jo21} depends on a choice of generators $I=(r_1,\dots, r_n)$. Since we use a quasi-smooth dg quiver $\dgQ$ according to this choice of generators, it follows that the morphism $i$ intertwines the perfect complex, i.e.,
$$(i\times i)^*\Ext_{\dgQ} \simeq \Ext^{\textnormal{Joy}}\,.
$$
Therefore, $i$ induces a vertex algebra homomorphism $V^{Q,I}\rightarrow V^\dgQ$ and a Lie algebra homomorphism $\widecheck V^{Q,I}\rightarrow \widecheck V^\dgQ$\,. This allows to transport the wall-crossing formula of \cite{Jo21} to the Lie algebra associated to $\dgQ$. 

\end{remark}

Theorem \ref{thm: joycewc} allows us to determine the invariant classes $[M_d^\theta]^\inva$ completely by writing everything in terms of an increasing stability condition. We say that $\theta$ is increasing if $\theta_{t(e)}>\theta_{s(e)}$ for any $e\in Q_1$. The following proposition says that for an increasing stability condition, the invariant classes of moduli spaces  are essentially trivial.

\begin{proposition}[{\cite[Theorem 6.19]{Jo21}}]\label{prop: increasingstability}
    Let $\theta$ be an increasing stability condition. We have
    \[
    [M_d^\theta]^\inva=\begin{cases}\overline{e^d\otimes 1}&\textup{ if }d=\delta_v \textup{ for some }v\in Q_0\\
    0 & \textup{ otherwise.}
        \end{cases}\]
\end{proposition}
Together with the wall-crossing formula, this determines the classes $[M_d^\theta]^\inva$ for every $\theta$.

\begin{example}\label{ex: wcgrassmannian}
We consider the Grassmannian $\Gr(N,k)$ and express its class using the wall-crossing formula. Recall the notations from Example \ref{ex: Grassmannian}. Let $\theta_2$ be a decreasing stability condition of the Kronecker quiver $K_N$ which corresponds to the Grassmannian and $\theta_1$ be an increasing stability condition. By Theorem \ref{thm: joycewc}, we have a wall-crossing formula
\[[M_{(1, k)}^{\theta_2}]=\sum_{d_1+\ldots+d_\ell=(1,k)}U(d_1, \ldots, d_\ell; \theta_1, \theta_2)\Big[\Big[\ldots\Big[[M_{d_1}^{\theta_1}]^\inva,[M_{d_2}^{\theta_1}]^\inva\Big],\ldots \Big],[M_{d_\ell}^{\theta_1}]^\inva\Big]\]
in the Lie algebra $\widecheck V^{K_N}$. By Proposition \ref{prop: increasingstability}, it suffices to consider the contributions where all $d_i$'s are either $\delta_1$ or $\delta_\infty$ and compute the wall-crossing coefficients. The computation of the coefficients can be done as in \cite[Proposition 13.10]{JS12} and we find
$$[\Gr(N,k)]=\frac{1}{k!}[e^{\delta_1},\ldots [e^{\delta_1},[e^{\delta_1},e^{\delta_\infty}]]\ldots ]\,.$$
In Section \ref{sec: sym Heke}, we further simplify this expression using the symmetrized Hecke operators and relate this formula to Schubert calculus.
\end{example}

\subsubsection{Vertex algebra and wall-crossing for framed moduli}

We briefly explain here the natural adaptation of these ideas to moduli of framed representations. For each $f\in \BN^{\dgQ_0}$ and $d\in \BZ^{\dgQ_0}$, we let $\CN^{\dgQ}_{f\rightarrow d}=\CN^{\dgQ}_d$ and define 
$$\CN^{\dgQ, \textup{fr}}\coloneqq \bigsqcup_{(f,d)\in \BN^{\dgQ_0}\times \BZ^{\dgQ_0}}\CN^{\dgQ}_{f\rightarrow d}\,. 
$$
Define the framed version of the Ext $K$-theory class on $\CN^{\dgQ, \textup{fr}}\times \CN^{\dgQ, \textup{fr}}$ as 
$$\Ext_\dgQ^{\textup{fr}}:=\Ext_\dgQ-\sum_{v\in \dgQ_0}\textup{R}\CH om(\CO^{(f_1)_v},\CV_v^{(2)})
$$
where $f_1$ is the framing vector of the first factor and $\CV^{(2)}_v$ is the universal complex pulled back from the second factor. The rank of this $K$-theory class restricted to $\CN^{\dgQ}_{f_1\rightarrow d_2}\times \CN^{\dgQ}_{f_2\rightarrow d_2}$ is given by
\[\chi_{\dgQ, \textup{fr}}\big((f_1, d_1), (f_2, d_2)\big):=\chi_{\dgQ}(d_1, d_2)-f_1\cdot d_2\,.\]

The stack $\CN^{\dgQ, \textup{fr}}$ is equipped with the following geometric data:
\begin{enumerate}
    \item A zero map $0:\pt\rightarrow \CN^{\dgQ, \textup{fr}}$. 
    \item A direct sum map $\Sigma:\CN^{\dgQ}_{f_1\rightarrow d_2}\times \CN^{\dgQ}_{f_2\rightarrow d_2}\rightarrow \CN^{\dgQ}_{(f_1+f_2)\rightarrow (d_1+d_2)}$. 
    \item A $B\BG_m$-action $\rho:B\BG_m\times \CN^{\dgQ, \textup{fr}}\rightarrow \CN^{\dgQ, \textup{fr}}$. 
    \item A $K$-theory class $\Theta^{\textup{fr}}=\sigma^*\Ext_\dgQ^{\textup{fr}}+(\Ext_\dgQ^{\textup{fr}})^\vee$.
\end{enumerate}
The first three items come directly from the corresponding maps of the unframed stack $\CN^{\dgQ}$. This data defines a framed vertex algebra 
$$V^{\dgQ, \textup{fr}}\coloneqq \Big(H_\ast(\CN^{\dgQ, \textup{fr}})\,,\, \ket{0}\,,\, T\,,\, Y\Big)\,. 
$$
We note that this construction is similar to the one in \cite[Section 5.2]{Jo21}, where the author considers representations of a quiver for which some subset of the vertices $\Ddot{Q}_0\subseteq Q_0$ are framed.

Structure of the framed vertex algebra can be described precisely. The same proof of Theorem \ref{thm: natural iso} shows that we have a vertex algebra embedding 
$$V^{\dgQ,\textup{fr}}\hookrightarrow \textup{VA}\Big(K(\dgQ)\times K(\dgQ), \chi_{\dgQ, \textup{fr}}^\sym\Big)\,.
$$
On the level of underlying vector spaces, this embedding at the connected component $\CN^{\dgQ}_{f\rightarrow d}$ is given by the composition
$$H_*(\CN^{\dgQ}_{f\rightarrow d})=H_*(\CN^{\dgQ}_{d})=e^{(f,d)}\cdot\BD_{K(\dgQ)}\hookrightarrow e^{(f,d)}\cdot \BD_{K(\dgQ)\times K(\dgQ)}
$$
where the last embedding is induced from $K(\dgQ)=\{0\}\times K(\dgQ)\subseteq K(\dgQ)\times K(\dgQ)$.

The framed vertex algebra is a natural place to write down the wall-crossing formulas for framed moduli spaces. Let $\dgQ$ be a quasi-smooth quiver and consider a framed moduli space $M^{\theta}_{f\rightarrow d}$. The universal framed representation of the moduli space induces a morphism $M^{\theta}_{f\rightarrow d}\rightarrow \CN^{\dgQ}_{f\rightarrow d}$, so we can pushforward its virtual class to the framed vertex algebra
$$[M^{\theta}_{f\rightarrow d}]^\vir\in V^{\dgQ, \textup{fr}}\,. 
$$
Note that this is different from the case of unframed representations, where moduli spaces only define a class on the homology of the rigidification or, equivalently, on the Lie algebra.

Finally, we explain how to obtain the wall-crossing formula in $V^{\dgQ,\textup{fr}}$ from the usual wall-crossing formula via framed/unframed correspondence. Given a fixed framing vector $f$, consider the auxiliary quiver $\bfQ^f$ as explained in Section \ref{sec: f/uf corr}. Then framed/unframed correspondence reads $M^{\theta}_{f\rightarrow d}=M^{\tilde{\theta}}_{(1,d)}$ for some $\tilde{\theta}$. On the other hand, if $\theta'$ is any stability condition of $\dgQ^f$, then wall-crossing formula reads 
\[[M_{(1,d)}^{\tilde\theta}]^\inva=\sum_{\tilde d_1+\ldots+\tilde d_\ell=(1,d)}U(\tilde d_1, \ldots, \tilde d_\ell; \theta', \tilde \theta)\Big[\Big[\ldots\Big[[M_{\tilde d_1}^{\theta'}]^\inva,[M_{\tilde d_2}^{\theta'}]^\inva\Big],\ldots \Big],[M_{\tilde d_\ell}^{\theta'}]^\inva\Big]\]
where $\tilde d_i=(1, d_i)$ for exactly one $1\leq i \leq \ell$ and $\tilde d_j=(0, d_j)$ for $j\neq i$. This identity holds in the Lie algebra $\widecheck V^{\dgQ^f}$. Since we have
\begin{align*}
    \chi_{\dgQ, \textup{fr}}((0, d_1), (0, d_2))&=\chi_{\bfQ^f}((0,d_1), (0,d_2))
    ,\\
    \chi_{\dgQ, \textup{fr}}((f, d_1), (0, d_2))&=\chi_{\bfQ^f}((0,d_1), (1,d_2))
    ,\\
    \chi_{\dgQ, \textup{fr}}((0, d_1), (f, d_2))&=\chi_{\bfQ^f}((1,d_1), (0,d_2))
    ,
\end{align*}
the above wall-crossing formula in $\widecheck V^{\dgQ^f}$ can instead be written as an identity in $\widecheck V^{\dgQ, \textup{fr}}$ after we replace $M_{(0,d)}^{\theta'}(Q^f, I^f)$ by $M_{0\to d}^{\theta'}(Q,I)$ and  $M_{(1,d)}^{\theta'}(Q^f, I^f)$ by $M_{f\to d}^{\theta'}(Q,I)$.\footnote{This requires a version of framed/unframed correspondence beyond limit stability condition which can be obtained by straightforward adaptation of Section \ref{sec: f/uf corr}.} Moreover, the same argument as explained in \cite[Lemma 5.11]{blm} shows that this formula can actually be lifted to the vertex algebra $V^{\dgQ, \textup{fr}}$.

\begin{example}
    \label{ex: grassva}
    We return to the example of the Grassmannian and the quiver with one vertex $Q=A_1$. We denote by $V^\Gr$ the vertex algebra 
    \[V^\Gr=V^{A_1, \textup{fr}}=H_\ast(\CN^{A_1, \textup{fr}})= H_\ast(BU)[Q, q^{\pm 1}]\]
    where we write $Q^N q^k=e^{(N, k)}$ for $(N, k)\in \BN\times \BZ$ corresponding to the connected component of framed representations $N\to k$. This is a vertex subalgebra of
    \[\textup{VA}(\BZ\times \BZ, \chi_{\Gr}^\sym)\]
    where
    \[\chi_{\Gr}((N_1, k_1), (N_2, k_2))=k_2(k_1-N_1)\,.\]
    The wall-crossing formula from Example \ref{ex: wcgrassmannian} becomes
    \[[\Gr(N, k)]=\frac{1}{k!}[q, \ldots [q, [q, Q^N]]\ldots ]\in V^\Gr\,,\]
    where $[u, v]=u_0 v$ is the partial lift of the Borcherds' Lie bracket discussed in \cite[Lemma 3.12]{blm}. 
\end{example}

\subsection{Vertex algebra comparison for dg quivers}\label{sec: assumptiondgquivers}

In this section, we show that Assumption \ref{ass: dg category} is satisfied for the dg category $\bfT$ associated to the dg quiver $\bfQ$. Thus, the general results from Section \ref{subsec: comparisonva} apply to $\bfQ$. 

We first connect the language specific to dg quivers and the general language for any dg category. Recall that we had two types of descendent theory, one using cohomology $H(\bfQ)$ as in Section \ref{sec: descendent algebra} and the other using the perfect complex $\Ext_{\bfQ}$ as in Section \ref{sec: Ext descendent}. We remark that the realization homomorphism in \ref{sec: descendent algebra} can be lifted, using the universal complex $\CV_v$ for each $v\in \dgQ_0$, to the stack of a dg quiver
\begin{center}
\begin{tikzcd}
\BD^{\dgQ}_d \arrow[r] \arrow[rd, dashed] & H^*(\CM_{d}^{Q,I})                 \\
                               & H^*(\CN_{d}^{\dgQ}) \arrow[u, "i^*"']
\end{tikzcd}
\end{center}
where $i:\CM_d^{Q,I}\rightarrow \CN_d^\dgQ$ is the morphism explained in Remark \ref{rem: abelian and triangulated}. 

As a first step toward the proof, we show that the two types of descendent theories, and respective Virasoro operators, are compatible.

\begin{proposition}\label{prop: comparison of two descendents}
    Let $\dgQ$ be a quasi-smooth dg quiver and $\bfT$ be the natural dg enhancement of $D^b(\bfQ)$. Then we have a commutative diagram 
\begin{center}
    \begin{tikzcd}[row sep=small]
\BD^\bfT_d \arrow[rd, ""] \arrow[dd, "\sim"  {rotate=90, anchor=north}, "\phi"'] &   \\
                                   & H^*(\CN_d^{\dgQ}) \\
\BD^\dgQ_d \arrow[ru, ""']                 &  
\end{tikzcd}
\end{center}
where $\phi$ maps the Ext descendent $\ch^L_k([P(v)])$ to the cohomological descendent $\ch_k(v)$. Furthermore, $\phi$ intertwines the Virasoro operators on the two descendent algebras. 
\end{proposition}
\begin{proof}
    For each $v\in\dgQ_0$ and a dg module $M$ of $\dgQ$, we have a quasi-isomorphism $\Hom^\bullet_\dgQ(P(v),M)\simeq M_v$. This implies $\Ext_\dgQ(P(v),-)\simeq \CV_v$ which justifies the definition of $\phi$ and commutativity. 
    
    Intertwining property of the derivation parts is clear. To show that the multiplication parts intertwines through $\phi$, we use the formula for the $K$-theoretic diagonal of a dg quiver
    $$\Delta^K=\sum_{v\in \dgQ_0}[S(v)]\otimes [P(v)]\in K(\dgQ)\otimes K(\dgQ)\,.
$$
Therefore, the multiplication part of the Ext Virasoro operator is
\begin{align}\label{eq: Ext Tn}
\nonumber    \bT_n
    &=\sum_{a+b=n}a!b!\,\sum_{v\in \dgQ_0} (-1)^a\ch_a^R(S(v))\ch_b^L(P(v))\\
    &=\sum_{a+b=n}a!b!\,\sum_{v\in \dgQ_0} \ch_a^L(S_\bfT^{-1}(S(v)))\ch_b^L(P(v))\,.
\end{align}
For each $v$, we have rational numbers $a_{vw}$ such that 
$$S_\bfT^{-1}(S(v))=\sum_{w\in \dgQ_0} a_{vw} [P(w)]\in K(\dgQ)\,.
$$
Taking $\chi_\dgQ(-,S(u))$ on the both sides, we obtain 
$$\chi_\dgQ(S(u),S(v))=\chi_\dgQ(S_\bfT^{-1}(S(v)),S(u))=\sum_{w\in \dgQ_0}a_{vw}\chi_\dgQ(P(w),S(u))= a_{vu}
$$
In conclusion, $\phi$ maps the expression \eqref{eq: Ext Tn} to 
$$
\sum_{a+b=n}a!b!\sum_{v,w\in \dgQ_0} a_{vw}\,\ch_a(w)\ch_b(v)=\sum_{a+b=n}a!b!\sum_{v,w\in \dgQ_0} \chi_\dgQ(v,w)\ch_a(v)\ch_b(w)
$$
which is precisely $\bT_n$ on $\BD^\bfQ$. \qedhere
\end{proof}

By the commutative diagram in the above proposition, Assumption \ref{ass: dg category} (2) is satisfied if and only if the cohomological realization homomorphism is an isomorphism. We prove this for any finite acyclic dg quivers. 

\begin{proposition}\label{prop: realization iso}
    Let $\dgQ$ be a finite and acyclic dg quiver and $\bfT$ be the natural dg enhancement of $D^b(\bfQ)$. Then the realization homomorphism \eqref{eq: assumption iso} is an isomorphism. 
\end{proposition}
\begin{proof}
When $\dgQ$ is a usual quiver $Q$, this statement is essentially \cite[Proposition 5.13]{Jo17}. Joyce remarks in the last paragraph of Section 6.2 in \cite{Jo21} that the same result in the abelian category setting for $Q$ with relations is true, as long as the relations are homogeneous (i.e. $r(q)$ is a linear combination of paths of the same length). Motivated by this, we give a proof for any dg quiver.

We abuse the notation by writing $Q_0$ for a quiver with the set of vertices given by $\dgQ_0$ but no arrows. Recall that $S=\BC[Q_0]$. We have a sequence of maps of higher stacks
$$\CN^{Q_0}\xrightarrow{f}\CN^{\bfQ}\xrightarrow{g} \CN^{Q_0}
$$
whose $B$-points for any commutative algebra $B$ is given as follows: $f$ sends a perfect $S\otimes B$-module $M$ to the $\BC[\bfQ]\otimes B$-module $\BC[\bfQ]\otimes_S M$  and $g$ sends a $\BC[\bfQ]\otimes B$-module $N$ perfect over $B$ to the $S\otimes B$-module $\BC[\bfQ]/J\otimes B$-module $\BC[\bfQ]/J\otimes_{\BC[\bfQ]} N$ . Since $\BC[\dgQ]/J\simeq S$, the composition $g\circ f$ is an identity morphism. 

Assume for a moment that $f\circ g:\CN^{\bfQ}\rightarrow\CN^{\bfQ}$ is homotopic to the identity map. Since $\CN^{Q_0}\simeq \prod_{v\in Q_0}\Perf^{(v)}$ is a product of $|Q_0|$ copies of the stack of perfect $\BC$-complexes $\Perf$, this gives an isomorphism
$$g^*:H^*(\CN^{Q_0})=\bigotimes_{v\in Q_0}H^*(\Perf^{(v)})\xrightarrow{\sim} H^*(\CN^{\bfQ}).
$$
The stack $\Perf_{d_v}$ is well-known to be homotopically equivalent to $BU$ (see \cite[Theorem 4.5]{B16} or \cite[Proposition 5.9]{Jo17}) and so
\[H^\ast(\Perf_{d_v})=\BQ[\ch_1(\CV_v), \ch_2(\CV_v), \ldots]\]
is the free polynomial ring in the Chern characters of the universal complex $\CV_v$. This proves that the realization homomorphism $\BD^{\dgQ}_d\rightarrow H^*(\CN_{d}^{\dgQ})$ is an isomorphism. 

We are left to construct a homotopy between $f\circ g$ and $\id$. Since $\bfQ$ is acyclic, we can pick an arbitrary function $w\colon Q_0\to \BZ$ that is increasing along arrows, i.e. such that $w(e)\coloneq w(t(e))-w(s(e))\in \BZ_{>0}$ for every $e\in \bfQ_1$. The weight extends to any path in $\BC[\bfQ]$ multiplicatively, i.e., $w(e_\ell\cdots e_1)\coloneqq\sum_{i=1}^\ell w(e_i)$. We define a map of higher stacks 
$$r\colon \BA^1\times \CN_{\bfQ}\rightarrow \CN_{\bfQ}
$$
whose action on $B$-points for any commutative algebra $B$ is described as follows. If $t\in B$ and $M$ is a $\BC[\bfQ]\otimes B$-module which is $B$-perfect, then $r(t,M)$ is a $\BC[\bfQ]\otimes B$-module whose underlying $B$-module is the same as $M$ with $\BC[\bfQ]$-module structure twisted by $w$ and $t$, i.e., 
$$a\cdot_w m\coloneqq t^{w(a)}(a\cdot m),\quad a\in \BC[\bfQ],\ m\in M
$$
when $a=e_\ell\cdots e_1$ and we extend the definition linearly.  
Since $B$ is a commutative algebra, this indeed defines a $\BC[\bfQ]\otimes B$-module. The morphism $r$ restricted to $1\in \BA^1$ is clearly the identity map. On the other hand, if we restrict $r$ to $0\in \BA^1$, then the $w$-twisted multiplication map by $\BC[\bfQ]$ becomes trivial away from $S=\BC[Q_0]$ because path of positive length have positive weight. In other words, $r$ restricted to $0\in \BA^1$ becomes the map $f\circ g$. This proves that $f\circ g$ and $\id$ are homotopic to each other after taking the topological realizations.
\end{proof}

\begin{proposition}\label{prop: checking assumptions}
    Let $\dgQ$ be a finite and acyclic dg quiver and $\bfT$ be the natural dg enhancement of $D^b(\bfQ)$. Then $\bfT$ satisfies Assumption \ref{ass: dg category}. 
\end{proposition}
\begin{proof}
Recall that $K(\bfQ)\simeq \BZ^{\dgQ_0}$ since its $K$-theory is generated by projective dg modules $P(v)$ and they satisfy $\chi(P(v),S(w))=\delta_{vw}$. Also, the proof of Proposition \ref{prop: realization iso} shows, in particular, that $\pi_0(\CN^\dgQ)$ is $\BZ^{\dgQ_0}$. Since Proposition \ref{prop: realization iso} precisely covers the Assumption \ref{ass: dg category} (2), we are left with Part (3). This follows from the standard resolution; indeed, the diagonal bimodule corresponds to $\BC[\bfQ]$ regarded as a bimodule over itself, and by Theorem \ref{thm: dgstandardresolution} its class in $K$-theory is
\[\sum_{v\in \bfQ_0}[P(v)]\boxtimes [R(v)]-\sum_{e\in \bfQ_1}(-1)^{|e|}[P(t(e))]\boxtimes [R(s(e))]=\sum_{v\in \bfQ_0}[P(v)]\boxtimes [S(v)]\]
where $R(v)\coloneqq {\bf 1}_{v}\cdot \BC[\bfQ]$. The equality comes from the following exact sequence of $\BC[\bfQ]^{\textup{op}}$-modules for each $v\in \dgQ_0$:
\[0\rightarrow \hspace{-10pt}\bigoplus_{e\in \dgQ_1,\ t(e)=v}\hspace{-10pt} R(s(e))[-|e|]\rightarrow R(v)\rightarrow S(v)\rightarrow 0\,. \qedhere
\]
\end{proof}

\begin{corollary}\label{cor: grossisoquiver}
   Let $\dgQ$ be a finite and acyclic dg quiver. Then there is a natural isomorphism of vertex algebras
       $$V^\bfQ\rightarrow \textup{VA}(K(\bfQ),\chi_{\bfQ}^{\sym}). 
$$
\end{corollary}
\begin{proof}
    By Proposition \ref{prop: checking assumptions}, Theorem \ref{thm: natural iso} applies to $\bfQ$.
\end{proof}

\subsection{Proof of Virasoro constraints for quivers}\label{sec: VA and Virasoro}

In this section, we apply the general theory developed up until here and show that the wall-crossing formula from Section \ref{sec: wall-crossing} implies the Virasoro constraints for quasi-smooth quivers, recoving the Virasoro constraints proven in \cite{bojko}. Recall Definition \ref{def: primarystates} of primary state.

\begin{corollary}
Suppose that $M_d^{\theta-\textup{ss}}=M_d^{\theta-\textup{st}}$. Then $M_d^{\theta}$ satisfies the Virasoro constraints (Theorem \ref{thm: wt0virasoro}) if and only if $[M_d^\theta]^\vir\in \widecheck V^{\bfQ}$ is a primary state.
\end{corollary}
\begin{proof}
By duality of Virasoro operators as in Proposition \ref{prop: duality of operators} and by Proposition \ref{prop: comparison of two descendents}, the Virasoro constraints of Section \ref{subsec: virasoroops} are equivalent to 
    \[0=\bL_\inv^\dagger([M_d^\theta]^\vir)=\sum_{n\geq -1}\frac{(-1)^n}{(n+1)!}L_{-1}^{n+1}\circ L_n([M_d^\theta]^\vir)\,.\]
The conclusion follows from Corollary \ref{cor: primarywt0}.
\end{proof}

In light of the previous corollary, it makes sense to talk about Virasoro constraints even when there are semistable representations, thanks to Joyce's classes from Theorem~\ref{thm: joycewc}. We say that $M_d^{\theta-\textup{ss}}$ satisfies Virasoro constraints if and only if $[M_d^{\theta}]^\inva \in \widecheck V^{\bfQ}$ is a primary state. 

\begin{theorem}\label{thm: quivervirasoro}
Let $\dgQ$ be any quasi-smooth dg quiver. For any $d\in \BN^{Q_0}$ and stability condition $\theta$, the class $[M_d^{\theta}]^\inva$ is a primary state. In other words, every moduli space of quiver representations satisfies the Virasoro constraints.    
\end{theorem}
\begin{proof}
    When $\theta$ is an increasing stability condition, the claim follows trivially from Proposition \ref{prop: increasingstability} and the definition of the operators $L_n$, namely \eqref{eq: virlva1}. By Joyce's wall-crossing formula (Theorem \ref{thm: joycewc}) and the fact that $\widecheck P_0$ is a Lie subalgebra (Corollary \ref{cor: liesubalgebra}) the result follows for every $\theta$.\qedhere
\end{proof}

\section{Application to surfaces with exceptional collections}\label{sec: applicaton}

The goal of this section is to clarify how the quiver Virasoro constraints are related to the sheaf Virasoro constraints for surfaces admitting exceptional collections, especially for del Pezzo surfaces. In particular, we prove Theorem \ref{main thm: abch implies Virasoro} and \ref{main thm: three surfaces}. 

\subsection{Exceptional collection and derived equivalence}
\label{sec: beilinson}

We start by recalling the notion of exceptional collection and the induced equivalence between a triangulated category with a full and strong exceptional collection and the derived category of a quiver with relations. From now on let $\CD$ be a triangulated category (e.g. the derived category of coherent sheaves on $X$ or representations of a quiver with relations $(Q,I)$).

\begin{definition}
A sequence $\FE=(E_1,  \ldots, E_n)$ of objects of $\CD$ is called an exceptional collection if
\[\Hom_\CD(E_i, E_j[\ell])=\begin{cases}
\BC& \textup{ if }i=j\textup{ and }\ell=0\\
0& \textup{ if }i>j\textup{ or }[i=j\textup{ and }\ell\neq 0]\,.
\end{cases}\]
An exceptional collection is said to be strong if moreover for any $i,j$ we have
\[\Hom_\CD(E_i, E_j[\ell])=0\quad\textup{ for }\ell \neq 0\,.\]
An exceptional collection is said to be full if the smallest triangulated subcategory of $\CD$ containing $\FE$ is $\CD$ itself.
\end{definition}

If $\CD$ admits a full exceptional collection $\FE$ then
\[K(\CD)=\bigoplus_{i=1}^n \BZ\cdot [E_i]\]
and $\chi_\CD\colon K(\CD)\times K(\CD)\to \BZ$ is given in the basis above by a triangular matrix with $1$ along the diagonal. In particular $\chi_\CD$ is non-degenerate.

A full exceptional collection $\FE$ always admits a unique left dual exceptional collection $\widetilde \FE=(\widetilde E_n, \ldots, \widetilde E_1)$ such that  
\[\Hom_\CD(\widetilde E_i, E_j[\ell])=\begin{cases}\BC&\textup{ if }i=j\textup{ and }\ell=0\\
0&\textup{ otherwise}
\end{cases}\,.\]
In particular
\[\chi\big(\widetilde E_i, E_j\big)=\delta_{ij}\,.\]

\begin{example}
Let $(Q, I)$ be an acyclic quiver with relations; without loss of generality denote the nodes as $Q_0=\{1, 2, \ldots, n\}$ in a way that there are no arrows $i\to j$ for $i>j$. 
Recall from Example \ref{ex: simpleprojective} the simple and projective modules $S(i)$ and $P(i)$. Then 
\[E_i=S(i)\,,\quad \widetilde E_i=P(i)\]
are dual full exceptional collections on $D^b(Q, I)$. We call $(E_1,\dots, E_n)$ the standard exceptional collection of $(Q, I)$. The dual collection is strong by projectivity of $P(i)$. The same applies to a dg quiver, see Example \ref{ex: hprojectivequiver}.
\end{example}

\begin{example}\label{ex: exceptionalcollections}
    The derived category of any smooth projective toric variety admits a full exceptional collection \cite{kawamata}. For $D^b(\BP^2)$ we can take 
    \[\FE=\big(\CO_{\BP^2}(-1)[2], \CO_{\BP^2}[1], \CO_{\BP^2}(1)\big)\,,\quad \widetilde\FE=\big(\CO_{\BP^2}(1), T_{\BP^2}, \CO_{\BP^2}(2)\big)\,.\]
    Both $\FE$ and $\widetilde \FE$ are strong. For $D^b(\BP^1\times \BP^1)$ we can take 
      \begin{align*}\FE&=\big(\CO_{\BP^1\times \BP^1}(0,-1)[2], \CO_{\BP^1\times \BP^1}[1], \CO_{\BP^1\times \BP^1}(1,-1)[1], \CO_{\BP^1\times \BP^1}(1,0)\big)\,,\quad\\ \widetilde
      \FE&=\big(\CO_{\BP^1\times \BP^1}(1,0), \CO_{\BP^1\times \BP^1}(1,1), \CO_{\BP^1\times \BP^1}(2,0), \CO_{\BP^1\times \BP^1}(2,1)\big)\,.
      \end{align*}
      In this case, $\FE$ is not strong, but $\widetilde \FE$ is strong.
\end{example}
Given a triangulated category $\CD$ and an exceptional collection $\FE$ whose left dual $\widetilde \FE$ is full and strong, we associate a quiver $Q$ together with an ideal of relations $I\subseteq \BC [Q]$. The set of vertices of $Q$ is $\{1, \ldots, n\}$ and the number of arrows $i\to j$ is equal to $\dim \Ext^1(E_i, E_j)$ for $i\neq j$; since there are no arrows $i\to j$ if $i\geq j$, the quiver $Q$ is acyclic. Then there is an ideal of relations $I$ such that
\[\BC [Q]/I=\textup{End}_{\CD}\Big(\bigoplus_{i=1}^n \widetilde E_i\Big)\,.\]

In the next theorem we let $\CD$ be a sufficiently nice triangulated category such as $D^b(X)$ for some smooth projective variety $X$ or $D^b(Q, I)$.

\begin{theorem}[{\cite[Theorem 6.2]{bondal}}]\label{thm: beilinson}
Suppose that $\CD$ has an exceptional collection $\FE$ such that its left dual exceptional collection $\widetilde \FE$ is full and strong. Then there is an isomorphism of derived categories
\[B\colon \CD\to D^b(Q, I)\]
sending $\FE$ to the canonical exceptional collection of $(Q, I)$. The isomorphism sends an object $F$ of $\CD$ to the complex of representations of $(Q, I)$ which at vertex $i$ has complex $\RHom_\CD\big(\widetilde E_i, F\big)$. 
\end{theorem}

By Theorem \ref{thm: koszuldgquiver}, there is a dg quiver $\bfQ$ with $\dim \Ext_{\CD}^{k+1}(E_i, E_j)$ arrows from $i$ to $j$ of degree $-k$ such that 
\[\CD\simeq D^b(Q, I)\simeq D^b(\bfQ)\,.\]
Up to choices of basis, this dg quiver is canonically associated to the pair $(\CD, \FE)$.

\begin{example}The theorem applied to the derived category of $D^b(\BP^2)$ with the full exceptional collection of Example \ref{ex: exceptionalcollections} establishes an isomorphism between $D^b(\BP^2)$ and the derived category of the Beilinson quiver with relations $B_3$, see Example \ref{ex: beilinson}. See also Example \ref{ex: beilinsondgquiver} where the associated dg quiver is discussed.

If we apply it to $D^b(\BP^1\times \BP^1)$ with the exceptional collection from Example \ref{ex: exceptionalcollections} we obtain an isomorphism with the derived category of the quiver
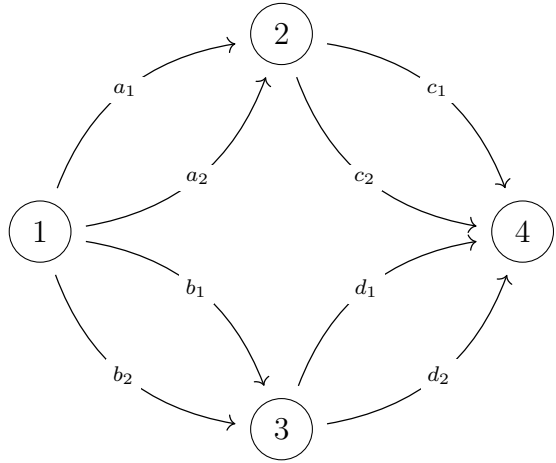
\begin{center}
\begin{tikzcd}  &  & \tikz \node[draw, circle]{2}; \arrow[rrdd, "c_2" description, bend right] \arrow[rrdd, "c_1" description, bend left] &  &   \\
             &  &                              &  &   \\
\tikz \node[draw, circle]{1};\arrow[rruu, "a_2" description, bend right] \arrow[rruu, "a_1" description, bend left] \arrow[rrdd, "b_2" description, bend right] \arrow[rrdd, "b_1" description, bend left] &  &                   &  & \tikz \node[draw, circle]{4}; \\ &  &   &   \\
&  & \tikz \node[draw, circle]{3};\arrow[rruu, "d_2" description, bend right] \arrow[rruu, "d_1" description, bend left] &  &  
\end{tikzcd}
\end{center}
subject to relations 
\[c_1a_1+ d_1b_1=c_2a_2+ d_2b_2=c_2 a_1+d_1 b_2=c_1 a_2+d_2 b_2=0\,.\]
\end{example}

\subsubsection{Comparing vertex algebras}
Let $X$ be a smooth projective variety and suppose that $\CD=D^b(X)$ satisfies the hypothesis of Theorem \ref{thm: beilinson}. The isomorphism of derived categories $D^b(X)\simeq D^b(Q, I)\simeq D^b(\bfQ)$ can be upgraded to an equivalence between the natural dg enhancements, so it induces a quasi-equivalence between the derived stacks $\mathbf{B}\colon \bfCN^X\to \bfCN^\bfQ$. As in Remark \ref{rem: natural}, the pushforward on homology 
\[B_\ast\colon H_\ast(\CN^{X})\to H_\ast(\CN^{\bfQ})\]
is an isomorphism between the vertex algebras canonically associated to (the derived enhancements of) $D^b(X)$ and $D^b(\bfQ)$. 

These vertex algebras are naturally isomorphic to the lattice vertex algebras $\textup{VA}(K(X), \chi_X^{\sym})$ and $\textup{VA}(K(Q), \chi_{\dgQ}^{\sym})$, see Corollary \ref{cor: grossisoquiver}. The isomorphism of derived categories induces an isomorphism $\phi\colon K(X)\to K(Q)$ that preserves the respective Euler pairings $\chi_X$ and $\chi_\bfQ$. The naturality of this isomorphism amounts to commutativity of the following diagram:
\begin{center}
\begin{tikzcd}
H_\ast(\CN^{X})\arrow[rr]\arrow[d, "B_\ast"]& &\textup{VA}(K(X), \chi_X^{\sym})\arrow[d, "\phi"]\\
H_\ast(\CN^{\dgQ})\arrow[rr]& &\textup{VA}(K(Q), \chi_{\dgQ}^{\sym})
\end{tikzcd}
\end{center}

\subsection{Bridgeland stability and the ABCH program}

Given a (sufficiently nice, for instance $D^b(X)$, $D^b(Q, I)$ or $D^b(\bfQ)$) triangulated category $\CD$, Bridgeland defined in \cite{Bri} a notion of (numerical) stability conditions on $\CD$ which we refer to as Bridgeland stability conditions. A Bridgeland stability consists in a pair $\sigma=(\CA, Z)$ where $\CA$ is the heart of a t-structure on $\CD$ and $Z\colon K_{\textnormal{num}}(\CD)\to \BC$ is a group homomorphism, called the central charge.\footnote{Here $K_{\textnormal{num}}(\CD)$ denotes the quotient of $K(\CD)$ by the kernel of the Euler paring. The numerical $K$-group is equal to the original $K$-group in all cases of our applications. } Among other things, it is required that $Z(F)\in \BH$ for $F\in \CA\backslash\{0\}$ where $\BH$ is the upper half plane with a negative real line
\[\BH=\{z\in \BC\colon \Im(z)>0\}\cup (-\infty, 0)\,.\]
A Bridgeland stability defines a slope function
\[\mu^\sigma(F)=-\frac{\Re(Z(F))}{\Im(Z(F))}\in \BR\cup\{+\infty\}\]
on $\CA\backslash\{0\}$. An object of $\CA$ is said to be $\sigma$-(semi)stable if every nonzero proper subobject has smaller (or equal) slope. Given $\v\in K(\CD)$ we denote by $\CM^{\sigma-\textup{st}}_\v\subseteq \CM^{\sigma-\textup{ss}}_\v$ the moduli stacks of $\sigma$-stable and $\sigma$-semistable objects in class $\v$. A stability condition defines a slicing of $\CD$, i.e. a collection of full aditive subcategories $\CP(\phi)\subseteq \CD$ for each $\phi\in\BR$: if $0<\phi\leq 1$ then $\CP(\phi)$ consists of the zero object and $\sigma$-semistable objects of $\CA\backslash\{0\}$ such that $Z(F)=m(F)e^{\pi \sqrt{-1} \phi}$ for some $m(F)\in \BR^+$; for the remaining $\phi$, we extend this notion by setting $\CP(\phi+1)=\CP(\phi)[1]$. 

Bridgeland shows that the space of stability conditions $\Stab(\CD)$ is a complex manifold; more precisely, the map that forgets $\CA$ and only remembers $Z$ defines a local homomorphism from $\Stab(\CD)$ to $\Hom_\BZ(K_{\textup{num}}(\CD), \BC)$. There is an $\BR$-action 
\[\BR\times \Stab(\CD)\ni(\phi, \sigma)\mapsto \sigma[\phi]\in \Stab(\CD)\] on $\Stab(\CD)$ which translates the associated slicing of $\sigma$ by $\phi$. In particular, for $0<\phi<1$ we have
\[\CM^{\sigma[\phi]-\textup{ss}}_\v=\begin{cases}
    \CM^{\sigma-\textup{ss}}_\v & \textup{if }\pi\geq \arg(Z(\v))>\phi \pi\\
    \CM^{\sigma-\textup{ss}}_\v[1] & \textup{if }0<\arg(Z(\v))\leq \phi \pi\,.\\
\end{cases}\]
By the notation $\CM^{\sigma-\textup{ss}}_\v[1]$, we mean the same moduli stack $\CM^{\sigma-\textup{ss}}_\v$ with a universal object shifted by $[1]$.

\subsubsection{Quiver stability conditions}

Suppose that $\FE$ is an exceptional collection of $\CD$ such that its dual $\widetilde \FE$ is full and strong. Then the extension closure of $E_1, \ldots, E_n$ is an abelian category $\CA_\FE$, which under the derived equivalence in Theorem \ref{thm: beilinson} corresponds to the abelian category $\Rep_{(Q,I)}$ of representations of a quiver with relations $(Q, I)$. For any choice of $\zeta=(\zeta_1, \ldots, \zeta_n)$ with $\zeta_i \in \BH\,\setminus (-\infty,0)$, we define a central charge $Z_{\zeta}$ by setting $Z_\zeta(E_i)=\zeta_i$ and extending it additively to $K(\CD)$. We say that a stability condition $\sigma=(\CA_\FE, Z_\zeta)$ constructed in this way is a quiver stability condition. 

If $\zeta_i=-\theta_i+\sqrt{-1}\in \BH$, the slope associated to $\sigma=(\CA_\FE, Z_\zeta)$ matches precisely the slope defined for representations of quivers in Section \ref{subsubsec: stabilityquivers}, so the derived equivalence $B:\CD\xrightarrow{\sim}D^b(Q,I)$ identifies $\CM^{\sigma-\textup{ss}}_\v$ with $\CM_{d}^{\theta-\textup{ss}}(Q, I)$ where $d$ is the image of $\v$ under the isomorphism $K(\CD)\simeq K(Q, I)$. It was observed in \cite[Proposition 8.1]{abch} that, more generally, for arbitrary $\zeta_i\in \BH\,\backslash(-\infty,0)$  the equivalence $B$ still identifies $\CM^{\sigma-\textup{ss}}_\v$ with $\CM_{d}^{\theta-\textup{ss}}(Q, I)$ where $\theta$ is defined by 
\[\theta_i=-\Re(\zeta_i)+\Im(\zeta_i)\frac{\sum_{j=1}^n \Re(\zeta_j)d_j}{\sum_{j=1}^n \Im(\zeta_j)d_j}\,.\]
Note that $\theta(d)=0$ for this choice of $\theta$.

\subsubsection{Geometric stability conditions} Suppose now that $S$ is a smooth projective surface and $\CD=D^b(S)$. Recall that given an ample divisor $H$ we may define the slope of a torsion-free sheaf $F$ by
\[\mu^H(F)=\frac{c_1(F)\cdot H}{\rk(F)}\,.\]
This defines the notion of $\mu^H$-(semi)stable torsion-free sheaves. The slope of pure one-dimensional sheaf $G$ is defined as
\[\mu^H(G)=\frac{\ch_2(G)}{c_1(G)\cdot H}\,.\]
This defines the notion of $\mu^H$-(semi)stability for pure one-dimensional sheaves.

In \cite{ab} the authors construct stability conditions $\sigma_{E, H}=(\CA_{E, H}, Z_{E, H})$ where $H, E$ are $\BR$-divisors with $H$ being ample. We call these geometric stability conditions. The heart of the $t$-structure is obtained as a tilt
\[\CA_{E, H}=\langle \Coh_{H, \leq E\cdot H}[1], \Coh_{H, > E\cdot H}\rangle\]
of $\Coh(S)$ with respect to $\mu^H$-stability. Here we write $\Coh_{H, >s}$ for the extension closure of torsion sheaves and $\mu^H$-semistable torsion-free sheaves $T$ with $\mu^H(T)>s$; similarly, $\Coh_{H, \leq s}$ is the extension closure of $\mu^H$-semistable torsion-free sheaves $F$ with $\mu^H(F)\leq s$. The central charge is given by
\[Z_{E, H}(F)=-\int_S e^{-E-\sqrt{-1}\, H}\ch(F)\in \BC\,. \]
Note that these stability conditions have the property that structure sheaves of points $\CO_x$ are stable and $Z(\CO_x)=-1$. By result of \cite{toda}, see also \cite{macri_schmidt}, geometric stability condition defines an Artin stack of finite type $\CM^{\sigma_{E,H}-\textup{ss}}_\v$. Furthermore, loc. cit. proves that $\CM^{\sigma_{E,H}-\textup{ss}}_\v$ is an open substack of all perfect complexes with trivial negative Ext groups constructed in \cite{lieblich}. Therefore $\CM^{\sigma_{E,H}-\textup{ss}}_\v$ is equipped with an obstruction theory whose associated virtual tangent complex at a point $F$ is given by $\RHom_S(F,F)[1]$.

Geometric stability is closely related to so called twisted Gieseker stability which we explain now. Recall from \cite{matsukiwentworth} that the $(H,E-\frac{K}{2})$-twisted Hilbert polynomial of $F\in \Coh(S)$ is defined as
$$P_{(H,E-\frac{K}{2})}(F,t)\coloneqq\int_X \ch(F)\cdot e^{-(E-\frac{K}{2})+tH}\cdot \td(S)\in \BR[t]\,. 
$$
The choice of $-K/2$, where $K=K_S$ denotes the canonical divisor, is made so that 
$$e^{\frac{K}{2}}\cdot \td(S)=(1,0,c\,[\pt]),\quad c=\chi(\CO_S)-\textstyle\frac{K^2}{8}\in \BQ
$$
has no codimension one part. Writing $\ch^E(-)\coloneqq\ch(-)\cdot e^{-E}$, we have
$$P_{(H,E-\frac{K}{2})}(F,t)=t^2\cdot \frac{\rk(F)H^2}{2}+t\cdot (H\cdot \ch^E_1(F))+(\ch^E_2(F)+c\cdot\rk(F))\,.
$$
If $F$ is torsion-free, we define the $(H,E-\frac{K}{2})$-twisted reduced Hilbert polynomial is 
$$p_{(H,E-\frac{K}{2})}(F,t)=t^2+t\cdot \frac{H\cdot \ch^E_1(F)}{\rk(F)H^2/2}
+\frac{\ch^E_2(F)}{\rk(F)H^2/2}+\frac{2c}{H^2}\,.
$$
We say that a torsion-free sheaf $F$ is $(H,E-\frac{K}{2})$-twisted Gieseker (semi)stable if for all $0\subsetneq F'\subsetneq F$ we have
$$p_{(H,E-\frac{K}{2})}(F',t)(\leq) p_{(H,E-\frac{K}{2})}(F,t)\quad \textnormal{for all}\quad t\gg0\,.
$$
In other words, $F$ is $(H,E-\frac{K}{2})$-twisted Gieseker (semi)stable if for all $0\subsetneq F'\subsetneq F$ we have one of the followings:
\begin{enumerate}
    \item [(i)]$\displaystyle \frac{H\cdot \ch^E_1(F')}{\rk(F')}
    <\frac{H\cdot \ch^E_1(F)}{\rk(F)}$\vspace{3pt}
    \item [(ii)]$\displaystyle \frac{H\cdot \ch^E_1(F')}{\rk(F')}
    =\frac{H\cdot \ch^E_1(F)}{\rk(F)}$ \ and \ $\displaystyle \frac{\ch_2^E(F')}{\rk(F')}
    (\leq)\frac{\ch_2^E(F)}{\rk(F)}$\,.
\end{enumerate}
For 1-dimensional sheaves, $(H,E-\frac{K}{2})$-twisted Gieseker stability is equivalent to $(H,E-\frac{K}{2})$-twisted slope stability, where the twisted slope of $F$ is given by
\[\frac{\ch_2^E(F)}{H\cdot \ch_1(F)}\,.\]

The geometric stability recovers $(H,E-\frac{K}{2})$-twisted Gieseker stability
in the large volume limit. Note that $\CA_{E,H}=\CA_{E,tH}$ for any $t>0$ in the following statement. 

\begin{lemma}[\cite{BriK3}]\label{lem: large volume}
    Let $H,E$ be $\BR$-divisors with $H$ being ample and $\v\in K(S)$. Let $t\gg 0$ be large enough according to $H$, $E$ and $\v$. Let $F\in \CA_{E,H}$ of type $\v$.
    \begin{enumerate}
        \item If $\dim(\v)=1$, then 
        \[F\textup{ is }\sigma_{E, tH}\textup{-(semi)stable}\Leftrightarrow F\textup{ is }(H,E-{\textstyle\frac{K}{2}})\textup{-twisted slope (semi)stable sheaf.}\]
        \item If $\dim(\v)=2$ and we assume further that $E\cdot H<\mu^H(\v)$, then
        \[F\textup{ is }\sigma_{E, tH}\textup{-(semi)stable}\Leftrightarrow F\textup{ is }(H,E-{\textstyle\frac{K}{2}})\textup{-twisted Gieseker (semi)stable sheaf.}\]
    \end{enumerate}
\end{lemma}
\begin{proof}

    The $\dim(\v)=2$ case is essentially Proposition 14.2 of loc. cit. Even though the cited reference considers only $K3$ surfaces, the same proof applies to arbitrary surfaces once we observe the characterization of the $(H,E-{\textstyle\frac{K}{2}})$-twisted Gieseker (semi)stability in the preceding paragraphs. See also \cite[Lemma 3.5]{rekuski} where this was used. 

    Even though the $\dim(\v)=1$ case is also standard, for completeness we explain its proof by adapting the one for the $\dim(\v)=2$ case. For fixed $H, E$ and $\v$, there are only finitely many walls for the parameter $t>0$ according to \cite[Lemma 6.24]{macri_schmidt}. Therefore, we can choose $t$ large enough so that there are no higher walls. Suppose that $F$ is $\sigma_{E,tH}$-(semi)stable. By the choice of $t$, $F$ is $\sigma_{E,t'H}$-(semi)stable for all $t'\gg 0$. Thus $F$ satisfies the assumption of \cite[Lemma 6.18]{macri_schmidt}, so we conclude that $F$ is a torsion sheaf supported in degree zero\footnote{This is because the other two conclusions in \cite[Lemma 6.18]{macri_schmidt} occur only if $\dim(\v)=2$.}. If $F$ is not $(H, E-\frac{K}{2})$-twisted slope (semi)stable, then we get a destabilizing subsheaf $0\subsetneq G\subsetneq F$. In other words, 
    $$\frac{\ch^E_2(G)}{H\cdot \ch_1(G)}(\nleq)\frac{\ch^E_2(F)}{H\cdot \ch_1(F)}\,.
    $$
    Since $G$ is a torsion sheaf, it is inside the tilted heart $\CA_{E,tH}=\CA_{E,H}$. This contradicts the $\sigma_{E,t'}$-(semi)stability of $F$ because for torsion sheaves we have
    \begin{equation}\label{eq: torsion comparison}
        \mu^{\sigma_{E,t'H}}(G)=\frac{1}{t'}\cdot\frac{\ch_2^E(G)}{H\cdot \ch_1(G)},\quad \mu^{\sigma_{E,t'H}}(F)=\frac{1}{t'}\cdot\frac{\ch_2^E(F)}{H\cdot \ch_1(F)}\,.
    \end{equation}
    Thus $F$ is $(H,E-\frac{K}{2})$-twisted slope (semi)stable. 

    Conversely, assume that $F$ is $(H,E-\frac{K}{2})$-twisted slope (semi)stable. Let $t$ be chosen as before so that there are no higher walls. Suppose that $F$ is not $\sigma_{E,tH}$-(semi)stable. Let $G$ be a $\sigma_{H,tE}$-destabilizing subobject, i.e.,
    $$\mu^{\sigma_{E,tH}}(G)(\nleq) \mu^{\sigma_{E,tH}}(F)\,.
$$
Note that the cohomology sheaf $H^{-1}(G)$ is zero because $H^{-1}(F)=0$. We divide into two cases; $\rk(G)>0$ and $\rk(G)=0$. If $\rk(G)>0$, then we have 
$$\lim_{t'\rightarrow \infty} \mu^{\sigma_{E,t'H}}(G)=-\infty < 0 =\lim_{t'\rightarrow \infty}\mu^{\sigma_{E,t'H}}(F)\,.
$$
By continuity of the slope formula, there exists some $t'>t$ such that $\mu^{\sigma_{E,t'H}}(G)=\mu^{\sigma_{E,t'H}}(F)$, contradicting the fact that there are no walls above $t$. Therefore $\rk(G)=0$, i.e., $G$ is a torsion sheaf. Since torsion sheaves on $S$ form an abelian subcategory of both $\CA_{E,H}$ and $\Coh(S)$, $G$ being a subobject of $F$ in $\CA_{E, H}$ implies that $0\subsetneq G\subsetneq F$ is a subobject in $\Coh(S)$. By the computation in \eqref{eq: torsion comparison}, $G$ destabilizes $F$ with respect to the $(H,E-\frac{K}{2})$-twisted slope (semi)stability, contradicting the assumption. 
\end{proof}

\begin{remark}\label{rem: Gieseker is geometric}
Consider the moduli stack $\CM^{H-\textup{ss}}_\v$ of Gieseker $H$-semistable sheaves of type $\v$. By Lemma \ref{lem: large volume}, we can identify this moduli stack as
$$\CM^{H-\textup{ss}}_\v=\CM^{\sigma-\textup{ss}}_\v
$$
for some geometric stability condition $\sigma$. When $\dim(\v)=1$, $\mu^H$-(semi)stability is equivalent to $\sigma_{(\frac{K}{2},tH)}$-(semi)stability for $t\gg0$. Now let $\dim(\v)=2$. Pick some small enough $n$ such that $(nH+\frac{K}{2})\cdot H<\mu^H(\v)$. Since $H$ and $nH$ are linearly dependent, $(H,nH)$-twisted Gieseker (semi)stability is equivalent to the usual Gieseker (semi)stability with respect to $H$. On the other hand, the lemma says that $(H,nH)$-twisted Gieseker (semi)stability is equivalent to $\sigma_{nH+\frac{K}{2},tH}$~-~(semi)stability for some $t\gg 0$. 
\end{remark}

The following Lemma will be useful later.

\begin{lemma} \label{lem: dimension 2}
Let $S$ be a surface with nef anticanonical surface. Let $\sigma_{E, H}$ be a geometric stability condition and $\phi\in \BR$. Then the abelian category $\CA_{E, H}[\phi]$ has homological dimension 2.
\end{lemma}
\begin{proof}
    Suppose that $F_1, F_2$ are objects in $\CA_{E, H}[\phi]=\CP((\phi, \phi+1])$. It was proven in \cite[Theorem 7.7]{mozgovoy} that $\Hom(\CP(\phi_1), \CP(\phi_2))=0$ if $\phi_2-\phi_1>2$. Since $F_2[k]$ is an object in $\CP((\phi+k, \phi+k+1])$ it follows that $\Ext^k(F_1, F_2)=\Hom(F_1, F_2[k])=0$ for $k\geq 3$, showing that the homological dimension is at most 2. Since $\CO_x[m]$ belongs to $\CA_{E, H}[\phi]$ for some $m\in \BZ$ and $\Ext^2(\CO_x[m],\CO_x[m])\neq 0$, homological dimension is precisely $2$.
\end{proof}

\subsubsection{The ABCH program} Let $S$ be a surface which admits (possibly many) full and strong exceptional collections, for instance any del Pezzo surfaces \cite[Example 8.6]{Bri-St}.\footnote{Having geometric helix implies that we have a full and strong exceptional collection.} It was proposed in \cite{abch, am} that one could try to use the exceptional collections on $S$ and corresponding quiver stability conditions to study moduli spaces with respect to geometric stability conditions. Note that the quiver stability conditions that we will consider are typically not geometric; for example $\CO_{\BP^2}(-1)[2]$ appears in the exceptional collection in Example \ref{ex: exceptionalcollections} but it is not in the heart of any geometric stability condition. However, in some cases the $\BR$-action on the space of stability conditions can be used to relate geometric and quiver stability conditions.

\begin{definition}\label{def: quiver description}
    Let $\sigma$ be a geometric stability condition and $\v\in K(S)$. We say that $\CM_\v^{\sigma-\textup{ss}}$ admits a quiver description if there is another geometric stability condition $\sigma'$ and $\phi\in \BR$ such that $\sigma'[\phi]$ is a quiver stability condition and, for $F$ in class $\v$,
    \[F\textup{ is }\sigma\textup{-(semi)stable}\,\Leftrightarrow\, F\textup{ is }\sigma'\textup{-(semi)stable}\,.\]
\end{definition}
In particular, when this happens the moduli stack $\CM_\v^{\sigma-\textup{ss}}$ can be identified with a moduli stack of representations $\CM^{\theta-\textup{ss}}_d(Q, I)$ of some quiver with relations via the equivalence $D^b(S)\simeq D^b(Q,I)$ composed with a twist $[k]$. 

A careful study of the wall and chamber structure on the space of stability conditions and an identification of ``quiver regions'' associated to certain choices of exceptional collections allowed the authors to prove the following:

\begin{theorem}[\cite{abch, am}]\label{thm: ABCH for three surfaces}
    Let $S=\BP^2$, $S=\BP^1\times \BP^1$ or $S=\textup{Bl}_\pt(\BP^2)$. Then, for any geometric stability condition $\sigma$ and $\v\in K(S)$, $\CM_\v^{\sigma-\textup{ss}}$ admits a quiver description.
\end{theorem}

The authors in \cite[Conjecture 1]{am} conjecture that the same is true for any del Pezzo surfaces. We remark that in the case of $S=\BP^2$ the only exceptional collections used are the ones obtained from tensoring by a line bundle Example \ref{ex: exceptionalcollections}.

\subsection{Application to sheaf Virasoro constraints}

Let $X$ be a smooth projective variety admitting a full exceptional collection. Virasoro constraints for moduli spaces of sheaves on $X$ are formulated in \cite{blm} using a descendent algebra $\BD^X$ generated by symbols of the form $\ch_k(\gamma)$ for a cohomology class $\gamma\in H^\ast(X)$ and Virasoro operators $\bL_n\colon \BD^X\to \BD^X$; the descendent algebra admits a realization map to $H^\ast(\CN^X)$, see Definition 2.5 in \cite{blm}. On the other hand, we defined in Definition~\ref{def: ext descendent} the Ext descendent algebra $\BD^\bfT$ to be generated by symbols of the form $\ch_k^L(\alpha)$ for $\alpha\in K(X)$, and a realization map to $H^\ast(\CN^X)$. We have an isomorphism between the two descendent algebras
$$\psi:\BD^{\bfT}\rightarrow \BD^X,\quad \ch^L_k(\alpha)\mapsto \ch_k\big(\ch(\alpha^\vee)\td(X)\big).
$$

As we did for quivers in Proposition \ref{prop: comparison of two descendents}, we can show that this isomorphism commutes with realization homomorphisms and intertwines the Virasoro operators on $\BD^\bfT$ (cf. Definition \ref{def: K Virasoro}) and on $\BD^X$ (cf. \cite[Section 2.3]{blm}).\footnote{Since $X$ admits a full exceptional collections, we have $H^{p,q}(X)=0$ for all $p\neq q$ by \cite[Proposition 1.9]{marcollitabuada}. In the language of \cite{blm} this means that $\ch_k(\gamma)=\ch_k^H(\gamma)$.} Indeed, the operators both on $\BD^{\bfT}$ and $\BD^X$ are shown to be dual to the canonical Virasoro operators on the lattice vertex algebra $V^X\coloneqq H_\ast(\CN^X)=H_\ast(\CN^{\,\bfT})$, see Proposition \ref{prop: duality of operators} and \cite[Theorem 4.12]{blm}. 

Suppose we are given a coarse moduli space $M_\v^\sigma$ parametrizing sheaves or complexes on $X$ with respect to some stability condition $\sigma$ such that there are no strictly $\sigma$-semistable objects of type $\v$. Suppose further that the natural perfect obstruction theory on $M_\v^\sigma$ is 2-term, so we have a natural virtual fundamental class $[M_\v^\sigma]^\vir$. The sheaf Virasoro constraints conjectures in \cite{blm} are the statement that the class $[M_\v^\sigma]^\vir$ in $\widecheck V^X$ is a primary state.

\begin{theorem}\label{thm: ABCH implies Virasoro constraints}
    Let $S$ be a del Pezzo surface, $\sigma$ be a geometric stability condition and $\v\in K(S)$. Suppose that $\CM_\v^{\sigma-\textup{ss}}$ admits a quiver description. 
    \begin{enumerate}
        \item [(i)]  If there are no strictly $\sigma$-semistable objects of type $\v$, then the coarse moduli space $M^\sigma_\v$ together with its natural virtual class satisfies the Virasoro constraints. 
        \item [(ii)]  If $\Ext^2_S(F,F')=0$ for all $\sigma$-semistable objects $F$ and $F'$ of type $\v$, then the moduli stack $\CM_\v^{\sigma-\textup{ss}}$ is smooth and tautologically generated. Furthermore, the Virasoro operators are geometric, i.e. they descend to the cohomology ring via the realization homomorphism $\BD^S\to H^\ast(\CM_\v^{\sigma-\textup{ss}})$. 
    \end{enumerate}
\end{theorem}

\begin{proof}
Let $\CM_\v^{\sigma-\textup{ss}}$ be a moduli stack admitting a quiver description. By definition, this means that we have an exceptional collection $\FE=(E_1,\dots, E_n)$ of $D^b(S)$ and a stability condition $\theta$ for a corresponding $(Q,I)$ such that 
$$B:\CM_\v^{\sigma-\textup{ss}}\xrightarrow{\sim}{\CM}_d^{\theta-\textup{ss}}\,.
$$
This isomorphism is obtained by the equivalence $B:D^b(S)\rightarrow D^b(Q,I)$, possibly composed by a shift $[k]$. By definition of $B$, it lifts to a quasi-equivalence between the natural dg enhancements which in turn defines an quasi-isomorphism between derived stacks ${\bf B}:\bfCN^S\rightarrow \bfCN^\dgQ$. Here $\dgQ$ is the canonical dg quiver associated to $(Q,I)$ which is quasi-smooth by Lemma \ref{lem: dimension 2}.  Then we have a homotopy commutative diagram 
\begin{equation}\label{eq: comm 1}
\begin{tikzcd}
\CM_\v^{\sigma-\textup{ss}} \arrow[r, "B"] \arrow[d, hook] & {\CM}_d^{\theta-\textup{ss}} \arrow[d, hook] \\
\CN^S \arrow[r, "B"]                 & \CN^\dgQ                
\end{tikzcd}
\end{equation}
where the vertical morphisms are open embeddings. Note that this induces an obstruction theory for both $\CM_\v^{\sigma-\textup{ss}}$ and $\CM_d^{\theta-\textup{ss}}$ whose virtual tangent complexes at points $F\in \CM_\v^{\sigma-\textup{ss}}$ and $V\in \CM_d^{\theta-\textup{ss}}$ are given by 
$$\RHom_S(F, F)[1],\quad \RHom_{\dgQ}(V,V)[1]
$$
respectively. By commutativity of \eqref{eq: comm 1}, $B:\CM_\v^{\sigma-\textup{ss}}\rightarrow \CM_d^{\theta-\textup{ss}}$ also identifies the obstruction theories. By Lemma \ref{lem: dimension 2}, the complex $\RHom_S(F, F)[1]$ lies in the perfect amplitude $[-1,1]$, hence so is $\RHom_{\dgQ}(V,V)[1]$. 

To show the statement (i), assume that there are no strictly $\sigma$-semistable objects of type $\v$. Then the good moduli spaces are simply the $\BG_m$-rigidification of the moduli stacks and we have a homotopy commutative diagram  
\begin{equation}\label{eq: comm 2}
\begin{tikzcd}
M_\v^{\sigma} \arrow[r, "B"] \arrow[d, hook] & {M}_d^{\theta} \arrow[d, hook] \\
(\CN^S)^{\textup{pl}} \arrow[r, "B"]                 & (\CN^\dgQ)^{\textup{pl}}                
\end{tikzcd}
\end{equation}
with the vertical morphisms being open embeddings. The good moduli spaces are equipped with the induced truncated obstruction theory whose virtual tangent complex lies in the amplitude $[0,1]$. Pushforward of the virtual classes define the elements in the Lie algebras associated to vertex algebras
$$[M^\sigma_\v]^\vir\in \widecheck V^S,\quad [{M}^\theta_d]^\vir \in \widecheck V^\dgQ\,.
$$
By Corollary \ref{cor: intertwining Virasoro}, $B_*:V^S\rightarrow V^\dgQ$ is a vertex algebra isomorphism intertwining the Virasoro operators; hence the induced isomorphism of Lie algebras $B_\ast\colon \widecheck V^S\to \widecheck V^\bfQ$ preserves the subspaces of primary states. Since $B_*([M^\sigma_\v]^\vir)=[{M}^\theta_d]^\vir$ by commutativity of \eqref{eq: comm 2}, $M^\sigma_\v$ satisfies the Virasoro constraints with its natural virtual class if and only if ${M}^\theta_d$ does. By Theorem \ref{thm: quivervirasoro}, we know that ${M}^\theta_d$ satisfies the Virasoro constraints with respect to the virtual class defined in Section \ref{sec: explicit virtual class}. To show that this virtual class is equal to the one coming from the derived enhancement of $\CN^\dgQ$, it suffices to compare the $K$-theory class of the virtual tangent complex by the result of \cite{Sie}. By Proposition \ref{prop: extquasismooth}, the $K$-theory class coming from the derived enhancement of $\CN^\dgQ$ is
$$\CO-\sum_{i\in Q_0}\CH\textup{om}(\CV_i, \CV_i)
    +\sum_{e\in Q_1}\CH\textup{om}(\CV_{s(e)}, \CV_{t(e)})
    -\sum_{\tilde r\in Q_2}\CH\textup{om}(\CV_{s(\tilde r)}, \CV_{t(\tilde r)})\,,
$$
which matches precisely the $K$-theory class of the virtual tangent complex given by the explicit description in Section \ref{sec: explicit virtual class}, proving the statement (i).

Now we consider the statement (ii). Assume that $\Ext^2_S(F,F')=0$ for all $F,F'$ in $\CM_\v^{\sigma-\textup{ss}}$. Via the isomorphism $B:\CM_\v^{\sigma-\textup{ss}}\xrightarrow{\sim}{\CM}_d^{\theta-\textup{ss}}$ this implies the analogous statement for $(Q,I)$. Therefore $(Q,I)$, $d$ and $\theta$ satisfy Assumption \ref{ass: smoothness}. By Theorem~\ref{thm: virasoro rep}, we have a surjective realization homomorphism 
$$\xi_d:\BD^{\dgQ}_d\rightarrow H^*({\CM}_d^{\theta-\textup{ss}})
$$
through which the operators $\bR_n^{\dgQ}$ descend. On the other hand, by taking cohomology of the diagram \eqref{eq: comm 1}, we obtain a commutative diagram 
\begin{center}
\begin{tikzcd}
H^*(\CM_\v^{\sigma-\textup{ss}}) & H^*({\CM}_d^{\theta-\textup{ss}}) \arrow[l, "B^*"]  \\
\BD_\v^{S}      \arrow[u]      & \BD_d^{\dgQ} \arrow[l, "B^*"]    \arrow[u] 
\end{tikzcd}
\end{center}
where the isomorphism $B^*:\BD^{\dgQ}_d\to \BD^S_\v$ is defined as in the proof of Theorem \ref{thm: natural iso}; explicitly, it is given by
\[B^\ast \ch_k(i)=\ch_k\big(\ch(\widetilde E_i^\vee)\td(X)\big)\,,\quad i\in Q_0.\] Note that this isomorphism $B^*$ intertwines the Virasoro operators $\bR_n^S$ and $\bR_n^\dgQ$ which can be checked directly from the definition of $B^*$.\footnote{More conceptually, this stems from the fact that the descendent algebras $\BD^S, \BD^\bfQ$ and respective Virasoro operators are identified with the naturally defined Ext descendent algebra and Ext Virasoro operators, see Definitions \ref{def: ext descendent} and \ref{def: K Virasoro}.} Therefore, statements about surjectivity of the realization homomorphism and descent of $\bR_n^S$ follows from the corresponding statements for $(Q,I)$. \end{proof}

By applying the above theorem to del Pezzo surfaces where the ABCH program is proven, we obtain the following result concerning sheaf moduli spaces. 

\begin{corollary}\label{cor: three cases}
    Let $\CM_\v^{H-\textup{ss}}$ be a moduli stack of nonzero-dimensional $H$-semistable sheaves on $\BP^2$, $\BP^1\times \BP^1$ or $\textup{Bl}_\pt(\BP^2)$. 
    \begin{enumerate}
        \item [(i)] If there are no strictly $H$-semistable sheaves of type $\v$, then the coarse moduli space $M_\v^H$ satisfies the Virasoro constraints with smooth fundamental class. 
        \item [(ii)] The moduli stack $\CM^{H-\textup{ss}}_\v$ is a smooth and tautologically generated. Furthermore, the Virasoro operators descends to the cohomology ring via the realization homomorphism. 
    \end{enumerate}

\end{corollary}

\begin{proof}
Let $\CM_\v^{H-\textup{ss}}$ be a moduli stack as in the statement. By Remark \ref{rem: Gieseker is geometric} and Theorem \ref{thm: ABCH for three surfaces}, it admits a quiver description. For any positive dimensional $H$-semistable sheaves $F$ and $F'$ of type $\v$, we have 
$$\Ext^2_S(F,F')\simeq \Hom_S(F',F\otimes K_S)^\vee=0
$$
since $F'$ and $F\otimes K_S$ are $\mu^H$-semistable sheaves with $\mu^H(F')>\mu(F\otimes K_S)$ by negativity of $K_S$. This implies that the moduli stack is smooth by deformation theory. Therefore (i) and (ii) follows from the corresponding statements from Theorem~\ref{thm: ABCH implies Virasoro constraints}. 
\end{proof}
Note that part (i) of the above Theorem is also shown in \cite{blm, bojko} when $\dim(\v)=2$. The case $\dim(\v)=1$ was previously only known conditionally on some technical conditions necessary for the application of a wall-crossing formula \cite[Assumption 5.8]{blm}.

The statement about the Virasoro operators descending for the moduli spaces of one-dimensional sheaves on $\BP^2$ is applied to the study of the cohomology ring of such moduli spaces in terms of generators and relations in \cite{KLMP}.

\section{Grassmannian and symmetric polynomials}\label{sec: Grass}

In this section, we study the Virasoro constraints of Grassmannians from the perspective of symmetric functions. Previously, notation $\Gr(N,k)$ was used to denote the Grassmannian parametrizing $k$-dimensional {\it quotients} of a fixed vector space $\BC^N$. In this section, we change the notation and use $\Gr(k,N)$ to denote the Grassmannian parametrizing $k$-dimensional {\it subspaces} of $\BC^N$. We shift to this convention because Grassmannians and symmetric functions are discussed often in this way.

\subsection{Algebra of symmetric functions}

We will now recall some basic definitions regarding the algebra of symmetric functions, following \cite{macdonald}. Let 
\[\Lambda=\varprojlim_{n} \BQ[x_1, \ldots, x_n]^{S_n}\]
be the ring of symmetric functions in infinitely many variables, where the inverse limit is understood in the category of graded rings. The ring $\Lambda$ can be described as the polynomial ring in infinitely many variables in a few different ways:
\begin{enumerate}
    \item $\Lambda=\BQ[e_1, e_2, \ldots]$ where $e_j$ is the $j$-th elementary symmetric function.
    \item $\Lambda=\BQ[p_1, p_2, \ldots]$ where $p_j$ is the $j$-th power sum.
    \item $\Lambda=\BQ[h_1, h_2, \ldots]$ where $h_j$ is the $j$-th complete symmetric function.
\end{enumerate}
The different sets of generators can be written in terms of each other by Newton's identities:
\begin{align*}\sum_{j\geq 0} e_j &=\exp\left(\sum_{k\geq 1}\frac{(-1)^{k+1}}{k}p_k\right)\\
\sum_{j\geq 0} h_j &=\exp\left(\sum_{k\geq 1}\frac{1}{k}p_k\right)=\left(\sum_{j\geq 0} (-1)^j e_j\right)^{-1}\,.\end{align*}

Above, we set $e_0=h_0=1$; the identities are understood in the completion of $\Lambda$ with respect to the degree. Each set of generators defines a basis of $\Lambda$ indexed by partitions $\lambda=(\lambda_1\geq \lambda_2\geq \cdots\geq \lambda_{\ell(\lambda)})$\,:
\[e_\lambda=\prod_{i=1}^{\ell(\lambda)}e_{\lambda_i}\,, \quad p_\lambda=\prod_{i=1}^{\ell(\lambda)}p_{\lambda_i}\,, \quad h_\lambda=\prod_{i=1}^{\ell(\lambda)}h_{\lambda_i}\,.\]
Another natural basis indexed by partitions is $\{m_\lambda\}$ where 
\[m_\lambda(x_1, \ldots, x_n)=\sum_{\sigma} \prod_{i=1}^n x_i^{\sigma_i}\]
where the sum is over all the distinct permutations of $\lambda$.

The ring $\Lambda$ admits a non-degenerate bilinear pairing $\langle -,-\rangle$, called the Hall inner product, that plays a particularly important role. This pairing is defined by setting
\[\langle p_\lambda, p_\mu\rangle=\delta_{\lambda \mu}z_\lambda\,.\]
In the formula above
\[z_\lambda=\prod_{i\geq 1}i^{m_i}m_i!\]
where $m_i$ is the number of times that $i$ appears in $\lambda$. It can be checked that for every $n>0$ the multiplication operator $p_n$ and the annihilation operator $p_{-n}:=n\frac{\partial}{\partial p_n}$ are adjoint with respect to the Hall pairing, i.e., 
\[\langle p_n f, g\rangle=\langle f, p_{-n}\,g\rangle\,.\]

Schur polynomials $\{s_\lambda\}$ form yet another natural basis of $\Lambda$ indexed by partitions. They can be defined by the following property:
\begin{definition}
The Schur polynomials $\{s_\lambda\}$ are the unique orthonormal basis of $\Lambda$ with respect to $\langle -,- \rangle$  such that 
\[s_\lambda=m_\lambda+\sum_{\mu<\lambda}a_{\lambda\mu}m_\mu\]
for some $a_{\lambda\mu}\in \BQ$ where the sum runs over partitions $\mu$ which are smaller than $\lambda$ in the lexicographic order.
\end{definition}
There are several different ways to define Schur polynomials more explicitly. For example they can be written in terms of the complete symmetric functions $h_i$ by a determinant formula:
\[s_\lambda=\det\big(h_{\lambda_i-i+j}\big)_{1\leq i,j\leq \ell(\lambda)}\,.\]
Another possible definition of Schur polynomials is presented in Proposition \ref{prop: heckeproperties}.

The ring $\Lambda$ admits an involution $\sigma\colon \Lambda\to \Lambda$ defined by $\sigma(p_j)=(-1)^{j-1}p_j$. This involution sends $s_\lambda$ to $s_{\lambda^t}$ where $\lambda^t$ is the conjugate partition.

\begin{example}\label{ex: schur}
Elementary symmetric functions and complete symmetric functions are both easily written as Schur functions:
\[e_j=s_{(1)^j}\,,\quad h_j=s_{(j)}\,.\]
We also have
\[s_{(2,2)}=\frac{1}{12}p_1^4 + \frac{1}{4}p_2^2 - \frac 13 p_1 p_3\,.\]
\end{example}

\subsection{Vertex algebra from symmetric functions}\label{sec: vertex, symmetric}

Let $\Lambda[Q,q^{\pm 1}]$ be the extension of $\Lambda$ by the algebra $\BQ[Q,q^{\pm 1}]$ associated to the monoid $\BN\times \BZ$. The Hall pairing extends to this vector space
$$\Lambda^*
[ Q,q^{\pm 1}]\otimes \Lambda_*[Q,q^{\pm 1}]\xrightarrow{\langle-,-\rangle}\BQ\,.
$$
The additional subscript and superscript of $*$ are just to indicate their relation to the cohomology and homology as we discuss next.

Recall the Grassmannian vertex algebra $V^\Gr=H_*(\CN^{A_1,\textup{fr}})$ from Example \ref{ex: grassva}.\footnote{We need to use the dual convention for the framed vertex algebra since we changed the convention of Grassmannian in this section.} Then, we have a  diagram similar to \eqref{eq: construction diagram}
\begin{equation}\label{eq: grassva and symmetric functions}
    \begin{tikzcd}
  H^*(\CN_{k\rightarrow N})   \arrow[r,phantom,"\otimes" description]       &[-2em] H_*(\CN_{k\rightarrow N}) \arrow[rr,""] \arrow[d, "\xi^\dagger"] & \ & \BQ \arrow[d, equal] \\
 Q^Nq^k\cdot \Lambda^*  \arrow[u,"\xi"]\arrow[r,phantom,"\otimes" description]    &[-2em]   Q^Nq^k\cdot \Lambda_* \arrow[rr,""]   &    \   & \BQ              
\end{tikzcd}
\end{equation}
where $\xi$ is a ring isomorphism sending $p_n$ to $n!\ch_n(\CV)$ and $\xi^\dagger$ is induced from the perfect pairings defined by topological pairing on the top and Hall pairing on the bottom. The resulting isomorphism $\xi^\dagger:V^{\Gr}\xrightarrow{\sim}\Lambda_*[Q,q^{\pm 1}]$ endows $\Lambda_*[Q,q^{\pm 1}]$ with a vertex algebra structure such that it becomes a vertex subalgebra of $\textup{VA}(\BZ\times\BZ,\chi^\sym_{\Gr})$ where 
$$\chi^\sym_{\Gr}\big((k_1,N_1),(k_2,N_2)\big)=2k_1k_2 - k_1N_2-k_2N_1\,. 
$$

Using the ambient lattice vertex algebra $\textup{VA}(\BZ\times\BZ,\chi^\sym_{\Gr})$, we can explicitly write down fields of $\Lambda_*[Q,q^{\pm 1}]$. Here, we record two of the fields acting on the component $Q^Nq^k\cdot \Lambda_*$ that are relevant to us:
\begin{align*}
    Y(p_1, z)&\coloneqq\sum_{n>0}p_{n}z^{-1+n}+2kz^{-1}+\sum_{n>0}2 p_{-n}z^{-1-n}\,,\\
    Y(q,z)&\coloneqq e^q\cdot \exp\left(\sum_{j>0}\frac{p_{j}}{j}z^j\right)\exp\left(-\sum_{j>0}\frac{2 p_{-j}}{j}z^{-j}\right)\,.
\end{align*}
We warn the reader that the annihilation operators in this vertex algebra are $2p_{-n}$ rather than $p_{-n}$ for $n>0$ since $\chi^\sym_\Gr((1,0),(1,0))=2$.

\subsection{Schubert calculus}

The descendent algebra $\BD^{A_1}_k$ of the quiver with only one node $A_1$ can be naturally identified with the algebra of symmetric functions $\Lambda^\ast$:
\begin{equation}\label{eq: symmetric functions and descendent}
\begin{tikzcd}[row sep=0cm]
Q^Nq^k\cdot \Lambda^\ast \arrow[r] & \BD^{A_1}_k \arrow[r] & H^\ast(\CN_{k\rightarrow N})\\
p_n \arrow[r, mapsto]& n! \ch_n \arrow[r, mapsto]& n!\ch_n(\CV)
\end{tikzcd}
\end{equation}
where $\CV$ is the universal rank $k$ complex over $\CN_{k\rightarrow N}=\CN_k$. Note that the composition of these maps is precisely $\xi$ in \eqref{eq: grassva and symmetric functions}. Similarly, the elementary symmetric functions $e_j$ are sent to the Chern classes of $\CV$ and the signed complete symmetric functions $(-1)^j h_j$ are sent to the Segre classes of $\CV$. 

Let $\BF$ be the rank $k$ universal subbundle of $\Gr(k,N)$. The geometric realization map for the Grassmannian
\[\xi_\BF\colon \Lambda^\ast\to H^\ast(\Gr(k,N))\]
sends by definition $e_j$ to $c_j(\BF)$. We recall that the cohomology ring of the Grassmannian is 
\[H^\ast(\Gr(k,N))\cong \BQ[c_1, \ldots, c_k]/I\]
where $I$ is the ideal generated by
\[\big[(1+c_1+c_2+\ldots+c_k)^{-1}\big]_j\,,\quad \textup{for }j>N-k\]
where $[-]_j$ denotes the degree $j$ part. The (Poincaré duals of) Schubert cycles
\[S_\lambda\in H^{\ast}(\Gr(k,N))\,,\quad \lambda\subseteq (N-k)^k\]
form a linear basis of the cohomology $H^\ast(\Gr(k,N))$. By Pieri's formulas, the map $\Lambda^\ast\to H^\ast(\Gr(k,N))$ sends the Schur polynomial $s_\lambda$ to the corresponding Schubert cycle $(-1)^{|\lambda|}S_\lambda$. In particular, $s_\lambda$ is mapped to 0 whenever $\lambda$ is not contained in the rectangular shape $(N-k)^k$. As an easy consequence of this fact, we can identify the class of the Grassmannian in the Grassmannian vertex algebra:

\begin{proposition}\label{prop: grassschur}
The class of the Grassmannian $\Gr(k,N)$ is given by
\[[\Gr(k,N)]=Q^N q^k \otimes (-1)^{k(N-k)}s_{(N-k)^k}\in \Lambda_\ast[Q, q^{\pm 1}]\,.\]
\end{proposition}
\begin{proof}
Write $[\Gr(k,N)]=Q^N q^k \otimes g_{N, k}$ for some $g_{N, k}\in \Lambda_\ast$. Let $\lambda$ be a partition and consider the associated Schur polynomial $s_\lambda \in \Lambda^\ast$. By the Schubert calculus considerations above we have
\[\langle g_{N, k}, s_\lambda\rangle=\int_{\Gr(k,N)} (-1)^{|\lambda|}S_\lambda=\begin{cases}(-1)^{k(N-k)}&\textup{if }\lambda=(N-k)^k\\
0&\textup{otherwise}\end{cases}\]
since $S_\lambda=0$ for any other partition of size at least $k(N-k)$. Since $\{s_\lambda\}_\lambda$ form an orthonormal basis for the Hall inner product it follows that 
\[g_{N, k}=(-1)^{k(N-k)}s_{(N-k)^k}\,. \qedhere\]
\end{proof}

\subsection{Virasoro constraints for the Grassmannian}

In this section we will write down the Virasoro constraints for the Grassmannian using the language of symmetric functions explained in the previous section. 

Recall from Section \ref{sec: vertex, symmetric} that $\Lambda_*[Q,q^{\pm 1}]$ is a vertex subalgebra of $\textup{VA}(\BZ\times\BZ,\chi^\sym_\Gr)$. Since the latter is a lattice vertex algebra associated to the nondegenerate pairing $\chi^\sym_\Gr$, it is equipped with the natural conformal element and the corresponding Virasoro operators. It is easy to check that the half of the Virasoro operators $\{L_n\}_{\geq -1}$ preserve the vertex subalgebra $\Lambda_*[Q,q^{\pm 1}]$ since $\BZ\times\{0\}\subset \BZ\times \BZ$ is the sublattice. Explicitly, the Virasoro operators on each component $Q^Nq^k\cdot \Lambda_*$ are given by\footnote{If we set $N=0$ in the formula, we recover precisely half of the Virasoro operators on $\textup{VA}(\BZ,2)$. }
\begin{equation}\label{eq: operatorsgrva}
L_n=\begin{cases}
\displaystyle\sum_{j\geq 1}p_{j}p_{-n-j}+\sum_{\substack{a+b=n\\a,b>0}}p_{-a} p_{-b}+(2k-N)p_{-n}& n>0\\
\displaystyle\sum_{j>0}p_{j}p_{-j}+k(k-N)\cdot\textup{id} & n=0
\end{cases}
\end{equation}
with $L_{-1}$ being the translation operator $T$. By using the Hall pairing, we obtain the dual Virasoro operators acting on $Q^Nq^k\cdot \Lambda^*$:
\begin{equation}\label{eq: operatorsgrva dual}
\bL_n:=(L_n)^\dagger=\begin{cases}
\displaystyle\sum_{j\geq 1}p_{n+j}p_{-j}+\sum_{\substack{a+b=n\\a,b>0}}p_{a} p_{b}+(2k-N)p_{n}& n>0\\
\displaystyle\sum_{j>0}p_{j}p_{-j}+k(k-N)\cdot\textup{id} & n=0
\end{cases}
\end{equation}
Under the identification $Q^Nq^k\cdot \Lambda^*\simeq \BD^{A_1}_k$ in \eqref{eq: symmetric functions and descendent}, these dual Virasoro operators agree with the framed Virasoro operators $\bL_n^{*\rightarrow N}$ defined in Section \ref{subsubsec: framedvirasoro}.\footnote{Here, we are again using a different convention for the framing as we do in Section \ref{sec: Grass}. } Therefore, the Virasoro constraints for the Grassmannian $\Gr(k,N)$ are equivalent to its class in $\Lambda_*[Q,q^{\pm 1}]$ being annihilated by $L_n$ operators for $n\geq 0$.

By Proposition \ref{prop: grassschur} and Hall pairing duality between $\bL_n^{*\rightarrow N}$ and  $L_n$, the Virasoro constraints of $\Gr(k,N)$ amount to the following identity: for $n>0$, we have\footnote{The $n=0$ case is a simple consequence of the dimension formula $\dim\Gr(k,N)=k(N-k)$. }
\begin{equation}
\label{eq: virasorograssexplicit}
\left(\sum_{j\geq 1}(n+j)p_{j}\frac{\partial}{\partial p_{n+j}}+\sum_{\substack{a+b=n\\a,b>0}}ab\frac{\partial }{\partial p_a}\frac{\partial }{\partial p_b}+(2k-N)n\frac{\partial }{\partial p_n}\right)s_{(N-k)^k}=0\,.
\end{equation}
This equation is a consequence of Theorem \ref{thm: quivervirasoro}, but it is also a well-known fact about the representation theory of the Virasoro Lie algebra \cite{mimachiyamada}, see Theorem \ref{thm: singularvectorsjack}. For completeness we will also give a self-contained proof in Proposition \ref{prop: virasorogr} using Hecke correspondences for Grassmannians.

For the Grassmannian, it happens that the Virasoro constraints fully determine the descendent integrals (up to scalar). This can be thought of as the analog that the Virasoro constraints for the Gromov--Witten of a point (equivalent to Witten's conjecture together with the string equation) determine all the integrals of $\psi$ classes. This is again well-known in the representation theory of the Virasoro Lie algebra, see the uniqueness part of Theorem \ref{thm: singularvectorsjack}.

\begin{proposition}\label{prop: uniquenessgrass}
Suppose that $Q^N q^k \otimes f\in \Lambda_\ast[Q, q^{\pm 1}]$ is a primary state in the Grassmannian vertex algebra, i.e. $L_n(Q^N q^k \otimes f)=0$ for $n\geq 0$ for the operators defined in~\eqref{eq: operatorsgrva}. Then $f$ is proportional to $s_{(N-k)^k}$.
\end{proposition}
\begin{proof} 
We start by noting that the $L_0$ equation is equivalent to $f$ being homogeneous of degree $d\coloneqq k(N-k)$. Since we have already showed that $s_{(N-k)^k}$ satisfies the Virasoro constraints $L_n(Q^N q^k \otimes f)=0$ it is enough to show that the Virasoro constraints determine $f$ after one fixes $\langle p_1^d, f\rangle$. Consider the partial order of the set of partitions defined by
\[\lambda'\prec \lambda\textup{ if and only if }\ell(\lambda')>\ell(\lambda)\textup{, or }\ell(\lambda')=\ell(\lambda)\textup{ and }m_1(\lambda')>m_1(\lambda)\]
where $m_1(\lambda)$ denotes the number of times that 1 appears in $\lambda$. Note that the partition $1^d$ is smaller than any other partition of $d$ with respect to $\prec$. We claim that if $f$ satisfies the Virasoro constraints then for any partition $\lambda\neq 1^d$ we can write $\langle p_\lambda, f\rangle$ as a linear combination of $\langle p_{\lambda'}, f\rangle$ for $\lambda'\prec \lambda$.

Let $\lambda=(\lambda_1, \lambda_2, \ldots, \lambda_\ell)$ with $\lambda_i$ non-decreasing. Let $m=m_1(\lambda)\geq 0$. Assuming that $\lambda\neq 1^d$ we have $m<\ell$ so let $t=\lambda_{m+1}>1$. Define the partition $\widetilde \lambda$ of $d-t+1$ by
\[\widetilde \lambda=(\underbrace{1, \ldots, 1, 1}_{m+1}, \lambda_{m+2}, \ldots, \lambda_\ell)\,.\]
By the Virasoro constraints for $f$ we have
\begin{equation}\label{eq: schubertfromvirasoro}0=\langle \bL_{t-1}(p_{\widetilde \lambda}),f\rangle=(m+1)\langle p_\lambda, f\rangle+\sum_{\lambda'\prec \lambda}C_{\lambda'}\langle p_{\lambda'}, f\rangle\,
\end{equation}
for explicit constants $C_{\lambda'}\in \BZ$. It is straightforward to see that only partitions $\lambda'\prec \lambda$ appear: all the partitions coming from the linear part of $\bL_{t-1}$ have bigger length and all the remaining partitions coming from the derivation part have $m_1(\lambda')=m+1>m=m_1(\lambda)$. It follows by induction with respect to the order $\prec$ that $\langle p_\lambda, f\rangle$ can be determined from $\langle p_1^d, f\rangle$, proving the proposition.
\end{proof}

\begin{example}
We illustrate the algorithm used in the proof of Proposition \ref{prop: uniquenessgrass} by computing the integrals of descendents in $\Gr(4,2)$. We have
\begin{align*}
0&=\int_{\Gr(4,2)}\bL_1(p_1^3)=3\int_{\Gr(4,2)}p_1^2 p_2\\
0&=\int_{\Gr(4,2)}\bL_2(p_1^2)=2\int_{\Gr(4,2)}p_1p_3 +\int_{\Gr(4,2)} p_1^4\\
0&=\int_{\Gr(4,2)}\bL_1(p_1p_2)=\int_{\Gr(4,2)} p_2^2+2\int_{\Gr(4,2)}p_1p_3\\
0&=\int_{\Gr(4,2)}\bL_3(p_1)=\int_{\Gr(4,2)}p_4+2\int_{\Gr(4,2)} p_1^2 p_2\,.
\end{align*}
In each equality, the first term of the right hand side is the leading term in the sense of the proof of the proposition and the equalities can be used to recursively determine these leading terms. By further using that $\int_{\Gr(4,2)}c_2^2=1$ \cite[Corollary 4.2]{EH} or $\int_{\Gr(4,2)}p_1^4=2$ \cite[Exercise 4.38]{EH}  we determine all the integrals of descendents:
\[\int_{\Gr(4,2)}p_1^4=2=\int_{\Gr(4,2)}p_2^2\,,\quad \int_{\Gr(4,2)}p_1p_3=-1\,,\quad \int_{\Gr(4,2)}p_4=0=\int_{\Gr(4,2)}p_1^2 p_2\,.\]
Note that these can be obtained from the expression for $s_{(2,2)}$ in Example \ref{ex: schur}.
\end{example}

\subsection{Hecke operators}
We define Hecke operators $H_n\colon \Lambda\to \Lambda$ as the following vertex operators:
\[H(z)=\sum_{n\in \BZ}H_n z^n= \exp\left(\sum_{j>0}\frac{p_{j}}{j}z^j\right)\exp\left(-\sum_{j>0}\frac{p_{-j}}{j}z^{-j}\right)\,.\]
This corresponds to the field $Y(q,z)$ of the lattice vertex algebra associated to $(\BZ,1)$ via isomorphism $\textup{VA}(\BZ,1)\simeq \Lambda_*[q^{\pm 1}]$. By Newton's identities we can rewrite the Hecke operators as
\[H_n=\sum_{j\geq 0}(-1)^j h_{j+n}\circ e_j^\perp\] 
where $(-)^\perp$ denotes the adjoint operator with respect to the Hall pairing. Such operators have been studied in depth from the point of view of the combinatorics of $\Lambda$, see for instance \cite{jing}. In the literature they are often called Bernstein operators or simply vertex operators; we have opted to call them Hecke operators due to their geometric meaning that we explain below. We summarize some of the well-known properties of these operators in the following proposition:
\begin{proposition}[{\cite{jing}}]\label{prop: heckeproperties}
The Hecke operators satisfy the following properties:
\begin{enumerate}
\item We have
\[[H_n, p_{m}]=-H_{n+m} \,,\quad\textup{for every }n\in \BZ,\, m\in \BZ\setminus \{0\}\,.\] 
\item The adjoint to a Hecke operator is given by
\[H_n^\perp=(-1)^n \sigma \circ H_{-n}\circ \sigma\,.\]
\item The Hecke operators satisfy the commuting relations
\[H_nH_m=-H_{m-1}H_{n+1}\,.\]
In particular $H_{n}H_{n+1}=0$.
\item The Schur polynomial $s_\lambda$ can be given in terms of Hecke operators by
\[s_\lambda=H_{\lambda_1}H_{\lambda_2}\ldots H_{\lambda_{\ell}}1.\]
\end{enumerate}
\end{proposition}
\begin{proof}
Parts (2), (3) and (4) are (2.5), (2.12) and (3.9) in \cite{jing} in the limit $t\rightarrow 0$. Part (1) is the defining property of vertex operators, see for example (5.4.3b) in \cite{Ka98}.\qedhere
\end{proof}
Note that the expression for Schur polynomials given by $(4)$ makes sense for any tuple $\lambda=(\lambda_1, \ldots, \lambda_\ell)$ of (not necessarily positive) integers. The commuting relations (3) can be used to reduce the Schur polynomial associated to any tuple to the Schur polynomial of a partitions.

These Hecke operators play a role in the following geometric setup. Consider the Grassmannians $\Gr(k,N)$ and $\Gr(k+1,N)$ and their respective universal bundles $\BF_k\subseteq \CO^N$, $\BF_{k+1}\subseteq \CO^N$. The flag variety $\Flag(k, k+1; N)$ is a projective bundle over both Grassmannians:
\begin{equation}\label{eq: heckeprojectivebundles}
\begin{tikzcd}
\BP(\CO^N/\BF_k)\arrow[d]& \Flag(k, k+1; N)\arrow[l, equal]\arrow[ld, "\pi_k"]\arrow[r, equal]\arrow[rd,"\pi_{k+1}"'] &
\BP(\BF_{k+1}^\vee)\arrow[d]\\
\Gr(k,N) & & \Gr(k+1,N)
\end{tikzcd}
\end{equation}
The universal flag on $\Flag(k, k+1; N)$ is 
\[\BF_k\subseteq \BF_{k+1}\subseteq \CO^N\]
where we omit the pullbacks via $\pi_k, \pi_{k+1}$. The line bundle $\BF_{k+1}/\BF_k$ can be interpreted in each projective bundle description as
\[\BF_{k+1}/\BF_k=\CO_{\BP(\CO^N/\BF_k)}(-1)=\CO_{\BP(\BF_{k+1}^\vee)}(1)\,.\] 
Let $\zeta=c_1(\BF_{k+1}/\BF_k)\in H^2(\Flag(k, k+1; N))$. We define \textit{geometric} Hecke operators by
\begin{align*}\widetilde H_\ell\colon H^\ast(\Gr(k,N))&\longrightarrow H^\ast(\Gr(k+1,N))\\
\gamma&\longmapsto (\pi_{k+1})_\ast\big( \zeta^\ell\cdot \pi_k^\ast \gamma\big)\,.
\end{align*}
The connection between these geometric Hecke operators and the previously defined Hecke operators on $\Lambda$ is explained in the next proposition:
\begin{proposition}\label{prop: heckegeom}
For every $\ell\geq 0$ the following diagram commutes:
\begin{center}
\begin{tikzcd}
\Lambda^\ast \arrow[d, two heads]\arrow[r, "H_{\ell-k}"]& \Lambda^\ast \arrow[d, two heads]\\
H^\ast(\Gr(k,N))\arrow[r, "\widetilde H_\ell"] & H^\ast(\Gr(k+1,N))\,.
\end{tikzcd}
\end{center}
\end{proposition}
\begin{proof}
The proposition is easily proved by comparing the two paths applied to $1\in \Lambda^\ast$ and by showing that $\widetilde H_\ell$ satisfies commutation relations with $p_a$ similar to the ones in Proposition \ref{prop: heckeproperties}(1). Indeed, we have by the projective bundle formula
\[\widetilde H_\ell(1)=(\pi_{k+1})_\ast(\zeta^\ell)=s_{\ell-k}(\BF_{k+1}^\vee)=h_{\ell-k}(\BF_{k+1})\,.\]
On the other hand, $H_{\ell-k}(1)=h_{\ell-k}$. By the push-pull formula we have
\[\widetilde H_\ell\big(p_a(\BF_k)\gamma\big)=(\pi_{k+1})_\ast \big((p_a(\BF_{k+1})-\zeta^a)\cdot \zeta^\ell\cdot \pi_{k}^\ast \gamma)=p_a(\BF_{k+1})\widetilde H_\ell(\gamma)-\widetilde H_{\ell+a}(\gamma)\,.\]
Comparing with Proposition \ref{prop: heckeproperties}(1) and using induction the result follows.\qedhere
\end{proof}

The geometric picture with the Hecke operators allows us to give a second proof of Proposition \ref{prop: grassschur}.

\begin{proof}[Proof of Proposition \ref{prop: grassschur} via Hecke operators]
Denote by $g_{N, k}\in \Lambda_\ast$ the symmetric function such that $[\Gr(k,N)]=Q^N q^k\otimes g_{N,k}$. Let $f\in \Lambda^\ast$. Using the projective bundles in \eqref{eq: heckeprojectivebundles} and Proposition \ref{prop: heckegeom} we find that
\begin{align*}
\langle g_{N, k}, f\rangle&=\int_{\Gr(k,N)}f=\int_{\Flag(k, k+1; N)}(-\zeta)^{N-k-1} \pi_k^\ast f\\
&=(-1)^{N-k-1}\int_{\Gr(k+1,N)}H_{N-2k-1}(f)=(-1)^{N-k-1}\langle g_{N, k+1}, H_{N-2k-1}(f)\rangle\\
&=(-1)^{k}\langle \sigma H_{2k+1-N}\sigma (g_{N, k+1}), f\rangle\,.
\end{align*}
In the last step we used \ref{prop: heckeproperties}(2). Thus we conclude that \[g_{N, k}=(-1)^{k}\sigma H_{2k+1-N}\sigma (g_{N, k+1})\,.\]
Using $g_{N, N}=1$ we can recursively obtain
\begin{equation}\label{eq: sign1}g_{N, k}=(-1)^{\sum_{j=k}^{N-1} j}\sigma H_{2k+1-N}H_{2k+3-N}\ldots H_{N-3}H_{N-1}(1)\,.
\end{equation}
Using Proposition \ref{prop: heckeproperties} (3) iteratively we find that 
\begin{equation}\label{eq: sign2}H_{2k+1-N}H_{2k+3-N}\ldots H_{N-3}H_{N-1}=(-1)^{\sum_{i=0}^{N-k-1} i}H_{k}H_{k}\ldots H_{k}
\end{equation}
Note that the signs of \eqref{eq: sign1} and \eqref{eq: sign2} combine to
\[\sum_{j=k}^{N-1} j+\sum_{i=0}^{N-k} i=\sum_{i=0}^{N-k-1}(2i+k)\equiv k(N-k) \mod 2\,.\]
Thus, putting  \eqref{eq: sign1} and \eqref{eq: sign2} together gives
\[g_{N, k}=(-1)^{k(N-k)}\sigma H_{k}H_{k}\ldots H_{k}(1)=(-1)^{k(N-k)}\sigma(s_{k^{N-k}})=(-1)^{k(N-k)}s_{(N-k)^k}\,.\qedhere\]
\end{proof}

We take the opportunity to give a direct and elementary proof of the Virasoro constraints for the Grassmannian based on these Hecke operators. For that, we need the commutator between the Virasoro and Hecke operators. In the following two propositions, we use the operators $\bL_n$ and $L_n$ from \eqref{eq: operatorsgrva} and \eqref{eq: operatorsgrva dual} with $k=N=0$, i.e.,
\begin{equation*}
L_n=\begin{cases}
\displaystyle\sum_{j\geq 1}p_{j}p_{-n-j}+\sum_{\substack{a+b=n\\a,b>0}}p_{-a} p_{-b}& n>0\\
\displaystyle\sum_{j>0}p_{j}p_{-j} & n=0
\end{cases}
\end{equation*}
and $\bL_n=(L_n)^\dagger$.

\begin{proposition}\label{prop: commvirasorohecke}
We have the following identities in $\End(\Lambda)$:
\begin{align*}
[\bL_n, H_m]&=(m+n-1)H_{m+n}+\sum_{j=1}^{n-1} H_{n+m-j}\circ p_j-H_m\circ p_n\\
&=(m+1)H_{n+m}+\sum_{j=1}^{n-1} p_j\circ H_{n+m-j}-p_n\circ H_m\,,\\
[L_n, H_m]&=(m-1)H_{m-n}+\sum_{j=1}^{n-1} H_{m-n+j}\circ p_{-j}-H_m\circ p_{-n}\\
&=(m-n+1)H_{m-n}+\sum_{j=1}^{n-1} p_{-j}\circ H_{m-n+j}-p_{-n}\circ H_m\,.
\end{align*}
\end{proposition}
\begin{proof}
Proofs can be obtained directly from Proposition \ref{prop: heckeproperties}(1), but see also the Appendix in \cite{liuyang} (note that the operators in loc. cit. differ from ours by a factor of $1/2$ in $\bT_n$, but since they analyse $\bR_n$ and $\bT_n$ separately we can deduce the identities here from their work).\qedhere
\end{proof}

We use this to give a direct proof of \eqref{eq: virasorograssexplicit}. 

\begin{proposition}\label{prop: virasorogr}
    Equation \eqref{eq: virasorograssexplicit} holds, i.e. for every $m, k, n>0$ we have
    \[L_n(s_{m^k})=(m-k)p_{-n}s_{m^k}\,.\]
\end{proposition}
\begin{proof}
    We prove the statement by induction on $k$. We leave $k=1$ for the reader to check. Suppose that the equation holds for some $k$. From Proposition \ref{prop: commvirasorohecke} we have
    \[[L_n, H_m]-[L_{n-1}, H_{m-1}]=2p_{-n+1}\circ H_{m-1}-p_{-n}\circ H_m\,.\]
    Note that we have $H_{m}s_{m^k}=s_{m^{k+1}}$ and $H_{m-1}s_{m^k}=0$; the latter holds due to Proposition \ref{prop: heckeproperties}(3). 
\end{proof}

\subsection{Symmetrized Hecke operators and wall-crossing}\label{sec: sym Heke}

Note that the Hecke operators used in the previous section appear naturally as fields associated to the state $q\in \Lambda[q^{\pm 1}]$ when we give $\Lambda_\ast[q^{\pm 1}]$ the vertex algebra structure associated to the lattice $(\BZ, 1)$. However, the Grassmannian vertex algebra $\Lambda_\ast[q^{\pm  1}, Q]$ is a sub-algebra of the lattice vertex algebra associated to $(\BZ\times \BZ, \chi^{\sym}_\Gr)$ where 
\[\chi^{\sym}_\Gr((N, k), (N', k'))=2kk'-(Nk'+N'k)\,.\]
In particular, the subalgebra $\Lambda[q^{\pm 1}]\subseteq V^\Gr$ is the lattice vertex algebra associated to $(\BZ, 2)$. Thus wall-crossing formulas on $V^\Gr$ are in principle not written using the previous Hecke operators, but rather
\[H^\sym(z)=\sum_{n\in \BZ}H_n^\sym z^n= \exp\left(\sum_{j>0}\frac{p_j}{j}z^j\right)\exp\left(-\sum_{j>0}\frac{2p_{-j}}{j}z^{-j}\right).\]
More precisely, we have
\[Y(q\otimes 1, z)Q^N q^k\otimes f=(-1)^{N-k}z^{2k-N}H^\sym(z)f\,.\]

Denoting by $[u, v]=u_{(0)}v$ the 0-th product in $V^\Gr=\Lambda_\ast[Q, q^{\pm 1}]$ we have in particular
\begin{equation}\label{eq: heckesymwc}
[q\otimes 1, Q^N q^k\otimes f]=(-1)^{N-k} Q^N q^{k+1}\otimes H_{N-2k-1}^\sym f
\end{equation}

In \cite[Theorem 3.10]{CJ}, Cai and Jing prove a formula for Schur polynomials of rectangular shapes in terms of the vertex operators $H^\sym_n$.\footnote{Their theorem is more general, it applies to Jack polynomials and to almost rectangular shapes, i.e. $\lambda=(m,m,\ldots, m, m-1)$.} Their proof is entirely combinatorial. We now use the wall-crossing formula from Example \ref{ex: wcgrassmannian} to give a new proof of this formula using the Grassmannian.

\begin{proposition}[{\cite[Theorem 3.10]{CJ}}]Let $m, k>0$. We have the following identity in $\Lambda$:
    \begin{equation}\label{eq: sign4}H^\sym_{m-k+1}H^\sym_{m-k+3}\ldots H^\sym_{m+k-3}H^\sym_{m+k-1}(1)=(-1)^{\binom{k}{2}} k!s_{(m)^k}\,.\end{equation}
\end{proposition}
\begin{proof}Let $N=m+k$. We show this by comparing Proposition \ref{prop: grassschur} with the wall-crossing formula in Example \ref{ex: grassva}. On one hand, Proposition \ref{prop: grassschur} gives the formula for the class of the Grassmannian
\[[\Gr(k,N)]=Q^N q^k\otimes (-1)^{k(N-k)}s_{(N-k)^k}\]
We recall the wall-crossing formula from Example \ref{ex: grassva}:
\[[\Gr(k,N)]=\frac{1}{k!}[q,\ldots [q,[q,Q^N]]\ldots ]\,.\]
Applying \eqref{eq: heckesymwc} repeatedly gives us
\begin{equation}\label{eq: sign3}[\Gr(k,N)]=Q^N q^k\otimes \frac{(-1)^{\sum_{j=N-k+1}^{N}j}}{k!}H^\sym_{N-2k+1}H^\sym_{N-2k+3}\ldots H^\sym_{N-3}H^\sym_{N-1}1\,.\end{equation}
The conclusion follows from comparing the two expressions and a sign analysis similar to the one in the previous section.\qedhere
\end{proof}

\begin{remark}
The identity \eqref{eq: sign4} seems quite remarkable to us. In general there is no way to write Schur functions $s_\lambda$ using the vertex operators $H^\sym$, and \eqref{eq: sign4} is a special phenomenon of rectangular shapes $\lambda$. On the other hand, from the point of view of Joyce's wall-crossing, it is surprising that the class of the Grassmannian can be expressed using the Hecke operators $H_m$. We speculate that this phenomenon might be an indication that it is possible to write new wall-crossing formulas using a version of Joyce's vertex algebra in which the we do not consider a symmetrized complex $\Theta=\Ext^\vee+\sigma^\ast \Ext$, but rather take $\Theta=\Ext$. This variation is studied in \cite{latyntsev}: it is no longer a vertex algebra, but rather a braided vertex algebra (in particular it is a field algebra, also known as non-local vertex algebra), and it has connections to Cohomological Hall algebras.
\end{remark}

The Virasoro operators $L_n$ have simpler commutators with $H_m^\sym$ than with $H_m$. This is expected due to the fact that both $L_n$ and $H_m^\sym$ are obtained from fields in the vertex algebra associated to the lattice $(\BZ, 2)$.

\begin{proposition}
    We have the identity in $\End(\Lambda_\ast)$:
    \[[L_n, H_m^\sym]=(m+1)H^\sym_{m-n}-2  p_{-n}\circ H_m^\sym \,.\]
\end{proposition}
\begin{proof}
A proof can be given as in \cite{liuyang}. Alternatively, we can use the identity \cite[(2.7.1)]{Ka98} for the vertex algebra associated to the lattice $(\BZ, 2)$.
\end{proof}

\subsection{Geometricity of Virasoro and Calogero-Sutherland operator}\phantom{.}\\
We showed in Theorem \ref{thm: virasoro rep} that the Virasoro constraints imply that the operators $\bR_n$ descend to cohomology or, equivalently, they preserve the ideal of relations. We now illustrate this phenomena in the case of the Grassmannian; it turns out that showing geometricity directly in examples is often much easier than showing the full constraints.

The derivation $\bR_n$ is a priori defined in the descendent algebra $\Lambda^\ast$. The ideal of relations of the Grassmannian is the ideal $I$ generated by $e_{k+1}, e_{k+2}, \ldots$ and $h_{N-k+1}, h_{N-k+2}, \ldots$. Equivalently, $I$ is the linear span of the Schur polynomials $s_\lambda$ with $\lambda$ not contained in $(N-k)^k$. We now compute the action of $\bR_n$ on the generator $e_j$, for $j>k$, of the ideal. Denote by $e_\bullet$ the formal sum $\sum_{j\geq 0}e_j$. We have
\[\bR_n e_\bullet=\bR_n\exp\left(\sum_{\ell\geq 1}(-1)^{\ell-1}\frac{p_\ell}{\ell}\right)=\left(\sum_{\ell\geq 1}(-1)^{\ell-1} p_{\ell+n}\right)e_\bullet\,.\]
Hence,
\begin{align*}
    \bR_n e_j=\sum_{\ell=1}^j (-1)^{\ell-1}p_{n+\ell} e_{j-\ell}=(-1)^{n}(j+n)e_{j+n}+\sum_{s=0}^{n-1}(-1)^s e_{j+s}p_{n-s}
\end{align*}
by Newton's identity, so we conclude that $\bR_n(e_j)\in I$ for $j>k$. A similar computation shows the same for the remaining generators.

\subsubsection{Calogero-Sutherland operator}

The geometricity of the Virasoro operators can be explained at once by the geometricity of the Calogero-Sutherland operator (also Laplace-Beltrami operator in the literature). Define the operator $\Delta\colon \Lambda^\ast\to \Lambda^\ast$ by
\begin{equation}\label{eq: csoperator}
\Delta=\frac{1}{6}\sum_{a+b+c=0}\colon p_ap_bp_c\colon=\frac{1}{2}\left(\sum_{a,b>0}p_a p_b p_{-a-b}+p_{a+b}p_{-a} p_{-b}\right)\,.\end{equation}
It is straightforward to show that
\[[p_n,\Delta]=n\left(\sum_{a>0}p_{a+n}p_{-a}+\frac{1}{2}\sum_{\substack{a,b>0\\a+b=n}}p_ap_b\right)\colon\,.\]
The right hand side is $n$ times $\bR_n+\frac{1}{2}\bT_n$, which are the Virasoro operators associated to the vertex algebra $\textup{VA}(\BZ, 1)$. Since $p_n$ descends to cohomology for $n\geq 1$ (as multiplication by $n!\ch_n(\BF)$), the operator $\Delta$ descending to cohomology implies that $\bR_n$ also does. Indeed, this is true:

\begin{proposition}
    The operator $\Delta$ on $\Lambda^\ast$ descends to an operator on $H^\ast(\Gr(k,N))$.
\end{proposition}
\begin{proof}
It is well-known that Schur polynomials are a basis of eigenvectors of $\Delta$, see e.g. \cite[Proposition 2]{frenkelwang} or \cite[Remark 1.11]{Sch-Vass}. For instance,
\[\Delta e_j=-\frac{j(j-1)}{2}e_j\textup{ and }\Delta h_j=\frac{j(j-1)}{2}h_j\,.\]
It follows immediately that $\Delta$ preserves the ideal of relations $I$.\qedhere
\end{proof}

\begin{remark}
    In \cite{Sch-Vass} the authors study a certain algebra of operators, related to the $W_{1+\infty}$ Lie algebra, and show that it acts on $\bigoplus_{n\geq 0} H^\ast(\textup{Hilb}^n(\BQ^2))$ and, more generally, on the cohomology of moduli spaces of instantons. They consider the operators $\mathsf D_{0,l}$ on $\Lambda^\ast$ (which for simplicity we specialize at $\kappa=1$) defined to have Schur polynomials as eigenvectors with specified eigenvalues:
    \[\mathsf D_{0,l}(s_\lambda)=\left(\sum_{\Box\in \lambda}\big(y(\Box)-x(\Box)\big)^{l-1}\right)s_\lambda\,.\]
    In particular, $\mathsf D_{0,2}=\Delta$, see \cite[Remark 1.11]{Sch-Vass}. By the same reasoning as above, the operators $\mathsf D_{0,l}$ descend to the cohomology of the Grassmannian. We suspect that it should be possible to interpret $\mathsf D_{0,l}$ as vertex operators defined by $\textup{VA}(\BZ, 1)$, similarly to the formula \eqref{eq: csoperator} which exhibits $\Delta$ as an operator coming from the state $p_1^3$. For instance, it seems that $\mathsf D_{0,3}$ can be obtained from the state $p_1^4-p_2^2$.
\end{remark}

\subsection{Central charge $c\neq 1$ and Jack functions}
In this last section, we work with the coefficients in $\BC$ instead of $\BQ$. Recall that the Virasoro Lie algebra is defined by 
\[\Vir=\bigoplus_{n\in \BZ}\BC L_n\oplus \BC C\]
with Lie bracket given by
\begin{align*}
[L_n, L_m]&=(n-m)L_{m+n}+\delta_{m+n}\frac{n^3-n}{12}C\\
[L_n, C]&=0
\end{align*}
We recall some basic notions of representation theory. Let $\Vir^{+}$ (respectively $\Vir^{-}$) be the sub Lie algebras spanned by $L_n$ for $n>0$ (respectively $n<0$) and let $\Vir^{0}$ be spanned by $L_0, C$. Given a representation $M$ of $\textup{Vir}$ we say that $w\in M$ is a singular vector of weight $(c, h)$ if
\[L_0(w)=hw\,,\quad C(w)=cw\,,\quad L_n(w)=0\,\textup{ for all }n>0\,. \]
We say that $w\in M$ is a highest weight vector if it is a singular vector and moreover $U(\Vir^-)\cdot w=M$. 

Fix parameter $\alpha, \beta\in \BC$ and define
\[\beta_0=\frac{\beta}{2}-\frac{1}{\beta}\,,\quad c=1-12\beta_0^2\,,\quad h=\frac{1}{2}\alpha^2-\alpha \beta_0\,.\]
Then there is a representation of $\Vir$ on $\Lambda$ such that $1\in \Lambda$ is a highest weight vector of weight $(c,h)$; this is called the Fock space representation of highest weight $(c,h)$ and we denote it by $\Lambda_{\alpha, \beta}$. The positive part of the Virasoro algebra acts as
\[L_n=\frac{\beta^2}{2}\sum_{\substack{s+t=n\\s,t>0}}p_{-s}p_{-t}+\sum_{s>0}p_sp_{-s-n}+(\alpha+\beta_0(n+1))\beta p_{-n}\]
and
\[L_0(f)=(h+\deg(f))f\quad \textup{ for }f \textup{ homogeneous}\,.\]
The representation $\Lambda_{\alpha, \beta}$ is irreducible if and only if there are no singular vectors other than scalars if and only if the Verma module of highest weight $(c,h)$ is irreducible. Note that when $\beta=\sqrt{2}$ and $\alpha=(2k-N)/\sqrt{2}$ the operators $L_n$, $n>0$ specialize to \eqref{eq: operatorsgrva}; thus, the Virasoro constraints for the Grassmannian say that $s_{(N-k)^k}$ is a singular vector of $\Lambda_{(2k-N)/\sqrt{2}, \sqrt{2}}$. In this case the central charge is $c=1$. 

The singular vectors of $\Lambda_{\alpha, \beta}$ are completely classified by \cite{mimachiyamada, wakimotoyamada}. It turns out that the singular vectors are given in terms of Jack functions of rectangular partitions. Jack functions $J_\lambda^t$ depend on a partition $\lambda$ and a parameter $t\in \BC$ and they are deformations of Schur polynomials: when $t=1$ the Jack function $J_\lambda^{t=1}$ is proportional to $s_\lambda$; we refer to \cite[Section VI.10]{macdonald} for an introduction to Jack functions. 

\begin{theorem}\label{thm: singularvectorsjack}
The representation $\Lambda_{\alpha, \beta}$ has a singular vector $f\in \Lambda_{\alpha, \beta}$ of degree $d>0$ if and only if there are two integers $r, s>0$ such that \[d=rs\,,\quad \alpha=(1+r)\frac{\beta}{2}-(1+s)\frac{1}{\beta}\,.\]
When that is the case, the only singular vector of degree $d$ (up to a scalar) is given by
\[\sigma J^{\beta^2/2}_{(r)^s}\,\]
where $\sigma$ is the involution on $\Lambda$ sending $p_j\mapsto (-1)^{j-1}p_j$.
\end{theorem}  
It would be very interesting to connect the singular vector $\sigma J^{\beta^2/2}_{(r)^s}$ to some moduli space of quiver representations/sheaves in the same way that $s_{(N-k)^k}$ is associated to the Grassmannian, and thus interpreting as Theorem \ref{thm: singularvectorsjack} as Virasoro constraints for such moduli space. Note that the Calogero-Sutherland operator \eqref{eq: csoperator} also admits a deformation for which the Jack polynomials are eigenvectors, see \cite[Remark 1.1]{Sch-Vass}.

\bibliographystyle{mybstfile.bst}
\bibliography{refs.bib} 

\begin{thebibliography}{MOOP}

\bibitem[AB]{ab}
D. Arcara and A. Bertram, \emph{Bridgeland-stable moduli spaces for
  {$K$}-trivial surfaces}, J. Eur. Math. Soc. (JEMS) \textbf{15} (2013), no.~1,
  1--38, With an appendix by Max Lieblich.

\bibitem[ABCH]{abch}
D. Arcara, A. Bertram, I. Coskun, and J. Huizenga, \emph{The minimal model
  program for the {H}ilbert scheme of points on {$\Bbb{P}^2$} and {B}ridgeland
  stability}, Adv. Math. \textbf{235} (2013), 580--626.

\bibitem[AHPS]{AHPS}
E. Ahlqvist, J. Hekking, M. Pernice, and M. Savvas, \emph{Good Moduli Spaces in
  Derived Algebraic Geometry}, 2023, arXiv preprint,
  \url{https://arxiv.org/abs/2309.16574}.

\bibitem[AM]{am}
D. Arcara and E. Miles, \emph{Projectivity of {B}ridgeland moduli spaces on del
  {P}ezzo surfaces of {P}icard rank 2}, Int. Math. Res. Not. IMRN (2017),
  no.~11, 3426--3462.

\bibitem[Bar]{bardzell}
M.~J. Bardzell, \emph{The alternating syzygy behavior of monomial algebras}, J.
  Algebra \textbf{188} (1997), no.~1, 69--89.

\bibitem[Bei]{beil}
A.~A. Beilinson, \emph{Coherent sheaves on {${\bf P}\sp{n}$} and problems in
  linear algebra}, Funktsional. Anal. i Prilozhen. \textbf{12} (1978), no.~3,
  68--69.

\bibitem[Bla]{B16}
A. Blanc, \emph{Topological {K}-theory of complex noncommutative spaces},
  Compos. Math. \textbf{152} (2016), no.~3, 489--555.

\bibitem[BLM]{blm}
A. Bojko, W. Lim, and M. Moreira, \emph{Virasoro constraints for moduli of
  sheaves and vertex algebras}, Invent. Math. \textbf{236} (2024), no.~1,
  387--476.

\bibitem[Bod]{Bod}
A. Bodzenta, \emph{D{G} categories and exceptional collections}, Proc. Amer.
  Math. Soc. \textbf{143} (2015), no.~5, 1909--1923.

\bibitem[Boj]{bojko}
A. Bojko, \emph{Universal Virasoro Constraints for Quivers with Relations},
  2024, arXiv preprint, \url{https://arxiv.org/abs/2310.18311}.

\bibitem[Bon]{bondal}
A.~I. Bondal, \emph{Representations of associative algebras and coherent
  sheaves}, Izv. Akad. Nauk SSSR Ser. Mat. \textbf{53} (1989), no.~1, 25--44.

\bibitem[Bor]{Borcherds}
R.~E. Borcherds, \emph{Vertex algebras, Kac-Moody algebras, and the Monster},
  Proceedings of the National Academy of Sciences \textbf{83} (1986), no.~10,
  3068--3071.

\bibitem[Bri1]{Bri}
T. Bridgeland, \emph{Stability conditions on triangulated categories}, Ann. of
  Math. (2) \textbf{166} (2007), no.~2, 317--345.

\bibitem[Bri2]{BriK3}
T. Bridgeland, \emph{Stability conditions on {$K3$} surfaces}, Duke Math. J.
  \textbf{141} (2008), no.~2, 241--291.

\bibitem[BS]{Bri-St}
T. Bridgeland and D. Stern, \emph{Helices on del {P}ezzo surfaces and tilting
  {C}alabi-{Y}au algebras}, Adv. Math. \textbf{224} (2010), no.~4, 1672--1716.

\bibitem[CJ]{CJ}
W. Cai and N. Jing, \emph{On vertex operator realizations of {J}ack functions},
  J. Algebraic Combin. \textbf{32} (2010), no.~4, 579--595.

\bibitem[DM]{Davison_Meinhardt}
B. Davison and S. Meinhardt, \emph{Cohomological {D}onaldson-{T}homas theory of
  a quiver with potential and quantum enveloping algebras}, Invent. Math.
  \textbf{221} (2020), no.~3, 777--871.

\bibitem[EH]{EH}
D. Eisenbud and J. Harris, \emph{3264 and All That: A Second Course in
  Algebraic Geometry}, Cambridge University Press, 2016.

\bibitem[EHX]{ehx}
T. Eguchi, K. Hori, and C.-S. Xiong, \emph{Quantum cohomology and {V}irasoro
  algebra}, Phys. Lett. B \textbf{402} (1997), no.~1-2, 71--80.

\bibitem[Fra]{franzen}
H. Franzen, \emph{Chow rings of fine quiver moduli are tautologically
  presented}, Math. Z. \textbf{279} (2015), no.~3-4, 1197--1223.

\bibitem[FW]{frenkelwang}
I.~B. Frenkel and W. Wang, \emph{Virasoro algebra and wreath product
  convolution}, J. Algebra \textbf{242} (2001), no.~2, 656--671.

\bibitem[GJT]{grossjoycetanaka}
J. Gross, D. Joyce, and Y. Tanaka, \emph{Universal structures in {$\mathbb
  C$}-{L}inear enumerative invariant theories}, SIGMA Symmetry Integrability
  Geom. Methods Appl. \textbf{18} (2022), Paper No. 068.

\bibitem[Gro]{Gr19}
J. Gross, \emph{The homology of moduli stacks of complexes}, 2019, arXiv
  preprint, \url{https://arxiv.org/abs/1907.03269}.

\bibitem[HZ]{HZ}
D. Happel and D. Zacharia, \emph{Algebras of finite global dimension},
  Algebras, quivers and representations, Abel Symp., vol.~8, Springer,
  Heidelberg, 2013, pp.~95--113.

\bibitem[Jin]{jing}
N.~H. Jing, \emph{Vertex operators and {H}all-{L}ittlewood symmetric
  functions}, Adv. Math. \textbf{87} (1991), no.~2, 226--248.

\bibitem[Joy1]{Jo17}
D. Joyce, \emph{Ringel–Hall style Lie algebra structures on the homology of
  moduli spaces}, 2019, Preprint,
  \url{https://people.maths.ox.ac.uk/joyce/hall.pdf}.

\bibitem[Joy2]{Jo21}
D. Joyce, \emph{Enumerative invariants and wall-crossing formulae in abelian
  categories}, 2021, arXiv preprint, \url{https://arxiv.org/abs/2111.04694}.

\bibitem[JS]{JS12}
D. Joyce and Y. Song, \emph{A theory of generalized {D}onaldson-{T}homas
  invariants}, Mem. Amer. Math. Soc. \textbf{217} (2012), no.~1020, iv+199.

\bibitem[Kac]{Ka98}
V. Kac, \emph{Vertex algebras for beginners}, second ed., University Lecture
  Series, vol.~10, American Mathematical Society, Providence, RI, 1998.

\bibitem[Kad]{kadeishvili}
T.~V. Kadei\v{s}vili, \emph{On the theory of homology of fiber spaces}, Uspekhi
  Mat. Nauk \textbf{35} (1980), no.~3(213), 183--188.

\bibitem[Kaw]{kawamata}
Y. Kawamata, \emph{Derived categories of toric varieties}, Michigan Math. J.
  \textbf{54} (2006), no.~3, 517--535.

\bibitem[Kel1]{Keller}
B. Keller, \emph{On differential graded categories}, International {C}ongress
  of {M}athematicians. {V}ol. {II}, Eur. Math. Soc., Z\"{u}rich, 2006,
  pp.~151--190.

\bibitem[Kel2]{kellerdefcy}
B. Keller, \emph{Deformed {C}alabi-{Y}au completions}, J. Reine Angew. Math.
  \textbf{654} (2011), 125--180, With an appendix by Michel Van den Bergh.

\bibitem[Kin]{King}
A.~D. King, \emph{Moduli of representations of finite-dimensional algebras},
  Quart. J. Math. Oxford Ser. (2) \textbf{45} (1994), no.~180, 515--530.

\bibitem[KLMP]{KLMP}
Y. Kononov, W. Lim, M. Moreira, and W. Pi, \emph{Cohomology rings of the moduli
  of one-dimensional sheaves on the projective plane}, 2024, arXiv preprint,
  \url{https://arxiv.org/abs/2403.06277}.

\bibitem[Kon]{kontsevich}
M. Kontsevich, \emph{Intersection theory on the moduli space of curves and the
  matrix {A}iry function}, Comm. Math. Phys. \textbf{147} (1992), no.~1, 1--23.

\bibitem[KW]{King-Walter}
A.~D. King and C.~H. Walter, \emph{On {C}how rings of fine moduli spaces of
  modules}, J. Reine Angew. Math. \textbf{461} (1995), 179--187.

\bibitem[KY]{kalckyang}
M. Kalck and D. Yang, \emph{Relative singularity categories II: DG models},
  2018, arXiv preprint, \url{https://arxiv.org/abs/1803.08192}.

\bibitem[Lat]{latyntsev}
A. Latyntsev, \emph{Cohomological Hall algebras and vertex algebras}, 2021,
  arXiv preprint, \url{https://arxiv.org/abs/2110.14356}.

\bibitem[Lie]{lieblich}
M. Lieblich, \emph{Moduli of complexes on a proper morphism}, J. Algebraic
  Geom. \textbf{15} (2006), no.~1, 175--206.

\bibitem[LPWZ]{LPWZ}
D.~M. Lu, J.~H. Palmieri, Q.~S. Wu, and J.~J. Zhang, \emph{Koszul equivalences
  in {$A_\infty$}-algebras}, New York J. Math. \textbf{14} (2008), 325--378.

\bibitem[LY]{liuyang}
X. Liu and C. Yang, \emph{Action of {V}irasoro operators on {H}all-{L}ittlewood
  polynomials}, Lett. Math. Phys. \textbf{112} (2022), no.~6, Paper No. 113,
  22.

\bibitem[Mac]{macdonald}
I.~G. Macdonald, \emph{Symmetric functions and {H}all polynomials}, second ed.,
  Oxford Mathematical Monographs, The Clarendon Press, Oxford University Press,
  New York, 1995, With contributions by A. Zelevinsky, Oxford Science
  Publications.

\bibitem[MOOP]{moop}
M. Moreira, A. Oblomkov, A. Okounkov, and R. Pandharipande, \emph{Virasoro
  constraints for stable pairs on toric threefolds}, Forum Math. Pi \textbf{10}
  (2022), Paper No. e20.

\bibitem[Mor]{moreira}
M. Moreira, \emph{Virasoro conjecture for the stable pairs descendent theory of
  simply connected 3-folds (with applications to the {H}ilbert scheme of points
  of a surface)}, J. Lond. Math. Soc. (2) \textbf{106} (2022), no.~1, 154--191.

\bibitem[Moz]{mozgovoy}
S. Mozgovoy, \emph{Wall-crossing structures on surfaces}, 2022, arXiv preprint,
  \url{https://arxiv.org/abs/2201.08797}.

\bibitem[MS]{macri_schmidt}
E. Macr\`{i} and B. Schmidt, \emph{Lectures on {B}ridgeland stability}, Moduli
  of curves, Lect. Notes Unione Mat. Ital., vol.~21, Springer, Cham, 2017,
  pp.~139--211.

\bibitem[MT]{marcollitabuada}
M. Marcolli and G. Tabuada, \emph{From exceptional collections to motivic
  decompositions via noncommutative motives}, J. Reine Angew. Math.
  \textbf{701} (2015), 153--167.

\bibitem[MW]{matsukiwentworth}
K. Matsuki and R. Wentworth, \emph{Mumford-{T}haddeus principle on the moduli
  space of vector bundles on an algebraic surface}, Internat. J. Math.
  \textbf{8} (1997), no.~1, 97--148.

\bibitem[MY]{mimachiyamada}
K. Mimachi and Y. Yamada, \emph{Singular vectors of the {V}irasoro algebra in
  terms of {J}ack symmetric polynomials}, Comm. Math. Phys. \textbf{174}
  (1995), no.~2, 447--455.

\bibitem[Opp]{oppermann}
S. Oppermann, \emph{Quivers for silting mutation}, Adv. Math. \textbf{307}
  (2017), 684--714.

\bibitem[Pan]{pandharipandecalculus}
R. Pandharipande, \emph{A calculus for the moduli space of curves}, Algebraic
  geometry: {S}alt {L}ake {C}ity 2015, Proc. Sympos. Pure Math., vol. 97.1,
  Amer. Math. Soc., Providence, RI, 2018, pp.~459--487.

\bibitem[Rei1]{Reineke_framed}
M. Reineke, \emph{Framed quiver moduli, cohomology, and quantum groups}, J.
  Algebra \textbf{320} (2008), no.~1, 94--115.

\bibitem[Rei2]{Reineke_quiver}
M. Reineke, \emph{Moduli of representations of quivers}, Trends in
  representation theory of algebras and related topics, EMS Ser. Congr. Rep.,
  Eur. Math. Soc., Z\"{u}rich, 2008, pp.~589--637.

\bibitem[Rek]{rekuski}
N. Rekuski, \emph{Stability of Kernel Sheaves on Del Pezzo Surfaces}, 2023,
  arXiv preprint, \url{https://arxiv.org/abs/2305.16306}.

\bibitem[Shk]{Shk}
D. Shklyarov, \emph{On Serre duality for compact homologically smooth DG
  algebras}, 2007, arXiv preprint, \url{https://arxiv.org/abs/math/0702590}.

\bibitem[Sie]{Sie}
B. Siebert, \emph{Virtual fundamental classes, global normal cones and
  {F}ulton's canonical classes}, Frobenius manifolds, Aspects Math., vol. E36,
  Friedr. Vieweg, Wiesbaden, 2004, pp.~341--358.

\bibitem[{Sta}]{stacks-project}
T. {Stacks Project Authors}, \emph{\textit{Stacks Project}},
  \url{https://stacks.math.columbia.edu}, 2018.

\bibitem[Su]{su}
H. Su, \emph{Path {$A_\infty$} algebras of positively graded quivers}, Front.
  Math. China \textbf{13} (2018), no.~1, 173--185.

\bibitem[SV]{Sch-Vass}
O. Schiffmann and E. Vasserot, \emph{Cherednik algebras, {W}-algebras and the
  equivariant cohomology of the moduli space of instantons on {$\bold{A}^2$}},
  Publ. Math. Inst. Hautes \'{E}tudes Sci. \textbf{118} (2013), 213--342.

\bibitem[Tod]{toda}
Y. Toda, \emph{Moduli stacks and invariants of semistable objects on {$K3$}
  surfaces}, Adv. Math. \textbf{217} (2008), no.~6, 2736--2781.

\bibitem[To{\"{e}}]{Toen}
B. To{\"{e}}n, \emph{The homotopy theory of {$dg$}-categories and derived
  {M}orita theory}, Invent. Math. \textbf{167} (2007), no.~3, 615--667.

\bibitem[TV]{TV07}
B. To\"{e}n and M. Vaqui\'{e}, \emph{Moduli of objects in dg-categories}, Ann.
  Sci. \'{E}cole Norm. Sup. (4) \textbf{40} (2007), no.~3, 387--444.

\bibitem[vB]{bree}
D. van Bree, \emph{Virasoro constraints for moduli spaces of sheaves on
  surfaces}, Forum Math. Sigma \textbf{11} (2023), Paper No. e4, 35.

\bibitem[Wit]{witten}
E. Witten, \emph{Two-dimensional gravity and intersection theory on moduli
  space}, Surveys in differential geometry ({C}ambridge, {MA}, 1990), Lehigh
  Univ., Bethlehem, PA, 1991, pp.~243--310.

\bibitem[WY]{wakimotoyamada}
M. Wakimoto and H. Yamada, \emph{The {F}ock representations of the {V}irasoro
  algebra and the {H}irota equations of the modified {KP} hierarchies},
  Hiroshima Math. J. \textbf{16} (1986), no.~2, 427--441.

\end{thebibliography}
\end{document}